\documentclass[oneside,english]{amsart}
\usepackage[T1]{fontenc}
\usepackage[utf8]{inputenc}
\usepackage{geometry}
\usepackage{babel}
\usepackage{mathtools}
\usepackage{amstext}
\usepackage{amsthm}
\usepackage{amssymb}
\usepackage{graphicx}
\usepackage{wasysym}
\usepackage{esint}
\usepackage[unicode=true,pdfusetitle,
 bookmarks=true,bookmarksnumbered=false,bookmarksopen=false,
 breaklinks=false,pdfborder={0 0 1},backref=false,colorlinks=false]
 {hyperref}

\makeatletter

\newcommand{\noun}[1]{\textsc{#1}}

\numberwithin{equation}{section}
\theoremstyle{plain}
\newtheorem{thm}{\protect\theoremname}[section]
\theoremstyle{plain}
\newtheorem{cor}[thm]{\protect\corollaryname}
\theoremstyle{definition}
\newtheorem{defn}[thm]{\protect\definitionname}
\theoremstyle{plain}
\newtheorem{prop}[thm]{\protect\propositionname}
\theoremstyle{plain}
\newtheorem{lem}[thm]{\protect\lemmaname}
\theoremstyle{remark}
\newtheorem{rem}[thm]{\protect\remarkname}

\DeclareMathOperator{\re}{Re}
\DeclareMathOperator{\imm}{Im}
\DeclareMathOperator{\da}{da}
\DeclareMathOperator{\tr}{tr}
\DeclareMathOperator{\dist}{dist}
\DeclareMathOperator{\diam}{diam}
\DeclareMathOperator{\osc}{osc}

\DeclareMathOperator{\subh}{sub}
\DeclareMathOperator{\superh}{sup}
\DeclareMathOperator{\harm}{har}
\DeclareMathOperator{\const}{C}
\DeclareMathOperator{\jacobisn}{sn}
\DeclareMathOperator{\jacobicn}{cn}
\DeclareMathOperator{\jacobisc}{sc}
\DeclareMathOperator{\smallc}{c}
\DeclareMathOperator{\wind}{w}
\renewcommand{\chi}{\varkappa}
\DeclareMathOperator{\jacobicd}{cd}
\DeclareMathOperator{\jacobind}{nd}

\makeatother

\providecommand{\corollaryname}{Corollary}
\providecommand{\definitionname}{Definition}
\providecommand{\lemmaname}{Lemma}
\providecommand{\propositionname}{Proposition}
\providecommand{\remarkname}{Remark}
\providecommand{\theoremname}{Theorem}

\begin{document}
\title[Convergence of Fermionic Observables in Massive FK-Ising Model]{ Convergence of Fermionic Observables in the Massive Planar FK-Ising
Model}
\author{S. C. Park}
\address{Sung Chul Park, School of Mathematics, Korea Institute for Advanced
Study, 85 Hoegi-ro Dongdaemun-gu, Seoul 02455 Republic of Korea}
\begin{abstract}
We prove convergence of the 2- and 4-point fermionic observables of
the FK-Ising model on simply connected domains discretised by a planar
isoradial lattice in massive (near-critical) scaling limit. The former
is alternatively known as a (fermionic) martingale observable (MO)
for the massive interface, and in particular encapsulates boundary
visit probabilties of the interface. The latter encodes connection
probabilities in the 4-point alternating (generalised Dobrushin) boundary
condition, whose exact convergence is then further analysed to yield
crossing estimates for general boundary conditions. Notably, we obtain
a massive version of the so-called Russo-Seymour-Welsh (RSW) type
estimates on isoradial lattice.

These observables
satisfy a massive version of \emph{s-holomorphicity} \cite{smirnov-ii},
and we develop robust techniques to exploit this condition which do
not require any regularity assumption of the domain or a particular
direction of perturbation. Since many other near-critical observables
satisfy the same relation (cf. \cite{bedc, cim21, par19}), these strategies
are of direct use in the analysis of massive models in broader setting.

\end{abstract}

\maketitle

\section{Introduction}

\emph{Fortuin-Kasteleyn (FK) percolation}, also known as the \emph{Random
Cluster model} \cite{rc}, is one of the most well-studied models
of equilibrium statistical mechanics. This is in part due to its coupling
with the well-known Potts model, called \emph{Edwards-Sokal coupling}
(see e.g. \cite{grimmett}). Concretely, the FK model (indexed by
$q\in[1,\infty)$) is a probability measure on subsets of weighted
edges (bonds) on the underlying graph, with $q=1$ case giving rise
to the familiar Bernoulli bond percolation. The FK model in two dimensions
shows natural \emph{duality}: it has been recently proved \cite{rc-discont,rc-cont}
that the model on square lattice exhibits continuous (for $q\in\left[1,4\right]$)
and discontinuous (for $q>4$) phase transition at the self-dual point
with constant weights $p=p_{\text{sd}}^{q}$.

In this paper, we focus our attention to the $q=2$ case (and re-purpose
the letter $q$ henceforth), also known as the \emph{FK-Ising model},
defined on an \emph{isoradial} graph, where each face is circumscribed
by a circle of a given (constant in space) radius $\delta>0$. The
coupled Potts model in that case is the famed (spin-)Ising model \cite{lenz, ising},
which has been subject to extensive mathematical analysis dating back
to Onsager's celebrated exact solution \cite{onsager}. We are mainly
interested in the behaviour of the model in the \emph{scaling limit},
where the underlying graph becomes infinite in sequence by discretising
a given simply connected domain by isoradial lattices of mesh size
$\delta\downarrow0$. This discrete setup was used in \cite{chsm2012}
to study the \emph{critical} scaling limit, i.e. the model is kept
at its critical (also self-dual) point. They showed that the emerging
continuous regime shows certain \emph{conformal invariance} as described
by Conformal Field Theory (\cite{BPZ}, see also e.g. \cite{di-francesco-mathieu-senechal}):
the 2-point and 4-point \emph{fermionic observables}, which are deterministic
functions defined on the discrete domain, converge to universal (independent
of the lattice setup) holomorphic functions on the continuous domain.
We show analogous convergence results for the \emph{massive} scaling
limit.

The massive scaling limit, roughly speaking, corresponds to studying
the model with weights at some $O(\delta)$ distance from the critical
ones. On the square lattice, the weights commonly scale $p=p_{\text{sd}}+m\delta$
for a given constant parameter $m>0$ (with the dual model having
weights \emph{below} $p_{\text{sd}}$). More generally, on isoradial
graph, there is a one-parameter family of weights (naturally parametrised
by the \emph{nome} $q\in\left[0,1\right]$), coupled to the \emph{Z-Invariant
Ising model} \cite{bax78}, which allows for explicit discrete analysis
(e.g. \cite{dt,BdTR19}). With this choice, we observe the emergence
of a \emph{massive regime} in the scaling limit. As opposed to the
scale invariant critical regime $(q=0$), the massive regime $(q=\frac{m\delta}{2}$)
is expected to have a finite length scale $\xi\propto\frac{1}{m}$.
For example, for the probability that two bulk points are connected
on the square lattice in the FK-Ising model, such exponential dropoff
may be derived from rigorous results on the massive spin-Ising model,
which was pioneered by the miraculous discovery of the third Painlev\'e
transcendent for the full plane correlations by Wu, McCoy, Tracy,
and Barouch \cite{wmtb} (see also \cite{sato-miwa-jimbo,patr,par19}).

Here we choose to follow the viewpoint taken by \noun{\cite{dgp}},
which in turn is inspired by the study of massive (Bernoulli) percolation
in two dimensions (e.g. \cite{kesten_mperc,nol08}). We consider the
discrete scale (which we will refer to as \emph{characteristic length}
in this paper following \cite{duma}) $L_{\rho,\epsilon,q=\frac{1}{2}m\delta}$
at which the probability of crossing a rectangular box of aspect ratio
$\rho>0$ drops to small but nonzero $\epsilon>0$, and show that
$\delta L_{\rho,\epsilon,q=\frac{1}{2}m\delta}$ remains a finite
nonzero quantity. On the square lattice, the lower bound $L_{\rho,\epsilon,q=\frac{1}{2}m\delta}\apprge\delta^{-1}$
was shown in \noun{\cite{dgp}} by first proving the so-called \emph{Russo-Seymour-Welsh
(RSW) type estimate} for the model, which is in itself of fundamental
interest (see \cite{duminil-copin-hongler-nolin,rsw-strong,dlm} for
results at criticality); the upper bound was shown in the recently
announced \cite{duma}. Here we derive analogous results on isoradial
lattice directly from analysis of the observable, in particular the
upper bound from an \emph{exact crossing estimate} on any conformal
quadrilateral with \emph{4-point (generalised) Dobrushin boundary
conditions}.

Many recent results dealing with conformal invariance of the critical
Ising model have been shown by first establishing convergence of the
fermionic observables to explicit holomorphic functions (e.g. \cite{smirnov-ii,chsm2012,hosm2013,hongler-kytola,chelkak-izyurov, chelkak-hongler,ghp, chi21}).
These results crucially exploit the discrete integrability condition
known as \emph{spin}, \emph{strong}, or simply \emph{s}-holomorphicity,
first formulated in \cite{smirnov-ii} for the critical case. Off
the critical point, the observables satisfy \cite{bedc} a discretised
notion of \emph{perturbed holomorphicity} (Bers-Vekua equation \cite{bers,vek,baratchart}),
part of which (in the form of \emph{massive harmonicity}) \noun{\cite{dgp}}
has already exploited on standard domains together with symmetries
of the square lattice. On a general isoradial lattice, without analogous
discrete symmetry, we are led to develop general techniques for the
analysis of such \emph{massive s-holomorphic} functions, along with
\cite{par19,cim21}. Recently Chelkak has introduced \emph{s-embeddings}, proving the convergence of, e.g. 2-point observables (exactly corresponding to ours), to a critical limit in a considerably more general setup \cite{Che18, s-emb}. These approaches build on the appearance of a fermionic
structure and difference identities in the discrete model, which had
been noted and used variously in, e.g., \cite{kaufman,kadanoff-ceva,perk,patr,mercat}.
See also \cite{palmer} for a comprehensive historical overview.

Analysis on simply connected domains with possibly rough boundary
becomes especially relevant from the viewpoint of the massive scaling
limit of the \emph{interface}: with 2-point Dobrushin boundary condition
(see Introduction), the law of the unique interface separating the
wired cluster from the dual-wired cluster should tend to a massive
perturbation of the critical limit, the \emph{Schramm-Loewner Evolution}
$SLE(16/3)$. Such convergence result is usually proved by showing
the convergence of a martingale observable (which our 2-point observable
serves as one) on general domains along with RSW-type estimates (see,
e.g., \cite{smirnov-i}): indeed, critical analogues of results discussed
in this paper almost immediately implies convergence of the critical
interface to $SLE(16/3)$ \cite{chelkak-duminil-copin-hongler-kemppainen-smirnov}.
In contrast, in the massive case, more analysis is presently needed
for unique identification of the scaling limit, mainly due to complications
in the variation analysis with respect to the domain slit by the interface.
We note that similar difficulties arise in the study of massive percolation
interface, whose conjectured Loewner driving function \cite{GPS18} is rather rough
and tricky to work with, despite many interesting results
on \emph{any given} scaling limit of the interface \cite{nowe}.

\subsection{Discrete Setting}

\subsubsection*{Isoradial Graph}

The setting of our discrete model is the isoradial\emph{ }graph, on
which the connection between discrete complex analysis and dimer and
critical Ising models have been studied in, e.g., see \cite{mercat,chsm2012}.
The relationship between the massive Ising model and the so-called
massive Laplace and Dirac operators on isoradial graphs has been also
made rigorous \cite{bdtr,dt}. For the convenience of the interested
reader, we choose to align broadly with \cite{chsm2012,dt} on conventions
and notations.

An \emph{isoradial lattice} is a planar lattice (i.e. graph tiling
of $\mathbb{C}$) where each face is circumscribed by a circle of
fixed radius (the \emph{mesh size}) $\delta>0$. We consider finite
subgraphs $G$, whose standard components we denote as follows (Figure
\ref{fig:intro_grid}L):
\begin{itemize}
\item the set of (\emph{primal}, or \emph{black}), \emph{vertices} $\Gamma(G)$;
\item the set of \emph{dual,} or\emph{ white, vertices} $\Gamma^{*}(G)$
corresponding to \emph{faces} of $G$, identified with the centres
of their circumscribing circles;
\item the \emph{dual graph} $G^{*}$ is also isoradial with vertices $\Gamma(G^{*})\cong\Gamma^{*}(G)$
and faces $\Gamma^{*}(G^{*})\cong\Gamma(G)$;
\item the \emph{double graph} with vertices in $\Lambda(G):=\Gamma(G)\cup\Gamma^{*}(G)$
and two vertices are adjacent if and only if they are incident on
$G$ has rhombic faces;
\item the set of these \emph{rhombi} $\lozenge(G)=\Lambda^{*}(G)$ is naturally
isomorphic to the sets of primal \emph{edges}, which is also in bijection
with the set of \emph{dual edges} $\lozenge\left(G^{*}\right)$ (perpendicular
to primal edges, connecting adjacent points in $\Gamma^{*}$);
\item any rhombus $z$ has the \emph{half-angle} $\overline{\theta}_{z}$
formed by its primal diagonal and any of its four rhombus edges;
\item the set of \emph{rhombus edges} $\Upsilon(G)$ corresponding to \emph{corners}
of the faces in $G$;
\item a corner $\xi=\left\langle uw\right\rangle $ for $u\in\Gamma,w\in\Gamma^{*}$
is given the direction $\nu(\xi):=\frac{w-u}{\left\vert w-u\right\vert }$ pointing
\emph{towards} the dual vertex.
\end{itemize}
We denote the corresponding full-plane sets as $\Gamma$, $\lozenge$,
etc. We impose the standard assumption that there is a \emph{uniform
angle bound} $\eta>0$ such that all half-angles $\overline{\theta}_{z}\in\left[\eta,\frac{\pi}{2}-\eta\right]$.

\subsubsection*{Isoradial Discretisation}

We paste together a finite (but asymptotically unbounded) number of
rhombi in $\lozenge$ to create a simply connected polygonal domain
which we identify with the underlying isoradial graph.

Specifically, we consider discretisations $\Omega^{\delta}$ of a
bounded planar simply connected domain $\Omega\subset\mathbb{C}$.
We refer to the boundary $\partial\Omega$ and the closure $\overline{\Omega}$
in the sense of \emph{prime ends}, homeomorphic to $\partial\mathbb{D}$
under the completion of a conformal map from the unit disc $\mathbb{D}$
to $\Omega$ (see e.g. \cite[Section 2.4]{pom92}). To speak of Dobrushin
boundary conditions, we endow $\Omega$ with some \emph{marked points}
in $\partial\Omega$, of which we treat 2- and 4-point cases explicitly.

For the 2-point case, we consider marked points $a,b$ partitioning
$\partial\Omega$ into two (open, but see Remark \ref{rem:capacity})
segments (going counterclockwise) $\left(ab\right),\left(ba\right)$.
Consider for $\delta\downarrow0$ simply connected polygonal domains
$\Omega^{\delta}$ composed of finitely many rhombi, whose boundary
is an arc of alternating primal and dual vertices. We select two \emph{marked
corners} (rhombus edges) $a^{\delta}=\left\langle a_{\text{b}}^{\delta}a_{\text{w}}^{\delta}\right\rangle $
and $b^{\delta}=\left\langle b_{\text{b}}^{\delta}b_{\text{w}}^{\delta}\right\rangle $
on the boundary such that, travelling along the boundary counterclockwise,
$a^{\delta}$ is traversed in the direction of $a_{\text{b}}^{\delta}\to a_{\text{w}}^{\delta}$,
and $b^{\delta}$ is traversed in the direction of $b_{\text{w}}^{\delta}\to b_{\text{b}}^{\delta}$.
We assume that $\Omega^{\delta}$ converges to $\Omega$ \emph{in
the Carath\'eodory sense}, with $a^{\delta},b^{\delta}\in\partial\Omega^{\delta}$
converging to $a,b\in\partial\Omega$ as prime ends.

Then the boundary arc $(a_{\text{w}}^{\delta}b_{\text{w}}^{\delta})$,
designated \emph{free}, is the path of the dual edges running from
$a_{\text{w}}^{\delta}$ to $b_{\text{w}}^{\delta}$ which will be
dual-wired in the model. The arc $(b_{\text{b}}^{\delta}a_{\text{b}}^{\delta})$
is \emph{wired}, and is similarly the path of the primal edges running
from $b_{\text{b}}^{\delta}$ to $a_{\text{b}}^{\delta}$. For conciseness,
we will frequently write $(a_{\text{w}}^{\delta}b_{\text{w}}^{\delta})=(a^{\delta}b^{\delta})$,
etc. The rhombi bisected by these edges form the boundary $\partial\lozenge\left(\Omega^{\delta}\right)$,
and the rest in the interior form the set $\lozenge\left(\Omega^{\delta}\right)$
where the random configurations are sampled.

For the 4-point (\emph{conformal quadrilateral}) case, we simply consider
two more corners $c^{\delta},d^{\delta}\to c,d$ along $\left(b^{\delta}a^{\delta}\right)$
such that they are respectively oriented in the same direction as
$a^{\delta},b^{\delta}$. Accordingly, the arcs $\left(a_{\text{w}}^{\delta}b_{\text{w}}^{\delta}\right),\left(c_{\text{w}}^{\delta}d_{\text{w}}^{\delta}\right)$
are free, and $\left(b_{\text{b}}^{\delta}c_{\text{b}}^{\delta}\right),\left(d_{\text{b}}^{\delta}a_{\text{b}}^{\delta}\right)$
are wired. Without loss of generality, by rotation if necessary, we
will henceforth assume that $b^{\delta}$ points upward: $\nu_{b^{\delta}}=i$
in both 2- and 4-point cases.

For more on the discrete boundary, see Sections \ref{subsec:dbvp}
and \ref{sec:Discrete-Regularity-Theory}.

We finish by noting that we can construct $\Omega^{\delta}$ simply
by taking the largest connected component of the intersection of $\lozenge$
and $\Omega$, filling in any holes, then choosing boundary corners
of $\Omega^{\delta}$ converging to marked points of $\Omega$, if
any. This is how we discretise rectangles and discs (see also \cite[Section 2.1]{chsm2012}).

\begin{figure}
\centering
\includegraphics[width=0.7\paperwidth]{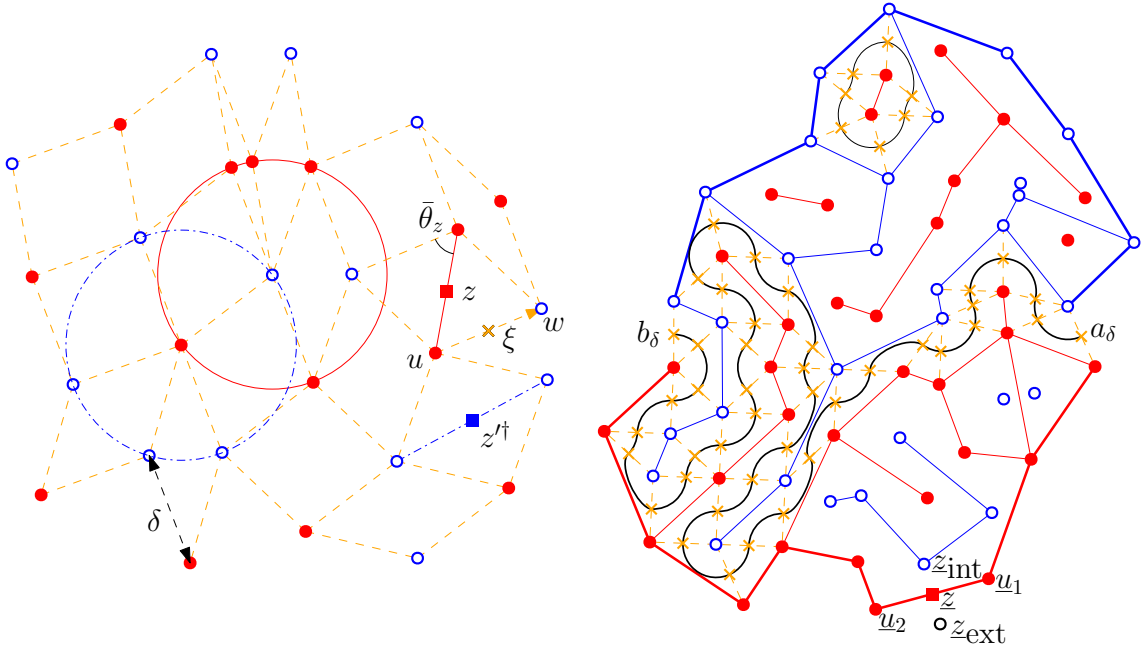}

\caption{(Left) Local view of an isoradial grid. Primal vertices $u$ in solid
red; dual vertices $w$ in hollow blue; rhombus edges $\xi$ in dashed
orange. Also note the primal edge $z$ and dual edge $z'^{\dagger}$,
both diagonals of their respective rhombi (also denoted $z,z'$).
(Right) A sample FK-Ising configuration. Primal edges $E$ in solid red;
dual edges $E^{*}$ in solid blue; the interface $\gamma$ in solid
black; boundary primal/dual wiring in bold. Local notation around
a boundary primal edge $z$ shown in bottom.}
\label{fig:intro_grid}
\end{figure}

\subsubsection*{Z-Invariant Weights and Mass Scaling}

On isoradial graphs, we consider the family of local weights on the
edges (i.e. rhombi) parametrised by the \emph{elliptic modulus} $k\in\left[0,1\right]$:
the \emph{Z-invariant weights}, which coincides with the \emph{critical}
weights in \cite{chsm2012} when $k=0$, which we call also the \emph{massless}
weights. While we study the model in the vicinity of $k=0$, let us
note here that the case $k=1$ corresponds to the degenerate case
where all edges must be sampled.

Locally, the weights are conveniently written in terms of \emph{abstract
angles} $\hat{\theta}_{z}\in\left[\bar{\theta}_{z},\frac{\pi}{2}\right]$
assigned to each edge $z$ satisfying the following relation $k$
and the geometric\emph{ }rhombus half-angle $\overline{\theta}_{z}$:
\[
u\left(\hat{\theta}_{z}\vert k\right)=\frac{u\left(\pi/2\vert k\right)}{\pi/2}\cdot\overline{\theta}_{z}=:\frac{K(k)}{\pi/2}\overline{\theta}_{z},
\]
where $u\left(\varphi\vert k\right):=\int_{0}^{\varphi}\frac{d\theta}{\sqrt{1-k^{2}\sin^{2}\theta}}$,
and the elliptic \emph{quarter-period} $K(k):=F\left(\pi/2\vert k\right)$.
Under this correspondence, we have the convenient relations (which
may be taken as the definitions for the functions on the left hand
side):
\[
\jacobisn\left(\frac{K(k)}{\pi/2}\overline{\theta}_{z}\right)=\sin\hat{\theta}_{z},\jacobicn\left(\frac{K(k)}{\pi/2}\overline{\theta}_{z}\right)=\cos\hat{\theta}_{z},\jacobisc\left(\frac{K(k)}{\pi/2}\overline{\theta}_{z}\right)=\tan\hat{\theta}_{z}.
\]

To take the scaling limit to obtain the \emph{massive} regime, we
need to scale $k^{2}\sim8m\delta$ in the limit $\delta\downarrow0$
for some fixed $m>0$. Equivalently, we take the real \emph{nome}
$q:=\exp\left[-\pi K(\sqrt{1-k^{2}})/K(k)\right]$ and scale $q\sim\frac{1}{2}m\delta$
(since $K$ is increasing, $q$ is also increasing in $k$, and $q\sim k^{2}/16$
for small $k$; see \cite[19.5.5]{dlmf}). Unfortunately, the standard
notation $m$ for the mass parameter is also used for the square of
$k$ in the elliptic function literature; however, we choose to exclusively
use $m$ for the former meaning, \textbf{assuming some relation }$q=q(\delta)$
\textbf{such that (say) $q\leq m\delta$ and }$\delta^{-1}q\xrightarrow{\delta\downarrow0}\frac{m}{2}$\textbf{
to have been fixed} whenever talking about a ($m$-)massive scaling
limit.

The parameters $\hat{\theta}_{z}$, and thus the relative weight of
open edges, increase in $k,q$ (see, e.g. \cite{carlson-todd}). In
fact, as $\delta\downarrow0$, we have 
\begin{equation}
\hat{\theta}_{z}=\overline{\theta}_{z}+m\delta\sin2\bar{\theta}_{z}+O(\delta^{2}),\label{eq:massivetheta}
\end{equation}
as seen from $\tan\hat{\theta}_{z}=\left(1+4q\right)\tan\overline{\theta}_{z}+O(q^{2})$
\cite[20.2(i), 22.2.9]{dlmf}.

\subsubsection*{FK-Ising Model}

Consider a (primal) \emph{configuration}, a subset $E\subset\lozenge\left(\Omega^{\delta}\right)$
of primal edges, and its \emph{dual configuration} $E^{*}$, consisting
of dual edges corresponding to the primal edges in $\lozenge\left(\Omega^{\delta}\right)\setminus E$
(Figure \ref{fig:intro_grid}R). To implement the boundary condition,
we consider the boundary edges on each wired arc as part of $E$ and
each free arc as $E^{*}$. We primarily consider 2-point and 4-point
\emph{(generalised) Dobrushin boundary conditions}: they are respectively
defined on marked domains $\left(\Omega^{\delta},a^{\delta},b^{\delta}\right)$
and $\left(\Omega^{\delta},a^{\delta},b^{\delta},c^{\delta},d^{\delta}\right)$,
with the boundary condition alternating between free and wired, starting
from the free arc $\left(a_{w}^{\delta}b_{w}^{\delta}\right)$. We
will announce explicitly the setup of the model whenever writing $\mathbb{P}$.

An edge $z\in E$ is termed \emph{open} (accordingly, $z\in E^{*}$
\emph{closed} or \emph{dual-open}). Given the connections made by
$E$, a connected component of $\Gamma\left(\Omega^{\delta}\right)$
is called an \emph{open }or \emph{primal cluster} (accordingly, \emph{closed}
or \emph{dual} cluster for $E^{*}$ and $\Gamma^{*}\left(\Omega^{\delta}\right)$).
Then given any corner $\left\langle uw\right\rangle \in\Upsilon\left(\Omega^{\delta}\right)$,
there is a curve (unique up to homotopy away from $E,E^{*}$) separating
the open cluster of $u$ and the closed cluster of $w$. Any such
curve might exit the domain through one of the marked boundary corners,
in which case it is an \emph{interface curve} (see Figure \ref{fig:intro_grid}R
and also Section \ref{subsec:Fermionic-Observables}) between boundary
clusters, or be simple loops within the domain, whose number we denote
as $\#\text{loops}\left(E\right)$. Define the ($q$-\emph{)massive
FK-Ising model} on $\Omega^{\delta}$ as the probability measure $\mathbb{P}$
on subsets of $\lozenge\left(\Omega^{\delta}\right)$ given by:
\[
\mathbb{P}\left(E\right)\propto\sqrt{2}^{\#\text{loops}\left(E\right)}\prod_{z\in E}\sin\frac{\hat{\theta}_{z}}{2}\prod_{z\in\lozenge\setminus E}\sin\left(\frac{\pi}{4}-\frac{\hat{\theta}_{z}}{2}\right).
\]

Clearly, the dual configuration $E^{*}$ has the FK-Ising law sampled
from the dual graph with weights switched and free arcs being \emph{dual-wired}.
In fact, while we only consider on the primal graph the \emph{subcritical}
massive scaling limit (where $\hat{\theta}_{z}>\bar{\theta}_{z}$)
the dual configuration has a \emph{supercritical} law, and thus our
treatment covers both the sub- and supercritical massive regimes simultaneously.

Let us finish by recalling the basic notions and properties of the
model, for which \cite{grimmett} serves as a comprehensive reference.
We are primarily interested in \emph{crossing events}, where given
subsets of the plane are connected by an open cluster in $E$. It
is clear that conditioning on a bounded number of edges only affect
the probability measure by a bounded factor, a consequence of the
\emph{finite energy property} of the model: therefore, the probabilities
for crossing of sets which are bounded lattice spacings apart are
uniformly comparable. This in particular allows for speaking of the
'same' domain endowed with different boundary conditions, which might
require in reality adding or taking way some layers of (dual-)wired
edges, which only affect crossing events up to a uniform factor.

(Open) crossings are the archetypal examples of \emph{increasing}
events: if the event contains $E$, then any superset of $E$ is also
in the event. The classical \emph{Fortuin-Kasteleyn-Ginibre (FKG)
inequality} says that increasing events are positively correlated,
that is, conditioning on an increasing event only augments another
increasing event's probability. Note also that the probability of
an increasing event increases when the weight parameter $\hat{\theta}_{z}$
increases or the boundary condition along some segment switches from
free to wired.

\subsection{\label{subsec:intro_statement}Statement of the Theorems}

Our fundamental result is on convergence of discrete fermionic observables,
to be defined precisely in Section \ref{sec:Massive-S-Holomorphic-Observable}.
These are discrete massive holomorphic functions, which converge to
continuous functions with analogous properties: we call a function
(see Section \ref{subsec:Functions-on-Physical} for notes on regularity)
$f$ defined on a simply connected $\Omega\subset\mathbb{C}$ \emph{massive
holomorphic} if it satisfies the Bers-Vekua equation
\[
\bar{\partial}f+mi\bar{f}=0\text{ in }\Omega,
\]
with constant coefficient $m>0$. Here we use the standard Wirtinger
derivatives $\partial:=\frac{1}{2}\left(\partial_{x}-i\partial_{y}\right)$
and $\bar{\partial}:=\frac{1}{2}\left(\partial_{x}+i\partial_{y}\right)$.

First, we show convergence of the \emph{2-point observable}, also
known as the \emph{fermionic martingale observable} (see \cite{makarov-smirnov}).
We say that a family $\left\{ F^{\delta}\right\} _{\delta>0}$ of
discrete function defined on $\lozenge\left(\Omega_{\delta}\right)$
\emph{converges} to $f$ if $f^{\delta}:=\left.f\right\vert _{\lozenge\left(\Omega_{\delta}\right)}$
is (locally uniformly) close to $F^{\delta}$ as $\delta\downarrow0$.
\begin{thm}
\label{thm:martingale}On a marked simply connected domain $\left(\Omega,a,b\right)$,
the \emph{2-point discrete massive fermionic observable} $F_{\left(\Omega^{\delta},a^{\delta},b^{\delta}\right)}^{\delta}$
(Definition \ref{def:2ptc} and Proposition \ref{prop:s-hol}) converges
as $\delta\downarrow0$ uniformly on compact subsets to the massive
holomorphic limit $f_{\left(\Omega,a,b\right)}$ as in Definition
\ref{def:c2pt}.
\end{thm}

\begin{proof}
One may extract from any subset of $\left\{ F_{\left(\Omega^{\delta},a^{\delta},b^{\delta}\right)}^{\delta}\right\} _{\delta>0}$
a subsequence which converges to a smooth function on $\Omega$ uniformly
in compact subsets by Proposition \ref{prop:2pt_bc} and Remark \ref{rem:subsequence}.
Then it remains to show that the limit is the unique function satisfying
the conditions laid out in Definition \ref{def:c2pt}, which is shown
in Proposition \ref{prop:2pt-uniqueness}. 
\end{proof}
Now we move to convergence of the crossing probability on a conformal
quadrilateral with 4-point Dobrushin boundary condition, which in
turn comes from convergence of the \emph{4-point observable} (Definition
\ref{def:d4pt}). Given the conformal quadrilateral $\left(\Omega^{\delta},a^{\delta},b^{\delta},c^{\delta},d^{\delta}\right)$,
we denote by $\stackrel{\leftrightarrow}{\Omega^{\delta}}$ the event
that the conformal quadrilateral is \emph{horizontally (open) crossed},
which we will fix as the event that the two boundary clusters respectively
containing $\left(b_{\text{b}}^{\delta}c_{\text{b}}^{\delta}\right)$
and $\left(d_{\text{b}}^{\delta}a_{\text{b}}^{\delta}\right)$ are
connected by primal edges in $E$.
\begin{thm}
\label{thm:2}For any conformal quadrilateral $\left(\Omega,a,b,c,d\right)$,
the crossing probability
\[
\mathbb{P}\left[\stackrel{\leftrightarrow}{\Omega^{\delta}}\right]=:\mathrm{P}^{\delta}
\]
of the massive 4-point Dobrushin FK-Ising model on $\left(\Omega^{\delta},a^{\delta},b^{\delta},c^{\delta},d^{\delta}\right)$
converges to a limit $\mathrm{P}^{\delta}\xrightarrow{\delta\downarrow0}\mathrm{p}_{m}\in\left(0,1\right)$.

$\mathrm{p}_{m}$ is uniquely determined by the condition that if
$t_{m}:=\frac{\mathrm{p}_{m}}{1-\mathrm{p}_{m}}$, then $\chi_{m}:=\left[\frac{t_{m}^{2}+\sqrt{2}t_{m}}{t_{m}^{2}+\sqrt{2}t_{m}+1}\right]^{2}\in\left(0,1\right)$
is the unique value for which there exists a massive holomorphic function
$f_{\left(\Omega,a,b,c,d\right)}$ in Definition \ref{def:c4pt}.
\end{thm}

\begin{proof}
The fact that $\mathrm{P}^{\delta}$ is encoded in the discrete 4-point
observable $F_{\left(\Omega^{\delta},a^{\delta},b^{\delta},c^{\delta},d^{\delta}\right)}^{\delta}$
of Definition \ref{def:d4pt} through the value $\chi^{\delta}$ defined
in Proposition \ref{prop:4pt_bc} is a combinatorial calculation identical
to the critical case, see \cite[(6.6)]{chsm2012}. As in the 2-point
case, we may extract subsequential limits from the set of discrete
observable, which we show to be unique in Corollary \ref{cor:4pt_uniqueness}.
\end{proof}

\subsubsection*{Application: RSW-Type Estimates and Upper Bound for the Characteristic
Length}

Using a degenerate case of the 2-point observable, we may show the
following uniform estimate of the crossing probability of a rectangle
of given aspect ratio, also known as a Russo-Seymour-Welsh (RSW) type
estimate. \cite{dgp} shows the following on the square lattice. We
give a proof on the isoradial graph, using the general method established
in \cite{duminil-copin-hongler-nolin}. Note that one may alternatively
use the 4-point connection probability from Theorem \ref{thm:2} to
obtain annulus crossing estimates rather straightforwardly, by using
the argument of \cite[Section 5.6]{s-emb}.

Recall that we discretise a rectangle $R(\rho):=\left(0,1\right)+\left(0,\rho i\right)$
using the intersection with an isoradial lattice. The following is
stated for the subcritical primal model, but it also implies the analogue
for the supercritical regime by duality. Note that we may standardise
rectangles of any size into $R(\rho)$ by rescaling the mass (which
then gets multiplied by the original horizontal side length). Here
horizontal crossing intuitively refers to a crossing event from left
to right sides.
\begin{thm}
\label{thm:rsw}Let $m>0$. There is a constant $\smallc(\rho,\eta,m)>0$
such that
\[
\smallc(\rho,\eta,m)\leq\mathbb{P}\left[\stackrel{\leftrightarrow}{R(\rho)^{\delta}}\right]\leq1-\smallc(\rho,\eta,m),
\]
for the massive FK-Ising model with $q\leq m\delta$ and any boundary
condition on the discrete rectangle $R(\rho)^{\delta}$.
\end{thm}

\begin{proof}
It suffices to show the upper bound for wired boundary condition;
other boundary conditions and the lower bound follow easily from monotonocity
in weights and duality. More specifically, we will show that the dual
model has a vertical crossing with probability bounded away from zero.
It also suffices to prove the estimate for some small fixed $m_{0}>0$,
since we may use the estimate at $m_{0}$ multiple times to obtain
crossing estimates of larger rectangles (using FKG inequality, cf.
the Bernouilli case in e.g. \cite{nol08}) at $m_{0}$, which then
translate to the above result for normalised rectangles at larger
masses.

Consider the (discretised) bottom and top middle boundary segments
$\left[\left(\frac{1}{3}\right)^{\delta},\left(\frac{2}{3}\right)^{\delta}\right]$
and $\left[\left(\frac{1}{3}+\rho i\right)^{\delta},\left(\frac{2}{3}+\rho i\right)^{\delta}\right]$.
Then defining the number of disjoint dual vertical crossings as $N$,
we have the second-moment estimate
\[
1-\mathbb{P}\left[\stackrel{\leftrightarrow}{R(\rho)^{\delta}}\right]\geq\mathbb{P}\left[N>0\right]\geq\frac{\mathbb{E}\left[N\right]^{2}}{\mathbb{E}\left[N^{2}\right]},
\]
so we need to give a lower bound for the numerator and an upper bound
for the denominator. By monotonicity, the latter, specifically
\[
\mathbb{E}\left[N^{2}\right]\leq\const\left(\rho,\eta\right)\delta^{2},
\]
may be obtained at criticality ($m=0$); this is the content of \cite[Proposition 4.3]{duminil-copin-hongler-nolin},
which is technically stated only for the square lattice but their
strategy applies with almost no modification on isoradial lattice
with angle bound $\eta>0$. Namely, \cite[Lemma 3.3]{duminil-copin-hongler-nolin}
connects the probability that the critical FK-interface passes through
a boundary corner to harmonic measure estimates through the use of
the fermionic observable (see Lemma \ref{lem:dbvp} and the proof
for isoradial analogues). These estimates are obtained by comparison
and explicit estimates on standard domains, which are straightforward
to obtain in the isoradial case using, e.g., \cite[Lemma A.3]{chsm2012}.
For the sake of conciseness we do not replicate the full proof. 

The 'massive content' is in the lower bound for the numerator: we
prove that $\mathbb{E}\left[N\right]\geq\smallc(\rho,\eta)\delta$
for some small $m_{0}>0$ in Corollary \ref{cor:rsw-num}. This finishes
the proof.
\end{proof}
\cite[Theorem 1.2]{dgp} and \cite[Theorem 1.3]{duma} respectively
provide lower and upper bounds on the square lattice for the characteristic
length (called correlation length in the former), defined as the size
(in terms of lattice spacings) of the smallest rectangle which is
crossed with at least a given cutoff probability (stated without loss
of generality in terms of the subcritical primal model). Note that
$q\propto p-p_{c}$ in their setup.
\begin{cor}
\label{cor:corlength}For $\rho>0$, consider the $q$-massive FK-Ising
model with any given boundary condition on the discrete rectangle
$R(\rho)^{\delta}$. For fixed $\epsilon\in\left(0,1\right)$, define
the \emph{characteristic length}
\[
\left(L_{\rho,\epsilon,q}\right)^{-1}:=\sup\left\{ \delta>0:\mathbb{P}\left[\stackrel{\leftrightarrow}{R(\rho)^{\delta}}\right]\geq1-\epsilon\right\} .
\]
Then for all small $\epsilon>0$, there are constants $0<\smallc(\rho,\eta,\epsilon)<\const(\rho,\eta,\epsilon)$
such that
\begin{equation}
\smallc(\rho,\eta,\epsilon)\leq qL_{\rho,\epsilon,q}\leq\const(\rho,\eta,\epsilon).\label{eq:corlength}
\end{equation}
\end{cor}

\begin{proof}
The lower bound is essentially the crossing bound of Theorem \ref{thm:rsw}:
if there is no such constant, there is a sequence $q_{j}\delta_{j}^{-1}\to0$
such that $\mathbb{P}\left[\stackrel{\leftrightarrow}{R(\rho)^{\delta}}\right]$
remains at least $1-\epsilon$, so for any $\epsilon$ smaller than
the upper bound in Theorem \ref{thm:rsw}, we obtain a contradiction.
The upper bound is shown in Section \ref{subsec:mtoinfinity}, starting
from the 4-point Dobrushin boundary case (Corollary \ref{cor:dob-correlation})
given before the general proof.
\end{proof}

\subsubsection*{Other Implications}

\cite{dgp} has highlighted a behaviour in the massive FK-Ising model
which qualitatively differ from that of massive Bernouilli percolation
on triangular lattice (e.g. \cite{nol08}): the characteristic length
$L_{\rho,\epsilon,q}$ scales like $q^{-1}$, which suggests that
it is not solely determined by the critical four-arm exponent as in
the Bernouilli percolation case. That is, the characteristic length
cannot be estimated by independently flipping the pivotal points,
where macroscopic four-arms start. In the massive Bernouilli percolation,
the analysis of four-arm exponents has led to uncovering the mutual
\emph{singularity} (i.e. absolutely continuous in neither direction)
of the massive and the critical scaling limits, both in terms of the
interface curve \cite{nowe} and the quad-crossing probability \cite{aumann,GPS18}.
The latter is done by considering asymptotically smaller quads, where
the model breaks up into independent pieces and crossing probabilities
are perturbed by a quantity determined by the four-arm exponent. In
our case, restrictions to these small quads do not become independent,
and crossing probabilities depend on the boundary condition; nonetheless,
hoping to carry out similar analysis in future works, we show in the
case of the 4-point Dobrushin boundary condition that the perturbation
to the quad-crossing probability decays as $m\downarrow0$ (corresponding
to reducing quad size at a fixed mass) like $m$ (Corollary \ref{cor:mtozero}).

With respect to the interface, we happen to prove massive versions
of the two results which implied convergence of the critical FK-Ising
interface in law (\cite{chelkak-duminil-copin-hongler-kemppainen-smirnov},
see also \cite{makarov-smirnov}): convergence of the discrete martingale
observable to a unique limit (Theorem \ref{thm:martingale}) and the
RSW-type crossing estimate (Theorem \ref{thm:rsw}; \cite[Theorem 1.3]{dgp}
on the square lattice). The latter in fact implies that some H\"older
exponent of the massive FK-Ising interface is bounded \cite{k-smirnov}.
The reason that we cannot then easily conclude that there is a unique
limit of the law of the discrete interface is because the analysis
of the continuous martingale observable turns out to be significantly
more convoluted than in the critical case; this is to be expected,
given that the massive scaling limit of the interface may well have
a distribution which is mutually singular with the critical limit
$SLE(16/3)$ (as in the case of the massive percolation interface
and $SLE(6)$, massive uniform spanning tree and $SLE(8)$ \cite{makarov-smirnov},
and conjectured in, e.g., \cite{GPS18} for any $SLE(\kappa)$ with
$\kappa\in(4,8]$). 

\subsection{Structure of the Paper}

This paper is organised as follows. In Section \ref{sec:Massive-S-Holomorphic-Observable},
we relate the probabilistic model to discrete complex analysis by
introducing discrete fermionic observables for the 2- and 4-point
boundary conditions; they satisfy \emph{massive s-holomorphicity},
a consequence of which is shown to be the existence of the \emph{discrete
square integral} as mentioned above. We finish by translating boundary
conditions for the model to the \emph{Riemann-Hilbert }boundary value
problem for the s-holomorphic observable and its square integral.
In Section \ref{sec:Massive-Holomorphic-Functions}, we define \emph{massive
holomorphic} functions, which are continuous counterparts of the massive
s-holomorphic functions and is shown later to be their scaling limits.
We also consider their conformal pullbacks to smooth domains $D$,
which satisfy a non-constant version of massive holomorphicity. Like
their discrete ancestors, they have well-defined (imaginary parts
of) square integrals, which are then analysed as solutions of an elliptic
PDE. Near the boundary, we analyse both massive holomorphic functions
and their pullbacks under the umbrella of \emph{generalised analytic
functions}, while deferring some of the computations to the Appendix.
In Section \ref{sec:Discrete-Regularity-Theory}, we pursue a discrete
version of the regularity theory in the previous section. The discrete
square integral is seen to be critical in the analysis, and many properties
of the continuum integral, such as the maximum principle, have analogues
here. These results imply a certain bulk precompactness for the collection
of discrete observables. We then show that the discrete boundary condition
is preserved in the limit and provide required estimates of the degenerate
observable used in the proof of the RSW-type estimate, both of which
can be done without fixing a unique continuum limit. In Section \ref{sec:Continuum-Observable},
we show that any subsequential limit of the discrete observables has
to be unique, therefore finishing the proof of their convergence.
In the 2-point case, we use primarily potential-theoretic estimates
of the Dirichlet Laplacian Green function; in the 4-point case, we
use maximum principle to show that the 4-point square integral naturally
breaks up into two 2-point ones. In Section \ref{sec:Asymptotic-Analysis-of},
we study how the primal crossing probability in the 4-point boundary
condition, which is exactly encoded in the continuum square integral,
varies as $m$ tends to $0$ and $\infty$. We also show how to use
the latter and RSW-type estimates to get the desired asymptotic for
the characteristic length $L$. We finish by providing more involved
computations and theory reference in the Appendix.

\subsection*{Acknowledgement}

The author is supported by a KIAS Individual Grant (MG077201, MG077202) at Korea
Institute for Advanced Study. The author thanks Dmitry Chelkak for inviting him to \'Ecole Normale Superieure Paris and in particular informing him of \cite[Lemma A.2]{s-emb}, inspiring the proof of Proposition \ref{prop:2pt-uniqueness}. The author also
thanks Cl\'ement Hongler, Konstantin Izyurov, Kalle Kyt\"ol\"a, R\'emy Mahfouf,
Francesco Spadaro, and Yijun Wan for interesting discussions.

\section{Massive S-Holomorphic Observables\label{sec:Massive-S-Holomorphic-Observable}}

\subsection{Fermionic Observables\label{subsec:Fermionic-Observables}}

In this section, we define the main discrete object of our study,
the 2- and 4-point fermionic observables. These are discrete functions
built to reflect the combinatorics of the discrete (FK-)Ising model
which then have nontrivial scaling limits which we can identify in
the continuum. The following definition is essentially same as that
in \cite[(2.2)]{chsm2012}, albeit the expectation is evaluated with
different weights (if $q\neq0$).

Consider the 2-point case first. Recall a corner $\xi$ is associated
with the direction $\nu_{\xi}$, the unit complex number pointing
from the primal vertex to the dual vertex. Given any configuration,
we may draw the \emph{interface} $\gamma^{\delta}$ as the (unique
up to homotopy) curve separating the open cluster of $b_{\text{b}}^{\delta}$
from the dual cluster of $b_{\text{w}}^{\delta}$. We will start $\gamma^{\delta}$
from $b^{\delta}$ (i.e. the midpoint of $b_{\text{b}}^{\delta}$
and $b_{\text{w}}^{\delta}$) and go through each corner (or the midpoint
thereof) orthogonally: see Figure \ref{fig:intro_grid}R.
\begin{defn}
\label{def:2ptc}On every corner $\xi$ of $\left(\Omega^{\delta},a^{\delta},b^{\delta}\right)$,
define the \emph{$2$-point discrete fermionic observable} $F_{\left(\Omega^{\delta},a^{\delta},b^{\delta}\right)}^{\delta}$
by the FK-Ising expectation
\begin{equation}
F_{\left(\Omega^{\delta},a^{\delta},b^{\delta}\right)}^{\delta}\left(\xi\right)=\left(\frac{1}{2\delta\nu_{b^{\delta}}}\right)^{1/2}\mathbb{E}\left[\mathbf{1}\left(\xi\in\gamma^{\delta}\right)\cdot e^{-\frac{i}{2}\wind(\gamma^{\delta}:b^{\delta}\rightsquigarrow\xi)}\right],\label{eq:def2pt}
\end{equation}
defined up to a global sign (corresponding to the choice of the square
root), where $\wind(\gamma^{\delta}:b^{\delta}\rightsquigarrow\xi)$
is the total turning of the tangent of $\gamma^{\delta}$ starting
from $b^{\delta}$ to $\xi$.
\end{defn}

The sign of $F_{\left(\Omega^{\delta},a^{\delta},b^{\delta}\right)}^{\delta}$,
if required, may be easily fixed in any given domain $\Omega$, say,
by requiring a strictly positive real part in a small fixed neighbourhood;
we did not specify such a choice above for the sake of conciseness
and naturalness. See also (the proof) of \cite[Theorem 4.3]{chsm2012}. 

Since $\wind$ is determined by $\xi$ up to integer multiples of
$2\pi$ (recall $\gamma^{\delta}$ passes through $\xi$ orthogonally
with primal vertex on the right), $F_{\left(\Omega^{\delta},a^{\delta},b^{\delta}\right)}^{\delta}\left(\xi\right)$
necessarily lies on the line $\left(i\nu_{\xi}\right)^{-1/2}\mathbb{R}$.
Therefore, they are not considered full complex values of the functions:
instead, they are \emph{projections} (on the complex plane) on respective
lines of the full values to be defined on edges. We extend their definitions
to edges (rhombus centres) through the following proposition. To account
for the difference between the abstract angle $\hat{\theta}_{z}$
and the geometric angle $\overline{\theta}_{z}$, we consider $\hat{\xi}$
(depending implicitly on $z$) to be the rhombus edge corresponding
to $\xi$ in the virtual rhombus where $\overline{\theta}_{z}$ is
replaced by $\hat{\theta}_{z}$: i.e. define $\nu_{\hat{\xi}}:=e^{\pm i\left(\hat{\theta}_{z}-\overline{\theta}_{z}\right)}\nu_{\xi}$
with sign alternating along the rhombus. The idea of the below notion
and proof precisely comes from rotating (in terms of the fixed-phase
corner values) the \emph{critical} s-holomorphicity relation \cite[(2.6)]{chsm2012}
between the values at $z$ and the virtual corner $\hat{\xi}$ to
the physical rhombus with angle $\hat{\theta}_{z}$.
\begin{prop}
\label{prop:s-hol}At every interior edge (rhombus centre) $z\in\lozenge$,
there may be assigned a unique value $F^{\delta}(z)=F_{\left(\Omega^{\delta},a^{\delta},b^{\delta}\right)}^{\delta}(z)$
which makes the following equality true:
\begin{equation}
\text{Proj}\left[F^{\delta}(z);\left(i\nu_{\hat{\xi}}\right)^{-1/2}\mathbb{R}\right]=\left(\frac{\nu_{\hat{\xi}}}{\nu_{\xi}}\right)^{-1/2}F^{\delta}(\xi),\label{eq:shol}
\end{equation}
where $\xi$ is any of the four edges of the rhombus (i.e. corners)
centred at $z$ and $\left(\frac{\nu_{\hat{\xi}}}{\nu_{\xi}}\right)^{1/2}$
should be chosen with positive real part.
\end{prop}

\begin{proof}
This is a rephrasing in the Z-invariant case of the so-called 3-point
\emph{propagation equation}, which is purely combinatorial and valid
with any local weight (parametrised by abstract angle $\hat{\theta}$
as in introduction). Around a rhombus centre $z$, consider the double-valued
\emph{real observable} $X$ branching at $z$:
\[
X(\xi)=\left(i\nu_{\xi}\right)^{1/2}F^{\delta}(\xi).
\]

Then for any triple of adjacent corners $\xi_{0,1,2}$ (going counterclockwise)
around $z$, the propagation equation reads (see e.g. \cite[(1.5)]{s-emb})
\[
X\left(\xi_{1}\right)=\cos\hat{\theta}_{z}\cdot X\left(\xi_{2}\right)+\sin\hat{\theta}_{z}\cdot X\left(\xi_{0}\right).
\]

By \cite[Lemma 3.4]{chsm2012}, this implies the existence of a unique
value $F^{\delta}(z)$ that satisfies the following relation for each
$\xi$,
\[
\text{Proj}\left[F^{\delta}(z);\left(i\nu_{\hat{\xi}}\right)^{-1/2}\mathbb{R}\right]=\left(i\nu_{\hat{\xi}}\right)^{-1/2}X(\xi),
\]
which is equivalent to (\ref{eq:shol}), being careful to use a single
branch $\left(i\nu\right)^{1/2}$ such that $\left(i\nu_{\xi}\right)^{1/2}$
is always perturbed by $e^{\pm\frac{i}{2}\left(\hat{\theta}_{z}-\overline{\theta}_{z}\right)}$
when $\nu_{\xi}$ is perturbed by $e^{\pm i\left(\hat{\theta}_{z}-\overline{\theta}_{z}\right)}$
(note $\hat{\theta}-\overline{\theta}\in\left[0,\frac{\pi}{2}\right)$).
\end{proof}
By Lemma \ref{lem:dmhol}, (\ref{eq:shol}) is a discrete notion of
\emph{massive holomorphicity} $\bar{\partial}f+im\bar{f}=0$. Equivalent
massive observables have been considered on the square lattice \cite{bedc,dgp,hkz}. 
\begin{defn}
We call (\ref{eq:shol}) (($q,k$)-\emph{massive) s-holomorphicity}.
\end{defn}

In the case of $4$-point observables, we need to work with two interface
curves: to define an observable as in Definition \ref{def:2ptc},
we need to merge them into a single interface. A natural way of doing
this, developed in the proof of \cite[Theorem 6.1]{chsm2012}, is
to \emph{externally connect} the boundary segments. We summarise the
construction here.

Compared to the original measure with two interface curves, externally
dual-connecting the two dual-wired boundary segments (specifically,
draw an external dual edge $\left\langle a_{\text{w}}^{\delta}d_{\text{w}}^{\delta}\right\rangle $
closely to the original wired segment $\left(d^{\delta}a^{\delta}\right)$)
yields a measure $\mathbb{P}_{\text{p}}$ which augments the relative
weights of configurations not in $\stackrel{\leftrightarrow}{\Omega^{\delta}}$
by a factor of $\sqrt{2}$. On the other hand, connecting the two
primal segments (draw an external primal edge $\left\langle c_{\text{b}}^{\delta}d_{\text{b}}^{\delta}\right\rangle $
close to the dual-wired segment $\left(c^{\delta}d^{\delta}\right)$)
externally augments the relative weights of configurations in $\stackrel{\leftrightarrow}{\Omega^{\delta}}$
by a factor of $\sqrt{2}$. In both cases, we define massive s-holomorphic
observables $F_{\text{p}}^{\delta},F_{\text{d}}^{\delta}$ by using
(\ref{eq:def2pt}), drawing single interfaces through $b^{\delta}$
thanks to addition of the external edges as above. Recall the probability
$\mathrm{P}^{\delta}=\mathbb{P}\left[\stackrel{\leftrightarrow}{\Omega^{\delta}}\right]=\mathbb{P}_{\text{p}}\left[\stackrel{\leftrightarrow}{\Omega^{\delta}}\right]$
that the primal segments $\left(b^{\delta}c^{\delta}\right)$ and
$\left(d^{\delta}a^{\delta}\right)$ are (internally) connected.
\begin{defn}
\label{def:d4pt}Let $\mathrm{Q}^{\delta}:=1-\mathrm{P}^{\delta}$.
Define the \emph{$4$-point Dobrushin fermionic observable} $F_{\left(\Omega^{\delta},a^{\delta},b^{\delta},c^{\delta},d^{\delta}\right)}^{\delta}$
by
\[
F_{\left(\Omega^{\delta},a^{\delta},b^{\delta},c^{\delta},d^{\delta}\right)}^{\delta}:=\frac{\mathrm{P}^{\delta}\left(\sqrt{2}\mathrm{P}^{\delta}+\mathrm{Q}^{\delta}\right)F_{\text{d}}^{\delta}+\mathrm{Q}^{\delta}\left(\mathrm{P}^{\delta}+\sqrt{2}\mathrm{Q}^{\delta}\right)F_{\text{p}}^{\delta}}{\mathrm{P}^{\delta}\left(\sqrt{2}\mathrm{P}^{\delta}+\mathrm{Q}^{\delta}\right)+\mathrm{Q}^{\delta}\left(\mathrm{P}^{\delta}+\sqrt{2}\mathrm{Q}^{\delta}\right)},
\]
on the corners of $\Omega^{\delta}$ and then edges by (\ref{eq:shol}).
\end{defn}

This observable encodes the desired connection probability $\mathrm{P}^{\delta}$
through the boundary value problem of Proposition \ref{prop:4pt_bc}.

\subsection{Integral of the Square and the Boundary Value Problem\label{subsec:dbvp}}

In this section, we define \emph{square integrals} of massive s-holomorphic
functions (functions $F$ on rhombus centres and edges satisfying
(\ref{eq:shol})). These are discrete counterparts of the imaginary
part of the line integral $\int F^{2}dz$. 
\begin{lem}
\label{lem:defh}Given a massive s-holomorphic function $F$ on a
simply connected discrete domain $\Omega_{\delta}$ , the real-valued
function $H$ on $\Lambda$ constructed by:
\begin{equation}
H(u)-H(w):=2\delta\cdot\left\vert F(\xi)\right\vert ^{2},\label{eq:defh}
\end{equation}
for a corner $\xi=\left\langle uw\right\rangle $ with adjacent $u\in\Gamma,w\in\Gamma^{*}$,
is well defined.

We will write $H:=\text{Im}\int^{\delta}F^{2}dz$ in the sense that,
the discrete derivatives across the centre $z\in\lozenge$ of the
rhombus bordered by $u,w,u_{z},w_{z}$ (see Figure \ref{fig:app}L)
satisfy

\begin{align}
H(u_{z})-H(u) & =\frac{\cos\hat{\theta}_{z}}{\cos\bar{\theta}_{z}}\text{Im}\left[\left(u_{z}-u\right)\cdot F(z)^{2}\right],\label{eq:inth}\\
H(w_{z})-H(w) & =\frac{\sin\hat{\theta}_{z}}{\sin\bar{\theta}_{z}}\text{Im}\left[\left(w_{z}-w\right)\cdot F(z)^{2}\right].\nonumber 
\end{align}
\end{lem}

\begin{proof}
It suffices to check the well-definedness on each rhombus, i.e. going
around the $4$ corners. Then
\[
\left\vert F(z)\right\vert ^{2}=\left\vert F(\left\langle uw\right\rangle )\right\vert ^{2}+\left\vert F(\left\langle u_{z}w_{z}\right\rangle )\right\vert ^{2}=\left\vert F(\left\langle uw_{z}\right\rangle )\right\vert ^{2}+\left\vert F(\left\langle u_{z}w\right\rangle )\right\vert ^{2},
\]
immediate from (\ref{eq:shol}), in fact implies well-definedness.
Both statements in fact can be interpreted as reformulation of the
massless case \cite[Proposition 3.6]{chsm2012} since $F$ may be
considered a massless s-holomorphic function on the virtual rhombus
with half-angle replaced $\bar{\theta}\to\hat{\theta}$ and rotated
values on corners as in Proposition \ref{prop:s-hol}.
\end{proof}
Now we present the so-called \emph{discrete Riemann-Hilbert boundary
value problem} which the observables satisfy, which are discrete ancestors
of the continuous boundary value problems of Definition \ref{def:crhh}.
While it is most intuitively phrased in terms of boundary phases of
the massive s-holomorphic observable $F$, the version that translates
most naturally to possibly rough domains is the corresponding condition
for $H$. Unlike their continuous counterparts (Proposition \ref{prop:crhf}
and Definition (\ref{def:crhh})), these two discrete conditions are
equivalent.

The following boundary condition, essentially identical to that in
the critical case (\cite[(2.5)]{chsm2012} and the \emph{boundary
modification} combining elements of \cite{chsm2012,duminil-copin-hongler-nolin,chelkak-hongler}),
is satisfied both for the $2$- and $4$-point observables on their
respective free and wired arcs, so we write $F^{\delta}$ for either.
The definition of $F^{\delta}$ and extension by (\ref{eq:shol})
holds for all edges in $\lozenge\left(\Omega^{\delta}\right)$; this
gives enough information to define $H^{\delta}:=\imm \int^{\delta}\left(F^{\delta}\right)^{2}dz$
on all (primal and dual) vertices on the closed domain bounded by
the boundary arcs $\partial\lozenge\left(\Omega^{\delta}\right)$
(i.e. where the boundary conditions are set) and the marked boundary
corners. We now extend $F^{\delta}$ to $\partial\lozenge\left(\Omega^{\delta}\right)$,
which then yields a natural extension of $H^{\delta}$ to the external
vertices, i.e. the vertices on the outer halves of boundary rhombi
bisected by $\partial\lozenge\left(\Omega^{\delta}\right)$.
\begin{lem}
\label{lem:dbvp}A 2- or 4-point observable $F^{\delta}$ may be extended
to the boundary edges $\underline{z}\in\partial\lozenge\left(\Omega^{\delta}\right)$
satisfying the following equivalent \emph{discrete Riemann-Hilbert}
boundary conditions.

Suppose $\underline{z}$ is on the wired arc (set $\upsilon:=1$)
such that the corresponding rhombus $\left\langle \underline{u}_{1}\underline{z}_{\text{int}}\underline{u}_{2}\underline{z}_{\text{ext}}\right\rangle $
(see Figure \ref{fig:intro_grid}R) is bisected by the primal boundary
edge $\left\langle \underline{u}_{1}\underline{u}_{2}\right\rangle $
with the dual vertex $\underline{w}_{\text{int}}$ in the interior.
Consider the unit tangent vector $\nu_{\text{tan}}=\nu_{\text{tan}}(\underline{z})=\frac{\underline{u}_{2}-\underline{u}_{1}}{\left\vert \underline{u}_{2}-\underline{u}_{1}\right\vert }$.
\begin{itemize}
\item $\left(RH\right)_{F}$: $F^{\delta}\left(\underline{z}\right)$ may
be defined as the unique value satisfying (\ref{eq:shol}) (with its
two interior corners), which belongs to $\sqrt{\frac{\upsilon}{\nu_{\text{tan}}}}\mathbb{R}$.
\item $(RH)_{H}$: the square integral $H^{\delta}:=\imm\int^{\delta}\left(F^{\delta}\right)^{2}dz$
may be defined at $\underline{z}_{\text{ext}}$ such that it
\begin{itemize}
\item stays constant on $\underline{u}_{1},\underline{u}_{2},\underline{z}_{\text{ext}}$;
\item is consistent with $F^{\delta}\left(\underline{z}\right)$ and (\ref{eq:inth}),
as long as one replaces $\sin\hat{\theta}_{\underline{z}}\to\frac{\sin\hat{\theta}_{\underline{z}}+1}{2}$.
\end{itemize}
\end{itemize}
Note that (\ref{eq:inth}) implies that $H^{\delta}(\underline{z}_{\text{int}})\leq H^{\delta}(\underline{z}_{\text{ext}})$:
$H^{\delta}$ has \emph{nonnegative outer normal derivative} on wired
arcs.

On the free arc, we exchange the roles of primal and dual vertices,
set $\upsilon:=-1$, and $\cos\hat{\theta}_{\underline{z}}\to\frac{\cos\hat{\theta}_{\underline{z}}+1}{2}$
on the boundary, yielding \emph{nonpositive outer normal derivative}.
\end{lem}

\begin{proof}
Without loss of generality, we will check the wired arc case. Note
that the interface $\gamma^{\delta}$ passes through $\left\langle \underline{u}_{1}\underline{z}_{\text{int}}\right\rangle $
if and only if it passes through $\left\langle \underline{u}_{2}\underline{z}_{\text{int}}\right\rangle $,
and with deterministic $\wind$ for both points. That is, we have
\begin{equation}
F^{\delta}\left(\left\langle \underline{u}_{1,2}\underline{z}_{\text{int}}\right\rangle \right)=\frac{e^{-\frac{i}{2}\wind(\gamma^{\delta}:b^{\delta}\rightsquigarrow\left\langle \underline{u}_{1,2}\underline{z}_{\text{int}}\right\rangle )}}{\left(2\delta\right)^{1/2}}\mathbb{P}\left[\left\langle \underline{u}_{1,2}\underline{z}_{\text{int}}\right\rangle \in\gamma^{\delta}\right],\label{eq:ftoprob}
\end{equation}
with the probability coinciding both cases. Note that $\wind(\gamma^{\delta}:b^{\delta}\rightsquigarrow\left\langle \underline{u}_{1}\underline{z}_{\text{int}}\right\rangle )-\left(\frac{\pi}{2}-\hat{\theta}_{\underline{z}}\right)=\wind(\gamma^{\delta}:b^{\delta}\rightsquigarrow\left\langle \underline{u}_{2}\underline{z}_{\text{int}}\right\rangle )+\left(\frac{\pi}{2}-\hat{\theta}_{\underline{z}}\right)$,
and in fact $e^{\frac{i}{2}\left[\wind(\gamma^{\delta}:b^{\delta}\rightsquigarrow\left\langle \underline{u}_{1}\underline{z}_{\text{int}}\right\rangle )-\left(\frac{\pi}{2}-\hat{\theta}_{\underline{z}}\right)\right]}$
is a square root of $i\tau$. Then with this choice of the square
root, it is simple to check that $F^{\delta}(\underline{z}):=\frac{\mathbb{P}\left[\left\langle \underline{u}_{1,2}\underline{z}_{\text{int}}\right\rangle \in\gamma^{\delta}\right]}{\sqrt{i\tau}\left(2\delta\right)^{1/2}\cos\left(\frac{\pi}{4}-\frac{\hat{\theta}_{\underline{z}}}{2}\right)}$
satisfies $(RH)_{F}$. The fact that $H(\underline{u}_{1})=H(\underline{u}_{2})$
is also clear from $\left\vert F^{\delta}\left(\left\langle \underline{u}_{1}\underline{z}_{\text{int}}\right\rangle \right)\right\vert =\left\vert F^{\delta}\left(\left\langle \underline{u}_{2}\underline{z}_{\text{int}}\right\rangle \right)\right\vert $
and (\ref{eq:inth}).

On the other hand, if we \emph{define} $H^{\delta}(\underline{z}_{\text{ext}})=H^{\delta}(\underline{u}_{1,2})$,
then according to (\ref{eq:inth}), it must be that $F^{\delta}(\underline{z})=\frac{\mathbb{P}\left[\left\langle \underline{u}_{1,2}\underline{z}_{\text{int}}\right\rangle \in\gamma^{\delta}\right]}{\sqrt{i\tau}\left(2\delta\right)^{1/2}\sin^{1/2}\hat{\theta}_{\underline{z}}}$.
This discrepancy can be fixed by, as in the statement, replacing $\sin\hat{\theta}_{\underline{z}}\to\frac{\sin\hat{\theta}_{\underline{z}}+1}{2}$
which equals $\cos^{2}\left(\frac{\pi}{4}-\frac{\hat{\theta}_{\underline{z}}}{2}\right)$.
\end{proof}
Now we specialise and illustrate what the resulting boundary values
of the square integrals of the discrete $2$- and $4$-point observables
look like. Note that the previous lemma implies that it is possible
to define a \emph{constant boundary value} on an entire wired or free
arc: on (say) a wired arc, the value $H(\underline{z}_{\text{ext}})=H(\underline{u}_{1,2})$
stays constant over each boundary edge $\underline{z}$, since two
adjacent boundary edges share a vertex. It remains to specify the
value on each boundary arc, which we now recall.
\begin{prop}[{\cite[(4.3)]{chsm2012}}]
\label{prop:2pt_bc}One may fix the additive constant so that the
square integral $H_{\left(\Omega^{\delta},a^{\delta},b^{\delta}\right)}^{\delta}:=\int^{\delta}\left(F_{\left(\Omega^{\delta},a^{\delta},b^{\delta}\right)}^{\delta}\right)^{2}dz$
of the $2$-point observable has the following boundary values:
\[
H_{\left(\Omega^{\delta},a^{\delta},b^{\delta}\right)}^{\delta}\equiv0\text{ on }\left(a^{\delta}b^{\delta}\right)\text{, }H_{\left(\Omega^{\delta},a^{\delta},b^{\delta}\right)}^{\delta}\equiv1\text{ on }\left(b^{\delta}a^{\delta}\right).
\]
\end{prop}

\begin{proof}
By Lemma \ref{lem:dbvp}, it suffices to characterise the difference
of values at $b_{\text{w}}^{\delta}$ (which belongs to $\left(a^{\delta}b^{\delta}\right)$)
and $b_{\text{b}}^{\delta}$ (which belongs to $\left(b^{\delta}a^{\delta}\right)$).
But since any interface $\gamma^{\delta}$ passes through $b^{\delta}$,
we have $F_{\left(\Omega^{\delta},a^{\delta},b^{\delta}\right)}^{\delta}(b^{\delta})=\frac{1}{\left(2\delta\right)^{1/2}}$.
By (\ref{eq:inth}), we have $H_{\left(\Omega^{\delta},a^{\delta},b^{\delta}\right)}^{\delta}(b_{\text{b}}^{\delta})-H_{\left(\Omega^{\delta},a^{\delta},b^{\delta}\right)}^{\delta}(b_{\text{w}}^{\delta})=1$,
as desired.
\end{proof}
\begin{prop}[{\cite[(6.5)]{chsm2012}}]
\label{prop:4pt_bc}One may fix the additive constant so that the
square integral $H_{\left(\Omega^{\delta},a^{\delta},b^{\delta},c^{\delta},d^{\delta}\right)}^{\delta}:=\int^{\delta}\left(F_{\left(\Omega^{\delta},a^{\delta},b^{\delta},c^{\delta},d^{\delta}\right)}^{\delta}\right)^{2}dz$
of the $4$-point observable has the following boundary values:
\begin{align*}
H_{\left(\Omega^{\delta},a^{\delta},b^{\delta}\right)}^{\delta}&\equiv0\text{ on }\left(a^{\delta}b^{\delta}\right)\text{, }H_{\left(\Omega^{\delta},a^{\delta},b^{\delta}\right)}^{\delta}\equiv1\text{ on }\left(b^{\delta}c^{\delta}\right)\text{, }\\
H_{\left(\Omega^{\delta},a^{\delta},b^{\delta}\right)}^{\delta}&\equiv\chi^{\delta}\text{ on }\left(c^{\delta}d^{\delta}\right)\cup\left(d^{\delta}a^{\delta}\right),
\end{align*}
where $T^{\delta}:=\frac{\mathrm{P}^{\delta}}{\mathrm{Q}^{\delta}},\chi^{\delta}:=\left[\frac{\left(T^{\delta}\right)^{2}+\sqrt{2}T^{\delta}}{\left(T^{\delta}\right)^{2}+\sqrt{2}T^{\delta}+1}\right]^{2}\in\left(0,1\right)$.
\end{prop}

\begin{proof}
Here, we need to study the jumps from $b_{\text{w}}^{\delta}$ to
$b_{\text{b}}^{\delta}$, $d_{\text{w}}^{\delta}$ to $d_{\text{b}}^{\delta}$,
and $a_{\text{b}}^{\delta}$ to $a_{\text{w}}^{\delta}$. We will
give the final case as an illustration. Note that with the external
dual edge $\left\langle a_{\text{w}}^{\delta}d_{\text{w}}^{\delta}\right\rangle $
as in $\mathbb{P}_{\text{p}}$, the interface $\gamma^{\delta}$ always
goes through $a^{\delta}$ if and only if $\left(b^{\delta}c^{\delta}\right)$
and $(d^{\delta}a^{\delta})$ are internally dual-connected (with
probability $\frac{\mathrm{P}^{\delta}}{\mathrm{P}^{\delta}+\sqrt{2}\mathrm{Q}^{\delta}}$),
while with the external primal edge $\left\langle c_{\text{b}}^{\delta}d_{\text{b}}^{\delta}\right\rangle $
as in $\mathbb{P}_{\text{d}}$, the interface $\gamma^{\delta}$ goes
through $d^{\delta}$ always. In both cases $\wind(\gamma^{\delta}:b^{\delta}\rightsquigarrow d^{\delta})$
is deterministic and in fact is the same. From Definition \ref{def:d4pt},
we have
\[
\left\vert \left(2\delta\right)^{1/2}F_{\left(\Omega^{\delta},a^{\delta},b^{\delta},c^{\delta},d^{\delta}\right)}^{\delta}(d^{\delta})\right\vert =\left\vert \frac{\mathrm{P}^{\delta}\left(\sqrt{2}\mathrm{P}^{\delta}+\mathrm{Q}^{\delta}\right)+\mathrm{P}^{\delta}\mathrm{Q}^{\delta}}{\mathrm{P}^{\delta}\left(\sqrt{2}\mathrm{P}^{\delta}+\mathrm{Q}^{\delta}\right)+\mathrm{Q}^{\delta}\left(\mathrm{P}^{\delta}+\sqrt{2}\mathrm{Q}^{\delta}\right)}\right\vert =\sqrt{\chi^{\delta}}.
\]
\end{proof}

\section{Massive Holomorphic Functions\label{sec:Massive-Holomorphic-Functions}}

In this section, we focus on the regularity theory of the solutions
of Bers-Vekua equation
\begin{equation}
\bar{\partial}g+\alpha i\bar{g}=0.\label{eq:brvk}
\end{equation}
We use the Wirtinger derivatives $\partial:=\frac{1}{2}\left(\partial_{x}-i\partial_{y}\right)$,
$\bar{\partial}:=\frac{1}{2}\left(\partial_{x}+i\partial_{y}\right)$
frequently.

In massive holomorphic functions of our interest, $\alpha$ is purely
real. This itself in fact constitutes a significant additional structure:
it allows for the definition of (the imaginary part of) the integral
of the square $h=\imm\int f^{2}dz$, which serves as a powerful analytic
tool. Consequently, novel results in this section mainly come from
the interplay between the square integral $h$ and the generalised
analytic function theory of complex $\alpha\in L^{r}\left(D\right)$.
For the latter, we use a result established recently in \cite{baratchart}
which necessitates use of Sobolev space methods, especially the Sobolev
and trace inequalities. While $f$ which we study turns out to be
smooth (Corollary \ref{cor:smoothness}) in the bulk, the use of Sobolev
trace is critical in treating the boundary behaviour of $f$.

\subsection{\label{subsec:Functions-on-Physical}Functions on Physical and Pullback
Domains}

We need to first precisely state exactly in which sense (\ref{eq:brvk})
should be satisfied. Recall the Sobolev space $W^{1,p}$ of complex-valued
functions having weak derivatives in $L^{p}$: being locally Lipschitz
(which is natural for functions obtained from Arzel\`a-Ascoli) is synonymous
to being in $W_{\text{loc}}^{1,\infty}\subset W_{\text{loc}}^{1,p}$
(e.g. \cite{evans-gariepy}). We recap more theory in the Appendix.

Any function $g$ which satisfies (\ref{eq:brvk}) with respect to
the weak derivative is called \emph{generalised analytic} \cite{vek}
or \emph{pseudoanalytic of the first type} \cite{bers}. We will study
two specific types thereof. Let $\Omega$ be a general bounded simply
connected domain as usual (the 'physical' domain), and $D$ be another
simply connected domain (the 'pullback' domain), which will be assumed
to be smooth (i.e. locally a graph of a smooth function) and bounded
unless otherwise stated. First, locally Lipschitz functions 
\begin{equation}
f\in W_{\text{loc}}^{1,\infty}\left(\Omega\right)\text{, }\bar{\partial}f+mi\bar{f}=0,\label{eq:mchol}
\end{equation}
which has constant $\alpha\equiv m$ (called \emph{massive holomorphic}),
will serve as scaling limits of discrete observables of Section \ref{sec:Massive-S-Holomorphic-Observable};
second, pullbacks 
\[
f^{\text{pb}}:=\left(f\circ\varphi\right)\cdot\left(\varphi'\right)^{1/2}\in W_{\text{loc}}^{1,\infty}\left(D\right)\text{, }\bar{\partial}f+m\left\vert \varphi'\right\vert i\bar{f}=0,
\]
for a conformal map $\varphi:D\to\Omega$, which has $\alpha=m\left\vert \varphi'\right\vert \in L^{2}\left(D\right)$,
will be used to study boundary conditions of $f$. The square root
$\left(\varphi'\right)^{1/2}$ may be chosen to be globally holomorphic
since $\varphi'$ never vanishes. Accordingly, recall the bounded
trace operator $\tr_{\partial D}:W^{1,p}(D)\to W^{1-\frac{1}{p},p}(\partial D)$
for $p>1$.
\begin{rem}
\label{rem:cgreens}Since we use it frequently, let us note the regularity
requirements for Green-Riemann's formula (which is simply Green's
theorem in the complex notation): for $g\in C^{1}\left(\overline{U}\right)$
on a Lipschitz bounded domain $U$,
\[
\int_{\partial U}gdz=2i\iint_{U}\bar{\partial}gd^{2}z.
\]
By the density of $C^{1}\left(\overline{U}\right)$ in $W^{1,p}\left(U\right)$
and continuity of the trace operator, we may apply the above to $g\in W^{1,p}\left(U\right)$
and its trace on $\partial U$. See e.g. \cite{evans-gariepy} for
a reference.

As \cite[Theorem 1.31]{vek} notes (and easily seen again by density
arguments), for any $g\in W_{\text{loc}}^{1,p}$ the value of weak
$\bar{\partial}g$ at its Lebesgue points $z$ (which are almost everywhere
in $D$) may be computed by a contour integral:
\[
\bar{\partial}g(z)=\lim_{r\to0}\frac{1}{2\pi r^{2}i}\int_{\partial B_{r}(z)}gdz.
\]
\end{rem}

We now state the factorisation theorem and its inverse for a general
$L^{r}$-coefficient $\alpha$ for $r\geq2$.  We will use the constant symbols
$\const(\cdot), \smallc(\cdot)$ (to be recycled) to denote strictly
positive quantities which depend only on quantities in the parenthesis.
\begin{thm}[{\cite[Lemma 3.1, Theorem 4.1]{baratchart}}]
\label{thm:similarity}Let $r\geq2$ and $g\in W_{\text{loc}}^{1,r}\left(D\right)$
is $\alpha$-generalised analytic (\ref{eq:brvk}) with $\alpha\in L^{p}\left(D\right)$
on a smooth and bounded simply connected domain $D$. Then there exists
unique $s=s_{g}^{D}\in W^{1,r}(D)$ such that $g=e^{s}\underline{g}$
for a holomorphic function $\underline{f}$ (the \emph{holomorphic
part)} on $D$, $\imm\tr_{\partial\Omega}s\equiv0$, $\int_{\partial D}\re\tr_{\partial D}s\left\vert dz\right\vert =0$,
and
\[
\left\vert s\right\vert _{W^{1,r}(D)}\leq\const(r,D)\left\vert \alpha\right\vert _{L^{r}(D)}.
\]
Conversely, given any holomorphic function $\underline{g}$, there
exists unique $\alpha$-generalised analytic $g\in W_{\text{loc}}^{1,r}\left(D\right)$,
which is factorised according to above as $g=e^{s}\underline{g}$.
\end{thm}

\begin{rem}
\label{rem:similarity}Theorem \ref{thm:similarity} is also called
Bers \emph{similarity principle} because it implies that generalised
analytic functions share many properties with holomorphic functions.
For example, if $r>2$, by Sobolev inequality $s$ has a H\"older continuous
representative; this means that $f$ can only vanish in $D$ polynomially
and at isolated points, exactly when $\underline{g}$ does. As \cite[Lemma 3.1]{baratchart}
notes, even when $r=2$, $s$ can only blow up in $\bar{D}$ on a
subset of Hausdorff dimension $0$: only there can it create additional
zeros or poles for $g$ which were not in $\underline{g}$. In fact,
since $e^{s}\in L^{q}\left(D\right)$ for any $q\in\left[1,\infty\right)$
(see Lemma \ref{lem:general_sobolev}), it cannot fully cancel any
zero or pole of $\underline{g}$ that has a power law behaviour.
\end{rem}

\begin{defn}
\label{def:holpart}Applying Theorem \ref{thm:similarity} with $r=2$
and $\alpha=m\left\vert \varphi'\right\vert $, we may factorise for unique
$s=s_{f^{\text{pb}}}^{D}$ 
\[
f^{\text{pb}}=e^{s_{f^{\text{pb}}}^{D}}\underline{f}^{\text{pb}},
\]
 where $\underline{f}^{\text{pb}}$ is holomorphic on $D$ and $\left\vert s_{f^{\text{pb}}}^{D}\right\vert _{W^{1,2}\left(D\right)}\leq\const(D)m\diam\Omega$.
As usual, we drop the indices when they are assumed throughout.

Note
that the derivative $L^{2}$-norm is conformally invariant: $\left\vert \nabla s_{f^{\text{pb}}}^{D}\right\vert _{L^{2}\left(D\right)}=\left\vert \nabla\left(s_{f^{\text{pb}}}^{D}\circ\varphi\right)\right\vert _{L^{2}\left(\Omega\right)}$.
Given the uniqueness of $s$, it is then straightforward to check
conformal invariance of $s$: for two bounded pullback domains $D_{1,2}$
using two conformal maps $\varphi_{1,2}:D_{1,2}\to\Omega$, $s^{D_{1}}=s^{D_{2}}\circ\left(\varphi_{2}^{-1}\circ\varphi_{1}\right)$.
Therefore, we may factorise $f=e^{s_{f}^{\Omega}}\underline{f}$ on
$\Omega$ with
\begin{alignat}{1}
 & s_{f}^{\Omega}:=s_{f^{\text{pb}}}^{D}\circ\varphi^{-1}\text{ and }\underline{f}:=\left(\underline{f}^{\text{pb}}\circ\varphi^{-1}\right)\cdot\left(\left(\varphi^{-1}\right)'\right)^{1/2},\label{eq:omega_factorisation}\\
 & \left\vert \nabla s_{f}^{\Omega}\right\vert _{L^{2}\left(\Omega\right)}\leq\const m\diam\Omega,\label{eq:delsl2norm}
\end{alignat}
independent of the choice of $D$. We again call $\underline{f}$
the \emph{holomorphic part} of $f$.
\end{defn}

\subsection{Basic Properties of $h$}

First, we note that the imaginary part of the integral of the square
is well-defined for massive holomorphic functions.
\begin{lem}
\label{lem:h_welldef}Given any massive holomorphic function $f\in W_{\text{loc}}^{1,\infty}\left(\Omega\right)$,
the contour integral
\[
h(z):=\imm\int_{z_{0}}^{z}f^{2}dz,
\]
is independent of the (piecewise smooth) contour from $z_{0}$ to
$z$, and $h\in W_{\text{loc}}^{2,\infty}\left(\Omega\right)$ with
weak Laplacian
\begin{equation}
\Delta h:=4\bar{\partial}\partial h=-8m\left\vert \partial h\right\vert .\label{eq:heq}
\end{equation}

Moreover, $h$ is conformally invariant: if $f^{\text{pb}}=\left(f\circ\varphi\right)\cdot\left(\varphi'\right)^{1/2}$
for a conformal map $\varphi:D\to\Omega$, then $\imm\int\left(f^{\text{pb}}\right)^{2}dz=h\circ\varphi=:h^{\text{pb}}$.
\end{lem}

\begin{proof}
Given two contours $\gamma,\gamma'$ from $z_{0}$ to $z$, concatenation
of $\gamma$ and the reverse of $\gamma'$ bounds an open subset of
$\Omega$; by local uniform continuity of $f$, we may assume this
subset is a smooth domain $U\Subset D$. Then Green-Riemann's formula
on $f\in W^{1,\infty}\left(U\right)$ gives
\[
\int_{\gamma}f^{2}dz-\int_{\gamma'}f^{2}dz=-2m\iint_{U}\left\vert f\right\vert ^{2}d^{2}z\in\mathbb{R},
\]
and therefore the imaginary part $h$ does not depend on $\gamma$.
Then we have (strong) derivative $\partial h=if^{2}/2$ and its (weak)
derivative
\[
4\bar{\partial}\partial h=-4m\left\vert f\right\vert ^{2}=-8m\left\vert \partial h\right\vert .
\]

The conformal covariance follows from a simple change of variables.
\end{proof}
The equation (\ref{eq:heq}) is an elliptic semilinear equation, whose
solutions satisfy comparison and maximum principles (see, e.g. \cite[Chapter 10]{GiTr})
which we show below for the sake of completeness. 
\begin{lem}
\label{lem:comparison}If $h_{1},h_{2}\in C^{2}\left(U\right)\cap C(\overline{U})$
are two solutions of (\ref{eq:heq}) respectively with masses $m=m_{1}\leq m_{2}$
on any bounded domain $U$, then $h_{1}\leq h_{2}$ on $\partial U$
implies it in $U$.
\end{lem}

\begin{proof}
Let $\epsilon,\eta>0$. Suppose $h^{\dagger}:=h_{1}-h_{2}-\epsilon\eta^{-1}e^{-\eta x}$
has a extremum in $U$. Then
\[
\partial h^{\dagger}:=\partial h_{1}-\partial h_{2}+\epsilon e^{-\eta x}=0,
\]
 so by setting $\eta\gg1$, we can make sure that at such an extremum
\[
\Delta h^{\dagger}=-8m_{1}\left(\left\vert \partial h_{1}\right\vert ^{2}-\left\vert \partial h_{2}\right\vert ^{2}\right)+8(m_{2}-m_{1})\left\vert \partial h_{2}\right\vert ^{2}+\epsilon\eta e^{-\eta x}>0.
\]
Therefore $h^{\dagger}$ does not have an interior maximum, and therefore
$h_{1}-h_{2}\leq\const(\diam\Omega)\epsilon\eta^{-1}$. Now $\epsilon\to0$
gives the result.
\end{proof}
We can improve the 'weak' maximum principle which comes from the comparison
principle to a strong maximum principle.
\begin{prop}
\label{prop:strong}A smooth solution $h$ of (\ref{eq:heq}) satisfies
the the strong maximum (resp. minimum) principles: if there is $z_{0}\in U$
such that
\[
h(z_{0})=\sup_{U}h\text{ (resp. }h(z_{0})=\inf_{U}h\text{)},
\]
$h$ is constant. Since this result is purely topological, we may
apply it to the pullback $h^{\text{pb}}=h\circ\varphi$ satisfying
$\Delta h^{\text{pb}}=-8m\left\vert \varphi'\right\vert \left\vert \partial h^{\text{pb}}\right\vert $.
\end{prop}

\begin{proof}
Since $h$ is superharmonic (recall that $\Delta h=-8m\left\vert \partial h\right\vert $),
strong minimum principle holds; alternatively, the following strategy
may also be applied.

Strong maximum principle follows in the standard manner once we have
the following negative solution of the same PDE on $B_{r}\left(0\right)\setminus\left\{ 0\right\} $
vanishing on $\partial B_{r}\left(0\right)$:
\[
h_{0}(z)=\text{Ei}(-4m\left\vert z\right\vert )-\text{Ei}(-4mr),
\]
where $\text{Ei}(x)=\int_{0}^{x}\frac{e^{t}}{t}dt.$ If $h$ is not
constant we may take a small ball (say) $\overline{B_{r}\left(0\right)}\subset\Omega$
such that $z_{0}\in\partial B_{r}(0)$ and $h<h(z_{0})=\sup_{\Omega}h$
in $B_{r}(0)$ (see e.g. the proof of \cite[Theorem 3.5]{GiTr}).
Then $\epsilon:=h(z_{0})-\sup_{\partial B_{r/2}\left(0\right)}h>0$,
and $h-h(z_{0})$ is bounded above by $\epsilon^{-1}h_{0}(z)$ in
$B_{r}\left(0\right)\setminus B_{r/2}\left(0\right)$ by the comparison
principle. However, $\epsilon^{-1}h_{0}$ has a strictly positive
outer derivative at $z_{0}$ on the boundary of $B_{r}(0)$, and by
comparison $h-h(z_{0})$ does as well. This contradicts the assumption
that $z_{0}$ is an interior maximum of $h$.
\end{proof}

\subsection{Bulk and Boundary Regularity}

In this section, we will give various regularity estimates of $f$
and $f^{\text{pb}}$, often using or in terms of the square integrals.

By bilinearity, the contour integral of the product $\imm\int fgdz$
for massive holomorphic functions (or their pullbacks) $f,g$ is well-defined.
By setting $g$ to be a massive equivalent of the Cauchy kernel $\frac{1}{z-w}$,
we can write a massive Cauchy's integral formula (Proposition \ref{prop:ccauchy}),
which we defer to the Appendix. As a direct consequence, we have the
following.
\begin{cor}
\label{cor:smoothness}A massive holomorphic function $f$, therefore
$h=\imm\int f^{2}dz$, is smooth on $\Omega$. As a consequence, the
pullbacks $f^{\text{pb}},h^{\text{pb}}$ are also smooth. 
\end{cor}

\begin{proof}
These are direct consequences of the estimates in Lemma \ref{lem:cdiff}
and (\ref{eq:ccauchy}), which may be differentiated infinite times
to yield smoothness.
\end{proof}
More quantitatively, we have following estimates of a massive holomorphic
function in the bulk, similarly to Proposition \ref{prop:dreg-bulk}.
\begin{prop}
\label{prop:creg-bulk}Let $z$ be a point in $\Omega$ (resp. $D$).
We have, for a universal $\const$,
\[
\left\vert f(z)\right\vert \leq\const\frac{\left(\osc_{B_{r}(z)}h\right)^{1/2}}{r^{1/2}}\text{ (resp.}\left\vert f^{\text{pb}}(z)\right\vert \leq\const\frac{\left(\osc_{B_{r}(z)}h\right)^{1/2}}{r^{1/2}}\text{)},
\]
where $\osc_{B_{r}(z)}h=\sup_{z',z''\in B_{r}(z)}\left\vert h(z')-h(z'')\right\vert $
for $B_{r}(z)\subset\Omega$, etc.

For $r\geq1$, we have
\[
\left\vert f(z)\right\vert \leq\const\left(\osc_{B_{r}(z)}h\right)^{1/2}e^{-2mr}.
\]
\end{prop}

\begin{proof}
Note that the estimate for $f^{\text{pb}}$ readily follows from that
for $f$ from conformal mapping and Koebe $1/4$-theorem. By integrating
RHS of (\ref{eq:ccauchy}) on concentric circles and Cauchy-Schwarz,
we have the bound (noting the asymptotics for $K_{0,1}$)
\begin{alignat*}{1}
\left\vert f(z)\right\vert  & \leq\const\frac{\int_{B_{r/2}(z)}\left\vert f\right\vert d^{2}z'}{r^{2}}\leq\const\frac{\left(\int_{B_{r/2}(z)}\left\vert f\right\vert ^{2}d^{2}z'\right)^{1/2}}{r},\\
\left\vert f(z)\right\vert  & \leq\const\frac{e^{-2mr}}{\sqrt{r}}\left(\int_{B_{r/2}(z)}\left\vert f\right\vert ^{2}d^{2}z'\right)^{1/2}\text{ for }r\geq1.
\end{alignat*}
therefore it suffices to bound $\left\Vert f\right\Vert _{L^{2}(B_{r/2}(z))}^{2}\leq\const r\osc_{B_{r}(z)}h$.
In turn it suffices to show $\left\Vert \nabla h\right\Vert _{L^{1}\left(B_{r/2}(z)\right)}\leq\const r\osc_{B_{r}(z)}h$
for a smooth superharmonic function $h$ in $\overline{B_{r}(z)}$.
This can be shown as in \cite[Theorem 3.12]{chsm2012}: we recap the
steps in Lemma \ref{lem:supsubreg}, which concerns the discrete case
but is fully analogous.
\end{proof}
On the boundary, we need to study the continuum variant of the discrete
condition $(RH)$. On rough domains the formulation $(RH)_{F}$ does
not have an obvious continuum analogue due to its use of the tangent
vector $\nu_{\text{tan}}$; however the following analogue of $(RH)_{H}$
is applicable on arbitrary simply connected domains.
\begin{defn}
\label{def:crhh}Suppose a boundary segment $\Sigma\subset\partial\Omega$
is designated to be either free (set $\upsilon:=1$) or wired ($\upsilon:=-1$).
We say that a massive holomorphic function $f$ on $\Omega$ has \emph{continuous
Riemann-Hilbert boundary values}, if $h$ satisfies the condition
$(rh)_{h}$, defined by:
\begin{itemize}
\item $(rh)_{h}$: the square integral $h:=\imm\int f^{2}dz$ on $\Omega$
extends continuously to $\Sigma$ and is a constant there (say, set
to $0$), and there is a sequence in $\Omega$ converging to any given
point in $\varphi\left(S\right)$ along which $\upsilon h\geq0$.
\end{itemize}
Since $h^{\text{pb}}=h\circ\varphi$, $(rh)_{h}$ may be equivalently
stated on any pullback domain $D$ with $S=\varphi^{-1}(\Sigma)\subset\partial D$
as
\begin{itemize}
\item $(rh)_{h}^{\text{pb}}$: the square integral $h^{\text{pb}}:=\imm\int\left(f^{\text{pb}}\right)^{2}dz$
on $D$ extends continuously to $S$ and is a constant there (say,
set to $0$), and there is a sequence in $D$ converging to any given
point in $S$ along which $\upsilon h\geq0$.
\end{itemize}
\end{defn}

When $m=0$, by the classical result going back to Kellogg \cite{kellogg}
the harmonic function $h^{\text{pb}}$, and its derivative $f^{\text{pb}}$,
in fact extend smoothly to $S$. This allows for a definition of a
continuous version of $(RH)_{F}$ equivalent to $(rh)_{h}$ using
the pullback $f^{\text{pb}}$ (which, in the case where $f$ is the
critical $2$- or $4$-point observable, simply corresponds to the
same physical observable defined on the pullback domain; see Section
\ref{subsec:contobs_intro}). For the general case where $m$ may
be nonzero, we define $(rh)_{f}$ as the following \emph{sufficient}
condition for $(rh)_{h}$. Recall Definition \ref{def:holpart}, the
decomposition $f^{\text{pb}}=e^{s}\underline{f}^{\text{pb}}$ for
$s\in W^{1,2}\left(D\right)$ and $\underline{f}^{\text{pb}}$ holomorphic.
\begin{prop}
\label{prop:crhf}Recall that there is a unit tangent vector $\nu_{\text{tan}}$
on any point on $\partial D$. Suppose on a segment $\varphi^{-1}(\Sigma)=S\subset\partial D$,
again designated to be either free ($\upsilon:=1$) or wired ($\upsilon:=-1$),
a massive holomorphic function $f$ satisfies the condition $(rh)_{f}$,
meaning that any (thus all) smooth pullback $f^{\text{pb}}$ satisfies
the following
\[
(rh)_{f}^{\text{pb}}:\underline{f}^{\text{pb}}\text{ extends smoothly to }S\text{ and is in }\sqrt{\frac{\upsilon}{\nu_{\text{tan}}}}\mathbb{R}\text{ along }S.
\]
Then it has Riemann-Hilbert boundary values on $S$, i.e. the square
integral $h$ satisfies $(rh)_{h}$.
\end{prop}

\begin{proof}
This is a combination of Propositions \ref{prop:rhh1} and \ref{prop:rhh2}
with $g=\underline{f}^{\text{pb}}$.
\end{proof}

\section{Discrete Regularity Theory\label{sec:Discrete-Regularity-Theory}}

In this section, we develop regularity theory for massive s-holomorphic
functions. By regularity in discrete setting, we mean the well-behavedness
of a function's or its discrete derivatives' values uniform in small
$\delta$. We will accordingly write in this section $A\apprle B$
(and similarly $A\apprge B$, etc.) when there is $\const(\eta)>0$
only depending on the uniform angle bound $\eta$ such that $A\leq\const(\eta) B$.

To begin, we recall here fundamental notions of discrete complex analysis
on isoradial graphs, with \emph{boundary modification} \cite{chsm2011,chsm2012}.
We state it for the primal $\Gamma\left(\Omega^{\delta}\right)$,
but doing the same on any subdomain or the (also isoradial) dual is
straightforward. On $\Omega^{\delta}$, we define $\Gamma\left(\Omega^{\delta}\right)$
as the set of primal vertices which have all of their neighbours in
$\Omega^{\delta}$. The rest (i.e. on a wired arc or on the outer
half of the boundary rhombus bisected by the free arc) belong to the
boundary $\partial\Gamma\left(\Omega^{\delta}\right)$. We count the
vertices $\underline{z}_{\text{ext}}\in\partial\Gamma\left(\Omega^{\delta}\right)$
accessed across distinct edges $\underline{z}=\left(z_{\text{int}}z_{\text{ext}}\right)^{*}$
(see Figure \ref{fig:intro_grid}R) as distinct regardless of their
physical locations. Write $\overline{\Gamma}(\Omega^{\delta}):=\Gamma\left(\Omega^{\delta}\right)\cup\partial\Gamma\left(\Omega^{\delta}\right)$.
We define the discrete Laplacian for any function $H_{0}$ defined
on a point $u$ in $\Gamma\left(\Omega^{\delta}\right)$ and its neighbours,
\begin{equation}
\left[\Delta^{\delta}H_{0}\right](u):=\frac{1}{\mu_{\Gamma}^{\delta}\left(u\right)}\sum_{z\sim u}t_{z}\cdot\left[H_{0}(u_{z})-H_{0}(u)\right],\label{eq:lap}
\end{equation}
where $\mu_{\Gamma}^{\delta}(u)=\frac{\delta^{2}}{2}\sum_{z\sim u}2\sin\bar{\theta}_{z}$,
and $t_{z}:=\tan\bar{\theta}_{z}$ if $z$ is an interior edge of
$\Omega^{\delta}$ and $t_{\underline{z}}:=\frac{2\sin\bar{\theta}_{\underline{z}}}{\cos\bar{\theta}_{\underline{z}}+1}$
if $\underline{z}$ is on a boundary wired arc of $\Omega^{\delta}$
(in the case of the dual $\Gamma^{*}$, which is itself isoradial,
the weight on the \emph{boundary free arc} simply has $\sin$ and $\cos$ switched).

In the interior, the weight $\mu_{\Gamma}^{\delta}(u)$ corresponds
to half of the sum of the area $\mu_{\lozenge}^{\delta}(z):=\delta^{2}\sin2\overline{\theta}_{z}$
of the rhombi $z$ incident to $u$. Similarly, define $\mu_{\Upsilon}^{\delta}(\xi):=\frac{1}{4}\left(\mu_{\lozenge}^{\delta}(z_{1})+\mu_{\lozenge}^{\delta}(z_{2})\right)$
as the weight of the corner $\xi$ bordering on rhombi $z_{1},z_{2}$.
By\emph{ area integral} of a discrete function (say $H_{0}$), we
mean expressions of type $\sum\mu_{\Gamma}^{\delta}(u)H_{0}(u)$,
using $\mu^{\delta}$ as the natural area element. Accordingly, we
may also define \emph{discrete $L^{p}$-norms} in this way. Note that the
uniform angle bound implies that any $\mu^{\delta}\asymp\delta^{2}$.
As in continuum, we factorise the discrete Laplacian into two \emph{discrete
Wirtinger derivatives} $\partial^{\delta},\bar{\partial}^{\delta}$,
respectively defined for functions $H_{0},F_{0}$ on $\Gamma,\lozenge$
(for definitions of both operators on both lattices, see, e.g. \cite{chsm2011}):
\begin{alignat}{1}
\left[\partial^{\delta}H_{0}\right](z) & :=\frac{H_{0}(v_{1})-H_{0}(v_{2})}{v_{1}-v_{2}}\text{ for }z\in\lozenge,\nonumber \\
\left[\bar{\partial}^{\delta}F_{0}\right](u) & :=-\frac{i}{2\mu_{\Gamma}^{\delta}(u)}\sum_{z_{s}\sim u}\left(w_{s+1}-w_{s}\right)F_{0}(z)\text{ for }u\in\Gamma.\label{eq:dwrit_def}
\end{alignat}
We have $\Delta^{\delta}=4\bar{\partial}^{\delta}\partial^{\delta}$,
and therefore the $\partial^{\delta}$-derivative of a discrete harmonic
function on $\Gamma$ (with respect to $\Delta^{\delta}$) is holomorphic
on $\lozenge$ with respect to $\bar{\partial}^{\delta}$. This factorisation
may be generalised to the massive setting and to the (modified) boundary:
see \cite{dt}. We however do not follow this perspective, instead studying
the 'square integrals' (and not integrals) on $\Gamma$ of massive
s-holomorphic functions on $\lozenge$.

On the boundary of $\Omega^{\delta}$ (and not any other subdomain),
we have introduced a boundary modification to the discrete laplacian
operator, dating back to \cite{chsm2012} at criticality (see also
\cite{dt} for an earlier application to the massive setting). Our
motivation is that this is exactly undoing the boundary length modification
in Lemma \ref{lem:dbvp}, which allowed us to define the locally constant
boundary condition $(RH)_{H}$. This corresponding modification of
the Laplacian enables us to use the crucial estimate (\ref{eq:hlap_primal})
on the boundary as well (since the correspondence (\ref{eq:laph_zthreefsq})
holds). The coefficients (also called \emph{conductances}) are modified
in terms of a uniformly bounded factor only on the boundary: therefore,
these only affect the estimates (say, \cite[Proposition 2.11]{chsm2011})
corresponding to the behaviour of a simple random walk within the
domain by a bounded factor as well.

The discrete Green's formula (\cite[(2.4)]{chsm2011}, simply verified
by summation by parts) states that, given two functions $H_{1,2}$
on $\overline{\Gamma}\left(\Omega^{\delta}\right)$, we have
\begin{align}\label{eq:discgreens}
&\sum_{u\in\Gamma(\Omega^{\delta})}\left[H_{1}\Delta^{\delta}H_{2}-H_{2}\Delta^{\delta}H_{1}\right](u)\mu_{\Gamma}^{\delta}\left(u\right)\\
&=\sum_{z_{\text{ext}}\in\partial\Gamma(\Omega^{\delta})}t_{z}\cdot\left[H_{1}(z_{\text{int}})H_{2}(z_{\text{ext}})-H_{2}(z_{\text{int}})H_{1}(z_{\text{ext}})\right].\nonumber
\end{align}
Thanks to the formula, we may reconstruct, as in continuum, the solution
of a discrete Poisson equation with zero boundary values using the\emph{
discrete Green's function} $G_{\Gamma\left(\Omega^{\delta}\right)}^{\delta}:\overline{\Gamma\left(\Omega^{\delta}\right)}^{2}\to\mathbb{R}_{\leq0}$.
It is the solution of ($\delta(\cdot,\cdot)$ being the Kronecker
delta):
\[
\text{for every }u\in\Gamma\left(\Omega^{\delta}\right)\text{, }\begin{cases}
\Delta^{\delta}G_{\Gamma\left(\Omega^{\delta}\right)}^{\delta}(\cdot,u)=\delta(\cdot,u) & \text{in }\Gamma\left(B^{\delta}\right),\\
G_{\Gamma\left(\Omega^{\delta}\right)}^{\delta}(\cdot,u)=0 & \text{on }\partial\Gamma\left(B^{\delta}\right).
\end{cases}
\]
This function is symmetric in its two variables. Let us take note
of the following elementary bounds: first, for $q\in\left[1,\infty\right)$,
\begin{equation}
\forall u\in\Gamma\left(\Omega^{\delta}\right)\text{, }\left(\sum_{u'\in\Gamma\left(\Omega^{\delta}\right)}\mu_{\Gamma}^{\delta}(u')\left\vert G_{\Gamma\left(\Omega^{\delta}\right)}^{\delta}(u',u)\right\vert ^{q}\right)^{1/q}\apprle\left(\frac{q}{e}\right)^{q}\left(\diam\Omega\right)^{2},\label{eq:gdiameter}
\end{equation}
which may simply be derived on the discrete ball $B_{\diam\Omega}^{\delta}(u)$
by comparison: there it follows easily from the pointwise estimate
$G_{\Gamma\left(B_{\diam\Omega}^{\delta}(u)\right)}^{\delta}(u',u)=\frac{1}{2\pi}\log\frac{\left\vert u-u'\right\vert }{\diam(\Omega)}+O(1)$
from full-plane Green's function estimates (e.g. \cite[(2.5)]{chsm2011},
see also \cite[Lemma A.8]{chsm2012}).

In the special case of the rectangle $\Omega=R:=\left(0,1\right)+(0,\rho i)$
for $\rho>0$, we have the following stronger estimate:
\begin{equation}
\sum_{u'\in\Gamma\left(R^{\delta}\right)}\mu_{\Gamma}^{\delta}(u')\left\vert G_{\Gamma\left(R^{\delta}\right)}^{\delta}(u',u)\right\vert \apprle\left(1+\rho\right)\dist\left(u,\partial R\right),\label{eq:gboundary}
\end{equation}
derived in a similar manner. Say, suppose $u$ is close to the real
line so that $\dist\left(u,\partial R\right)=\imm u$. By comparison,
we have $\left\vert G_{\Gamma\left(R^{\delta}\right)}^{\delta}\right\vert \leq\left\vert G_{\Gamma\left(\mathbb{H}^{\delta}\right)}^{\delta}\right\vert $,
where we have the continuous Green's function $G_{\mathbb{H}}(u',u)=\frac{1}{2\pi}\log\left\vert \frac{u'-u}{u'-\bar{u}}\right\vert $
by reflection of the full-plane Green's function $\frac{1}{2\pi}\log\left\vert u'-u\right\vert $.
Mimicking this construction with the discrete full-plane Green's function,
we derive straightforwardly $\left\vert G_{\Gamma\left(\mathbb{H}^{\delta}\right)}^{\delta}\left(u',u\right)\right\vert \apprle\frac{1}{2\pi}\log\left\vert \frac{u'-\bar{u}}{u'-u}\right\vert +\delta+\frac{\delta^{2}}{\left\vert u'-u\right\vert ^{2}}$,
from which (\ref{eq:gboundary}) follows by radial integration in
$\delta\apprle\left\vert u'-u\right\vert \apprle1+\rho$.

For more properties of $G_{\Gamma\left(\Omega^{\delta}\right)}^{\delta}$,
we refer to \cite{chsm2011}.

\subsection{Pointwise Properties of $H$}

Suppose a massive s-holomorphic function $F$ is given on (some neighbourhood
of) $\lozenge\cup\Upsilon$. We note the properties of $H=\int^{\delta}F^{2}dz$
which allow us to work with it in similar ways as the continuous integral.
The following is easily seen to be valid in the presence of the boundary
modification as well.
\begin{prop}
\label{prop:core-discrete}The Laplacian of $H$ satisfies
\begin{alignat}{1}
\left[\Delta^{\delta}H\vert_{\Gamma}\right](u) & \apprge-m\sum_{\xi \sim u}\left\vert F(\xi)\right\vert ^{2}\text{ on }\Gamma,\label{eq:hlap_primal}\\
\left[\Delta^{\delta}H\vert_{\Gamma^{*}}\right](w) & \apprle m\sum_{\xi \sim w}\left\vert F(\xi)\right\vert ^{2}\text{ on }\Gamma^{*},\label{eq:hlap_dual}
\end{alignat}
where the sum is over the corners $\xi=\xi_z$ incident to $u, w$ respectively (see Fig.~\ref{fig:app}L).

Around any $u\in\Gamma,w\in\Gamma^{*}$, again over incident edges $z$ and corners $\xi$,
\begin{alignat}{1}
\sum_{z\sim u}\left\vert \partial^{\delta}H\vert_{\Gamma}(z)\right\vert  & \asymp\sum_{z\sim u}\left\vert F(z)\right\vert ^{2}\asymp\sum_{\xi\sim u}\left\vert F(\xi)\right\vert ^{2},\label{eq:delh_fsq}\\
\sum_{z\sim w}\left\vert \partial^{\delta}H\vert_{\Gamma^{*}}(z)\right\vert  & \asymp\sum_{z\sim w}\left\vert F(z)\right\vert ^{2}\asymp\sum_{\xi\sim w}\left\vert F(\xi)\right\vert ^{2}.\nonumber 
\end{alignat}

As a consequence, the $L^{p}$ norm of $\partial^{\delta}H\vert_{\Gamma}$
or $\partial^{\delta}H\vert_{\Gamma^{*}}$ is uniformly comparable
to that of $F^{2}$ with constant depending only on $p\geq1,\eta$.
\end{prop}

\begin{proof}
By (\ref{eq:inth}), Lemma \ref{lem:dbvp}, and (\ref{eq:lap}), we
have (in the bulk or near $\partial\Gamma\left(\Omega^{\delta}\right)$)
\begin{equation}
\left[\Delta^{\delta}H\right](u)=\frac{-1}{\mu_{\Gamma}^{\delta}\left(u\right)}\sum_{z\sim u}\frac{\cos\hat{\theta}_{z}}{\cos\bar{\theta}_{z}}\text{Re}\left[\left(w_{z}-w\right)\cdot F(z)^{2}\right]=:-\frac{i}{2}\re\left[\bar{\partial}_{3}^{\delta}F^{2}\right](u),\label{eq:laph_zthreefsq}
\end{equation}
which, combined with Proposition \ref{prop:fsq_zbar}, gives (\ref{eq:hlap_primal}).
The Laplacian on $w$ may be calculated from duality (cf. Remark \ref{rem:fsq_duality}),
under which $H$ and the first term in (\ref{eq:zbarthree_fsq}) changes
sign but the second does not, flipping the direction of the inequality.

Now we will show the primal estimate (\ref{eq:delh_fsq}), from which the dual estimate in the second line follows again by duality. Given estimates
of $F(z)$ from corner values (\ref{eq:refsq}), we may easily show,
by noting that $\cos\hat{\theta},\sin\hat{\theta}$ are uniformly
bounded away from $0$ and $1$ for small enough $q$, all but the
following direction of (\ref{eq:delh_fsq}):
\[
\sum_{\xi \sim u}\left\vert F(\xi)\right\vert ^{2}\apprle\sum_{z\sim u}\left\vert \partial^{\delta}H\vert_{\Gamma}(z)\right\vert .
\]
If in (\ref{eq:refsq}) $z_{0}\sim u$ is any edge with $\left[\da_{u}F\right](z_{0})=-1$,
\[
\frac{\cos\hat{\theta}_{z_{0}}}{\sin^{2}\hat{\theta}_{z_{0}}}\left(\left\vert F(\xi_{z_{0}})\right\vert ^{2}+\left\vert F(\xi^{z_{0}})\right\vert ^{2}\right)\leq-\re\left[F(z)^{2}\nu_{T}(z)\right]\leq\left\vert \partial^{\delta}H\vert_{\Gamma}(z_{0})\right\vert 
\]
which gives a local bound near $z_{0}$. For the other edges $z_{1}\sim u$
with $\left[\da_{u}F\right](z_{1})=1$, we have, say,
\begin{align*}
&\re\left[F(z_{1})^{2}\nu_{T}(z_{1})\right]=\\
&\frac{\left(\frac{1}{\cos\hat{\theta}_{z_{1}}}-\cos\hat{\theta}_{z_{1}}\right)\left\vert F(\xi^{z_{1}})\right\vert ^{2}-\left(\sqrt{\cos\hat{\theta}_{z_{1}}}\left\vert F(\xi_{z_{1}})\right\vert -\frac{1}{\sqrt{\cos\hat{\theta}_{z_{1}}}}\left\vert F(\xi^{z_{1}})\right\vert \right)^{2}}{\sin^{2}\hat{\theta}_{z_{1}}},
\end{align*}
which uniformly bounds the difference between $\left\vert F(\xi_{z_{1}})\right\vert $
and $\left\vert F(\xi^{z_{1}})\right\vert $ by $\left\vert F(\xi^{z_{1}})\right\vert $
and $\left\vert \partial^{\delta}H\vert_{\Gamma}(z_{1})\right\vert ^{1/2}$. Dividing
corners adjacent to $u$ into the above two types, there is always
at least one corner $z_{0}\sim u$ for which we may use the former
estimate; from there we may bound any other corner value using the
latter estimate.

In other words, the upper bound for any corner value $\vert F(\xi)\vert ^2$ may be obtained by adding a uniformly
bounded number (since there are a uniformly bounded number of corners by the angle bound) of $\left\vert \partial^{\delta}H\vert_{\Gamma}(z)\right\vert $ for edges $z\sim u$, and this can be repeated (again, uniformly bounded number of times) for $\sum_{\xi \sim u}\left\vert F(\xi)\right\vert ^{2}$, giving the remaining direction.
\end{proof}
\begin{rem}
Unlike in \cite{par19}, which exploited the $L^{2}$-bound coming
from the first term in the Laplacian, we do not use (or show) the
sign-definiteness of $A$ in our analysis. Instead, we rely on the
discrete maximum/minimum principles, as well as the domination of
the Laplacian (Lemma \ref{lem:Nthpower}). On the other hand, $A_{k}$
must tend to a strictly positive quantity by comparing \emph{a posteriori}
with the continuous Laplacian; this may also be shown purely from
a more careful discrete derivation, see \cite[Proposition 3.8]{cim21}.
\end{rem}

Note that in continuum we have $\Delta u^{p}=p(p-1)u^{p-2}\left\vert \nabla u\right\vert ^{2}+pu^{p-1}\Delta u$.
Our crucial lemma below uses its discrete counterpart to control the
possibly negative Laplacian $\asymp\vert F\vert ^{2}$ using the gradient squared
$\asymp\left\vert F\right\vert ^{4}$.
\begin{lem}
\label{lem:Nthpower}For any real nonnegative square integral $H$
defined on $u\in\Gamma$ and its neighbours in $\Gamma$,
\begin{alignat*}{1}
\left[\Delta^{\delta}H\vert_{\Gamma}^{2}\right](u) & \apprge\sum_{z\sim u}\left\vert F(z)\right\vert ^{4}-mH\vert_{\Gamma}(u)\sum_{z\sim u}\left\vert F(z)\right\vert ^{2},\\
\left[\Delta^{\delta}H\vert_{\Gamma}^{N}\right](u) & \apprge-m^{2}H\vert_{\Gamma}^{\min}(u)^{N},
\end{alignat*}
for any $N\geq2$, where $H_{0}^{\text{min}}\left(u\right):=\min\left(\left\{ H_{0}(u_{z})\right\} _{u_z\sim u},H_{0}(u)\right)\geq0$
is the minimum of $H_{0}$ among $u$ and its neighbours in $\Gamma$.
On $\Gamma^{*}$, we have the same estimate of $-H\vert_{\Gamma^{*}}\geq0$
in place of $H\vert_{\Gamma}$.
\end{lem}

\begin{proof}
Using basic algebra, one may verify that, for any real nonnegative
function $H_{0}$ defined on $u\in\Gamma$ and its neighbours, we
have
\[
\left[\Delta^{\delta}\left(H_{0}\right)^{N}\right](u)\geq\frac{N(N-1)H_{0}^{\text{min}}\left(u\right)^{N-2}}{\mu_{\Gamma}^{\delta}(u)}\sum_{z\sim u}\mu_{\lozenge}^{\delta}(z)\left\vert \partial^{\delta}H_{0}(z)\right\vert ^{2}+NH_{0}(u)^{N-1}\Delta^{\delta}H_{0}(u),
\]
which directly gives the first estimate, noting that all $\mu\asymp\delta^{2}$
and $\Delta^{\delta}H(u)\apprge-m\sum_{z\sim u}\left\vert F(z)\right\vert ^{2}$
by Proposition \ref{prop:core-discrete}. Then using Proposition \ref{prop:core-discrete},
we have (using $\mu\asymp\delta^{2}$ to simplify)
\begin{alignat*}{1}
\left[\Delta^{\delta}H^{N}\right](u) & \apprge NH^{\text{min}}\left(u\right)^{N-2}\left((N-1)\left(\sum_{z\sim u}\left\vert F(z)\right\vert ^{2}\right)^{2}-mH^{\min}(u)\sum_{z\sim u}\left\vert F(z)\right\vert ^{2}\right)\\
 & \geq-\frac{Nm^{2}H^{\min}(u)^{N}}{4\left(N-1\right)}\geq-\frac{m^{2}H^{\text{min}}(u)^{N}}{2},
\end{alignat*}
minimising the quadratic in $\sum_{z\sim u}\left\vert F(z)\right\vert ^{2}$.
\end{proof}
We now show the strong maximum/minimum principles for $H$. While
we do not prove or use a discrete analogue to Lemma \ref{lem:comparison}
in this paper, let us note that there is a discrete comparison principle
\cite[Proposition 2.11]{s-emb} which may be adapted to the massive
setting.
\begin{prop}
\label{prop:dmax}Suppose $H$ achieves a global maximum or minimum
on an interior $z\in\Lambda$. Then $H$ is constant.
\end{prop}

\begin{proof}
We will carry out the maximum case; the minimum case is exactly analogous.

Consider $\Lambda_{M}:=\left\{ z\in\Lambda:H^{\delta}(z)=\max H^{\delta}\right\} $.
Note that
\begin{itemize}
\item $F$ is zero on any edge between adjacent $u_{1},u_{2}\in\Gamma\cap\Lambda_{M}$;
\item any $u\in\Gamma$ adjacent to a $w\in\Gamma^{*}\cap\Lambda_{M}$ is
also in $\Lambda_{M}$.
\end{itemize}
Then it is easy to see that $\Lambda\setminus\Lambda_{M}$ must be
empty unless $\Lambda_{M}$ consists of isolated points $u\in\Gamma$.
This means that $u\in\Lambda_{M}$ and $F^{\delta}$ is nonzero on
any corners around $u$. But this is impossible: in \eqref{eq:refsq}
choose $z\sim u$ such that $\left[\da_{u}F\right](z)=-1$, then $H(u_{z})-H(u)=-\re\left[2\cos\hat{\theta}_{z}\nu_{T}F^{2}(z)\right]>0$.
\end{proof}

\subsection{Bulk Estimates}

First, we give some lemmas. The first is a re-cap of the steps 1-2
in the proof of \cite[Theorem 3.12]{chsm2012}, which deduces interior
derivative $L^{1}$-bound of a sub- or superharmonic function from
its oscillation. Define the discrete ball $\Gamma\left(B_{r}^{\delta}\right)$
as the largest simply connected subset of $\Gamma$ containing $\Gamma\cap B_{r}$.
\begin{lem}[{Part of \cite[Theorem 3.12]{chsm2012}}]
\label{lem:supsubreg}Suppose $H_{0}$ is either a sub- or superharmonic
function on $\Gamma\left(B_{r}^{\delta}\right)$. Then for a universal
constant $\const>0$,
\[
\sum_{z\in\lozenge(B_{r/2}^{\delta})}\mu_{\lozenge}^{\delta}(z)\left\vert \partial^{\delta}H_{0}(z)\right\vert \apprle r\osc_{\Gamma\left(B_{r}^{\delta}\right)}H_{0},
\]
as long as $r\geq\const\delta$ for a universal constant $\const>0$.
\end{lem}

\begin{proof}
The constant $\const$ is determined precisely so that the balls we
take in the following (and the lemmas cited) are nonempty: we use
a finite number of balls whose radii are explicit multiples of $r$.

Before we split $H_{0}$ as usual into superharmonic and harmonic
parts on $B_{3r/4}^{\delta}$, we need to first consider the superharmonic
part on $B_{r}^{\delta}$.
\begin{enumerate}
\item The superharmonic part $\sum_{u\in\Gamma\left(B_{r}^{\delta}\right)}\mu_{\Gamma}^{\delta}(u)G_{\Gamma\left(B_{r}^{\delta}\right)}^{\delta}(\cdot,u)\Delta^{\delta}H_{0}(u)\geq0$
on $\Gamma\left(B_{r}^{\delta}\right)$:
\begin{enumerate}
\item is bounded above by $2\osc_{\Gamma\left(B_{r}^{\delta}\right)}H_{0}$;
\item $\sum_{u\in B_{r}^{\delta}}\left\Vert G_{\Gamma\left(B_{r}^{\delta}\right)}^{\delta}(\cdot,u)\right\Vert _{L^{1}\left(B_{r}\right)}\left\vert \Delta^{\delta}H_{0}(u)\right\vert \apprle r^{2}\osc_{\Gamma\left(B_{r}^{\delta}\right)}H_{0}$;
\item for $u\in\Gamma\left(B_{3r/4}^{\delta}\right)$, $\left\Vert G_{\Gamma\left(B_{r}^{\delta}\right)}^{\delta}(\cdot,u)\right\Vert _{L^{1}\left(\Gamma\left(B_{r}^{\delta}\right)\right)}\apprge r^{2}$
\cite[Lemma A.8]{chsm2012}, so 
\[
\left\Vert \Delta^{\delta}H_{0}\right\Vert _{L^{1}\left(\Gamma\left(B_{3r/4}^{\delta}\right)\right)}\leq\const\osc_{\Gamma\left(B_{r}^{\delta}\right)}H_{0}.
\]
\end{enumerate}
\item The superharmonic part $H_{0}^{\superh}:=\sum_{u\in\Gamma\left(B_{3r/4}^{\delta}\right)}\mu_{\Gamma}^{\delta}(u)G_{\Gamma\left(B_{3r/4}^{\delta}\right)}^{\delta}(\cdot,u)\Delta H_{0}(u)$
on $\Gamma\left(B_{3r/4}^{\delta}\right)$: by \cite[Lemma A.9]{chsm2012},
\begin{alignat*}{1}
\left\Vert \partial^{\delta}H_{0}^{\superh}\right\Vert _{L^{1}\left(\Gamma\left(B_{r/2}^{\delta}\right)\right)} & \leq\sum_{u\in\Gamma\left(B_{3r/4}^{\delta}\right)}\mu_{\Gamma}^{\delta}(u)\left\Vert \partial^{\delta}G_{B_{r}^\delta}(\cdot,u)\right\Vert _{L^{1}\left(\Gamma\left(B_{r/2}^{\delta}\right)\right)}\left\vert \Delta^{\delta}H_{0}^{\superh}(u)\right\vert \\
 & \apprle\sum_{u\in\Gamma\left(B_{3r/4}^{\delta}\right)}\mu_{\Gamma}^{\delta}(u)\cdot r\left\vert \Delta^{\delta}H_{0}(u)\right\vert \apprle r\osc_{\Gamma\left(B_{r}^{\delta}\right)}H_{0}.
\end{alignat*}
\item The harmonic part $H_{0}^{\harm}=H_{0}-H_{0}^{\superh}$ on $B_{3r/4}$:
\begin{enumerate}
\item has oscillation at most $\osc_{\Gamma\left(B_{r}^{\delta}\right)}H_{0}$;
\item by Harnack inequality \cite[Proposition 2.7]{chsm2011}, $\left\vert \partial^{\delta}H_{0}^{\harm}(z')\right\vert \apprle\frac{\osc_{\Gamma\left(B_{r}^{\delta}\right)}H_{0}}{r-\left\vert z'\right\vert }$,
and therefore 
\[
\left\Vert \partial^{\delta}H_{0}^{\harm}\right\Vert _{L^{1}\left(\Gamma\left(B_{r/2}^{\delta}\right)\right)}\apprle r\osc_{\Gamma\left(B_{r}^{\delta}\right)}H_{0}.
\]
\end{enumerate}
\end{enumerate}
\end{proof}
The bulk estimate in the discrete case is similar to the continuous
case: we bound a massive s-holomorphic function $F$ and its discrete
derivative using massive Cauchy formula, which we can bound by the
oscillation of $H=\imm\int F^{2}dz$.
\begin{prop}
\label{prop:dreg-bulk}For any massive s-holomorphic function $F$
and $z\in\lozenge\left(\Omega^{\delta}\right)$,
\begin{align}
\left\vert F(z)\right\vert   \apprle\frac{\left(\left(1+m^{2}d^{2}\right)\osc_{\Gamma\left(B_{d/2}^{\delta}(z)\right)}H\right)^{1/2}}{d^{1/2}},\label{eq:dbulk}
\end{align}
where $z\sim z^\delta \in\lozenge\left(\Omega^{\delta}\right)$ and $d=\dist(z,\partial\Omega^{\delta})\geq2\const\delta$
as in Lemma \ref{lem:supsubreg}. If $m$ and $d$ are held uniformly away from $0$ and $\infty$, we have the 'discrete derivative' estimate
\begin{equation*}
\left\vert F(z)-F(z^\delta)\right\vert   \apprle\frac{\delta\left( \osc_{\Gamma\left(B_{d/2}^{\delta}(z)\right)} H\right)^{1/2}}{d^{3/2}}.
\end{equation*}
As usual, $\osc_{\Gamma\left(B_{d/2}^{\delta}(z)\right)}H$
may be replaced by $\osc_{\Gamma^{*}\left(B_{d/2}^{\delta}(z)\right)}H$
through duality in both estimates.

In addition, we have the following bound of $\osc_{\Lambda\left(B_{d/2}^{\delta}(z)\right)}H$:
\begin{alignat}{1}
\max_{\Gamma\left(B_{d/2}^{\delta}(z)\right)}H & \apprle\frac{\left(1+m^{2}d^{2}\right)^{2}}{d^{2}}\sum_{u\in\Gamma\left(B_{d}^{\delta}(z)\right)}\mu_{\Gamma}^{\delta}(u)H(u),\label{eq:maxbound}\\
\min_{\Gamma^{*}\left(B_{d/2}^{\delta}(z)\right)}H & \apprge\frac{\left(1+m^{2}d^{2}\right)^{2}}{d^{2}}\sum_{u\in\Gamma^{*}\left(B_{d}^{\delta}(z)\right)}\mu_{\Gamma^{*}}^{\delta}(u)H(u).\nonumber 
\end{alignat}
\end{prop}

\begin{proof}
As in \cite[Theorem 3.12]{chsm2012}, similarly to Proposition \ref{prop:creg-bulk},
we may use the Cauchy formula (Proposition \ref{prop:dcauchy_formula})
and the kernel estimates (Proposition \ref{prop:dcauchy_est}) to
bound:
\begin{alignat*}{1}
\left\vert F(z)\right\vert  & \apprle\frac{\left(\sum_{z'\in\lozenge\left(B_{d/4}^{\delta}(z)\right)}\mu_{\lozenge}^{\delta}(z')\left\vert F\right\vert ^{2}(z')\right)^{1/2}}{d},\\
\left\vert F(z)-F(z^\delta)\right\vert  & \apprle\frac{\delta\left(\sum_{z'\in\lozenge\left(B_{d/4}^{\delta}(z)\right)}\mu_{\lozenge}^{\delta}(z')\left\vert F\right\vert ^{2}(z')\right)^{1/2}}{d^{2}},
\end{alignat*}
where we convert contour integrals to an $L^{1}$ norm by averaging
over concentric discrete circles, and Cauchy-Schwarz to move to $L^{2}$
norm. Therefore it suffices to bound $\sum_{z'\in\lozenge\left(B_{d/4}^{\delta}(z)\right)}\mu_{\lozenge}^{\delta}(z')\left\vert F\right\vert ^{2}(z')\apprle\left(\left(1+m^{2}d^{2}\right)\osc_{\Gamma(B_{d/2}^{\delta}(z))}H\right)d$.
Again, as in the proof of Proposition \ref{prop:creg-bulk}, given
(\ref{eq:delh_fsq}), we may simply substitute the LHS with $\sum_{z'\in\lozenge\left(B_{d/4}^{\delta}(z)\right)}\mu_{\lozenge}^{\delta}(z')\left\vert \partial^{\delta}H\vert_{\Gamma}(z')\right\vert. $

Since the estimate is obvious if $H\equiv0$, suppose it is not the
case and assume $\min_{\Gamma\left(B_{d/2}^{\delta}(z)\right)}H=\osc_{\Gamma\left(B_{d/2}^{\delta}(z)\right)}H$
by translation and scaling. Decompose $H\vert_{\Gamma}^{2}=H_{\subh}+H_{\superh}$,
where
\begin{equation}
\begin{cases}
H_{\subh}=H\vert_{\Gamma}^{2} & \text{on }\partial\Gamma\left(B_{d/2}^{\delta}(z)\right);\\
\Delta^{\delta}H_{\subh}=\max\left(\Delta^{\delta}H^{2},0\right) & \text{in }\Gamma\left(B_{d/2}^{\delta}(z)\right),
\end{cases}\begin{cases}
H_{\superh}=0 & \text{on }\partial\Gamma\left(B_{d/2}^{\delta}(z)\right);\\
\Delta^{\delta}H_{\superh}=\min\left(\Delta^{\delta}H^{2},0\right) & \text{in }\Gamma\left(B_{d/2}^{\delta}(z)\right).
\end{cases}\label{eq:subsupsep}
\end{equation}
We will show below the bound 
\begin{equation}
\osc_{\Gamma\left(B_{d/2}^{\delta}(z)\right)}H_{\superh}=\max_{\Gamma\left(B_{d/2}^{\delta}(z)\right)}H_{\superh}\apprle m^{2}d^{2}\max_{\Gamma(B_{d/2}^{\delta}(z))}H^{2}\apprle m^{2}d^{2}\osc_{\Gamma(B_{d/2}^{\delta}(z))}H^{2},\label{eq:suphalf}
\end{equation}
then by maximum principle $\osc_{\Gamma\left(B_{d/2}^{\delta}(z)\right)}H_{\subh}\apprle\left(1+m^{2}d^{2}\right)\osc_{\Gamma(B_{d/2}^{\delta}(z))}H^{2}$.
Therefore, using Lemma \ref{lem:supsubreg}, we have the bound 
\[
\sum_{z'\in\lozenge\left(B_{d/4}^{\delta}(z)\right)}\mu_{\lozenge}^{\delta}(z')\left\vert \partial^{\delta}H^{2}\vert_{\Gamma}(z')\right\vert \apprle d\left(1+m^{2}d^{2}\right)\osc_{\Gamma(B_{d/2}^{\delta}(z))}H^{2},
\]
and since by assumption $\min_{\Gamma\left(B_{d/2}^{\delta}(z)\right)}H\asymp\max_{\Gamma\left(B_{d/2}^{\delta}(z)\right)}H$,
\[
\sum_{z'\in\lozenge\left(B_{d/4}^{\delta}(z)\right)}\mu_{\lozenge}^{\delta}(z')\left\vert \partial^{\delta}H\vert_{\Gamma}(z')\right\vert \apprle d\left(1+m^{2}d^{2}\right)\osc_{\Gamma(B_{d/2}^{\delta}(z))}H.
\]

Now let us show (\ref{eq:suphalf}). $H_{\superh}=\sum_{u\in\Gamma\left(B_{d/2}^{\delta}(z)\right)}\mu_{\Gamma}^{\delta}(u)G_{\Gamma\left(B_{d/2}^{\delta}(z)\right)}^{\delta}(\cdot,u)\Delta^{\delta}H_{\superh}(u)$,
and the Laplacian is bounded below:
\begin{alignat*}{1}
\Delta^{\delta}H^{2} & \apprge-m^{2}\left(H^{\text{min}}\right)^{2}\apprge-m^{2}\osc_{\Gamma(B_{d/2}^{\delta}(z))}H^{2},
\end{alignat*}
applying Lemma \ref{lem:Nthpower}. Multiplying it with (\ref{eq:gdiameter}),
we have (\ref{eq:suphalf}).

We finish by showing (\ref{eq:maxbound}). This time assume $\min_{\Gamma\left(B_{d}^{\delta}(z)\right)}H=\osc_{\Gamma\left(B_{d}^{\delta}(z)\right)}H$,
and re-do the decomposition (\ref{eq:suphalf}) this time on $\Gamma\left(B_{d}^{\delta}(z)\right)$.
Then $H_{\superh}$ may already be bounded on all of $\Gamma\left(B_{d}^{\delta}(z)\right)$:
by Cauchy-Schwarz,
\begin{alignat}{1}
0\leq H_{\superh} & \leq\left(\sum_{u\in\Gamma\left(B_{d}^{\delta}(z)\right)}\mu_{\Gamma}^{\delta}(u)\left(G_{\Gamma\left(B_{d}^{\delta}(z)\right)}^{\delta}(\cdot,u)\right)^{2}\right)^{1/2}\left(\sum_{u\in\Gamma\left(B_{d}^{\delta}(z)\right)}\mu_{\Gamma}^{\delta}(u)\left(\Delta^{\delta}H_{\superh}(u)\right)^{2}\right)^{1/2}\nonumber \\
 & \apprle m^{2}d\left(\sum_{u\in\Gamma\left(B_{d}^{\delta}(z)\right)}\mu_{\Gamma}^{\delta}(u)\left(H(u)\right)^{4}\right)^{1/2},\label{eq:greenssquare}
\end{alignat}
applying (\ref{eq:gdiameter}). The subharmonic part satisfies the
mean value bound (e.g. see \cite[Proposition A.2]{chsm2011}, which
is for discrete harmonic functions but straightforward to modify for
subharmonic functions): for $u'\in\Gamma\left(B_{d/2}^{\delta}(z)\right)$,
\begin{alignat*}{1}
H_{\subh}(u') & \apprle\frac{1}{d^{2}}\sum_{u\in\Gamma\left(B_{d}^{\delta}(z)\right)}\mu_{\Gamma}^{\delta}(u)H_{\subh}(u)\leq\frac{1}{d^{2}}\sum_{u\in\Gamma\left(B_{d}^{\delta}(z)\right)}\mu_{\Gamma}^{\delta}(u)\left(H(u)\right)^{2}\\
 & \leq\frac{1}{d}\text{\ensuremath{\left(\sum_{u\in\Gamma\left(B_{d}^{\delta}(z)\right)}\mu_{\Gamma}^{\delta}(u)\left(H(u)\right)^{4}\right)^{1/2}}},
\end{alignat*}
again by Cauchy-Schwarz. So we have
\[
H(u')^{2}\leq\frac{1}{d}\left(1+m^{2}d^{2}\right)\left(\sum_{u\in\Gamma\left(B_{d}^{\delta}(z)\right)}\mu_{\Gamma}^{\delta}(u)\left(H(u)\right)^{4}\right)^{1/2}.
\]
Noting again $\min_{\Gamma\left(B_{d}^{\delta}(z)\right)}H\asymp\max_{\Gamma\left(B_{d}^{\delta}(z)\right)}H$,
we have the desired bound.
\end{proof}
\begin{rem}
\label{rem:subsequence}Proposition \ref{prop:dreg-bulk} implies
that, once we renormalise any massive holomorphic function $F$ such
that its square integral $H$ is uniformly bounded in a domain (or,
by Proposition \ref{prop:dmax}, simply on the boundary), $F$ and
its discrete derivative is uniformly bounded on any compact subset.
By, say, piecewise linear interpolation, we may apply Arzel\`a-Ascoli
to get a locally Lipschitz limit $f$ (in the sense defined in Section
\ref{subsec:intro_statement}). Taking a sequence of increasing compact
subsets whose union covers the whole domain and diagonalising, we
may assume that $f$ is defined on the whole domain. Then it only
remains to uniquely identify $f$ to finish the proof of the scaling
limit.
\end{rem}

\subsection{Boundary Estimates\label{subsec:Boundary-Estimates}}

On boundary, we provide two a priori estimates for the discrete function
$H$ which will yield necessary information to fix its limit $h$.
Recall that $(RH)_{H}$ fixes a constant boundary condition for $H$;
we show that uniformly bounded $H$ satisfies it with a uniform (in
$\delta$) modulus of continuity, so that any continuous limit of
$H$ inherits the same continuity up to boundary. The key idea again
is to use Lemma \ref{lem:Nthpower} to control the Laplacian, as in
(\ref{eq:suphalf}); see \cite[Remark 4.3]{s-emb} for a possible
alternative strategy.
\begin{prop}
\label{prop:massive-beurling}Suppose $H=\imm\int F^{2}dz$ takes constant
boundary value $H(S^{\delta})$ (in the sense of Lemma \ref{lem:dbvp})
on a discrete boundary segment $S^{\delta}\subset\partial\Omega^{\delta}$.
Then there is a universal exponent $\beta_{0}>0$, such that the following
holds:
\[
\left\vert H(v)-H(S^{\delta})\right\vert \apprle\osc_{\Omega^{\delta}}H\left(1+\left(m\dist\left(v,\partial\Omega^{\delta}\setminus S^{\delta}\right)\right)^{2}\right)\left(\frac{\dist\left(v,S^{\delta}\right)}{\dist\left(v,\partial\Omega^{\delta}\setminus S^{\delta}\right)}\right)^{\beta_{0}}.
\]
\end{prop}

\begin{proof}
This is a generalisation of the so-called weak Beurling estimate for
harmonic functions (e.g. see \cite[Proposition 2.11]{chsm2011});
we use a similar iteration strategy. We will first bound $H(v)$ from
above using its restriction to $\Gamma$
(which bounds it globally due to \eqref{eq:defh}). Since
the estimate is invariant under adding a constant to $H$, let $\min_{\Gamma\left(\Omega^{\delta}\right)}H=\osc_{\Gamma\left(\Omega^{\delta}\right)}H>0$,
and decompose $H\vert_{\Gamma\left(\Omega^{\delta}\right)}^{2}=H_{\subh}+H_{\superh}$
as in (\ref{eq:subsupsep}), but in this case in some connected component
of the boundary neighbourhood $N_{d}^{\delta}:=\Gamma(\Omega^{\delta})\cap B_{d}(v_{0})$
for some $v_{0}\in S$ and $d<\dist\left(v_{0},\partial\Omega^{\delta}\setminus S^{\delta}\right)$
(see Figure \ref{fig:boundary}T).

By the weak Beurling estimate and comparison with harmonic majorant,
there is some universal exponent $\beta_{0}\in(0,1)$ such that $\max_{N_{d/2}^{\delta}}\left(H_{\subh}-H(S^{\delta})^{2}\right)\leq2^{-\beta_{0}}\max_{N_{d}^{\delta}}\left(H\vert_{\Gamma\left(\Omega^{\delta}\right)}^{2}-H(S^{\delta})^{2}\right)$.
On the other hand, as in (\ref{eq:suphalf}), we have $\osc_{N_{d/2}^{\delta}}H_{\superh}\leq\const m^{2}d^{2}\max_{N_{d}^{\delta}}H^{2}$
with $\const$ only depending on the angle bound $\eta$. Starting
from, say, $d_{0}=\frac{1}{2}\dist\left(v,\partial\Omega^{\delta}\setminus S^{\delta}\right)$,
we have the recursive inequality
\begin{align*}
&\max_{N_{2^{-\left(n+1\right)}d_{0}}^{\delta}}\left(H\vert_{\Gamma\left(\Omega^{\delta}\right)}^{2}-H(S^{\delta})^{2}\right)\\
&\leq2^{-\beta_{0}}\max_{N_{2^{-n}d_{0}}^{\delta}}\left(H\vert_{\Gamma\left(\Omega^{\delta}\right)}^{2}-H(S^{\delta})^{2}\right)+\const m^{2}\left(2^{-n}d_{0}\right)^{2}\max_{\Gamma\left(\Omega^{\delta}\right)}H^{2}.
\end{align*}
Rearranging,
\begin{alignat*}{1}
&\max_{N_{2^{-\left(n+1\right)}d_{0}}^{\delta}}\left(H\vert_{\Gamma\left(\Omega^{\delta}\right)}^{2}-H(S^{\delta})^{2}\right)+\frac{\const m^{2}d_{0}^{2}\max_{\Gamma\left(\Omega^{\delta}\right)}H^{2}}{(2^{2-\beta_{0}}-1)4^{n}}\\
& \leq2^{-\beta_{0}}\left[\max_{N_{2^{-n}d_{0}}^{\delta}}\left(H\vert_{\Gamma\left(\Omega^{\delta}\right)}^{2}-H(S^{\delta})^{2}\right)+\frac{\const m^{2}d_{0}^{2}\max_{\Gamma\left(\Omega^{\delta}\right)}H^{2}}{(2^{2-\beta_{0}}-1)4^{(n-1)}}\right],
\end{alignat*}
and we finally have
\[
H\vert_{\Gamma\left(\Omega^{\delta}\right)}^{2}(v)-H(S^{\delta})^{2}\apprle\left(1+\left(m\dist\left(v,\partial\Omega^{\delta}\setminus S^{\delta}\right)\right)^{2}\right)\max_{\Gamma\left(\Omega^{\delta}\right)}H^{2}\left(\frac{\dist\left(v,S^{\delta}\right)}{\dist\left(v,\partial\Omega^{\delta}\setminus S^{\delta}\right)}\right)^{\beta_{0}},
\]
iterating the above decay from $d_{0}$ to $z\in N_{2^{-\left(n+1\right)}d_{0}}^{\delta}$. 

Since $H\geq\frac{1}{2}\max_{\Gamma\left(\Omega^{\delta}\right)}H$
and $\max_{\Gamma\left(\Omega^{\delta}\right)}H=2\osc_{\Omega^{\delta}}H$
by assumption, the upper bound follows. Corresponding bound for the
minimum may be derived analogously on $\Gamma^{*}\left(\Omega^{\delta}\right)$
by replacing $H\vert_{\Gamma\left(\Omega^{\delta}\right)}$ with $-H\vert_{\Gamma^{*}\left(\Omega^{\delta}\right)}$.
\end{proof}
\begin{rem}
\label{rem:massive-beurling-gen}Examining the proof above, it is
easy to make a few generalisations, both resembling the harmonic case:
first, the distance $\dist\left(v,\partial\Omega^{\delta}\setminus S^{\delta}\right)$
may be replaced by $\dist_{\Omega^{\delta}}\left(v,\partial\Omega^{\delta}\setminus S^{\delta}\right)$,
which is defined as the radius of the smallest neighbourhood in $\Omega^{\delta}$
around $v$ in which $v$ and $\partial\Omega^{\delta}\setminus S^{\delta}$
are connected (cf. \cite[Proposition 2.11]{chsm2011}); second, $\beta_{0}$
coincides with the corresponding exponent in the harmonic case, and
therefore may be set to $\beta_{0}=1$ if, e.g., $S^{\delta}$ is
part of a discrete rectangle side (see \cite[Lemma 3.17]{chsm2011}
for the harmonic exponent). This in particular implies, from (\ref{eq:defh}),
$\left\vert F\right\vert ^{2}$ is at most comparable to $\delta^{-1}\osc_{\Omega^{\delta}}H$
near $S^{\delta}$.
\end{rem}

Finally, we show that $(RH)_{H}$ is preserved in the limit by showing
that the remaining component, the sign of the normal derivative, is
preserved in the limit. We use the setup of the corresponding part
in the proof of \cite[Theorem 6.1]{chsm2012}, while utilising the
estimates of Lemma \ref{lem:Nthpower} to mitigate the effect of the
additional Laplacian term.
\begin{prop}
\label{prop:drhh_crhh}Suppose $H$ satisfies $(RH)_{H}$ on a boundary
segment $S^{\delta}\subset\partial\Omega^{\delta}$ tending to $S\subset\partial\Omega$
in Caratheodory sense and has a subsequential limit $h$ on $\Omega$
provided by Remark \ref{rem:subsequence}. Then $h$ satisfies $(rh)_{h}$
on $S$.
\end{prop}

\begin{proof}
Fix a subsequence $\delta_{j}\downarrow0$ (which we henceforth suppress
from notation) along which $H$ converges to $h$ locally uniformly.
Without loss of generality, suppose $\upsilon=-1$ on $S$ and $h(S)=1$.
We only work with $H\vert_{\Gamma\left(\Omega^{\delta}\right)}$ in
this case; for $\upsilon=1$, we apply the same argument to $-H\vert_{\Gamma^{*}\left(\Omega^{\delta}\right)}$.

Suppose by contradiction that there is a crosscut which has its two
endpoints $p_{1},p_{2}$ on $S$ and bounds a subdomain $\underline{\Omega}\subset\Omega$
where, by rescaling, $0<h<1$. In fact, as in the proof of \cite[Theorem 6.1]{chsm2012},
we may suppose that $\underline{\Omega}$ is bounded by $(p_{2}p_{1})\subset\partial\Omega$
and three line segments $C:=\left(p_{1}q_{1}\right]\cup\left[q_{1}q_{2}\right]\cup\left[q_{2}p_{2}\right)$
in $\Omega$ such that $\dist\left(q_{1,2},\partial\Omega\right)=\left\vert p_{1,2}-q_{1,2}\right\vert $.
Choose intermediate intervals $\left[\tilde{p}_{2}\tilde{p}_{1}\right]\subset\left(p_{2}p_{1}\right)$
and $\left[\tilde{q}_{1}\tilde{q}_{2}\right]\subset\left(q_{1}q_{2}\right)$.
Discretise (just on $\Gamma\left(\Omega^{\delta}\right)$) to get
$\underline{\Omega}^{\delta}$, $C^{\delta}$ and marked points $p_{1,2}^{\delta},q_{1,2}^{\delta},\tilde{p}_{1,2}^{\delta},\tilde{q}_{1,2}^{\delta}$
which converge to their respective continuous points (Figure \ref{fig:boundary}T).

Let $V=\omega\left(\cdot,\left[\tilde{p}_{2}^{\delta}\tilde{p}_{1}^{\delta}\right],\partial D^{\delta}\right)$
be the harmonic measure of $\left[\tilde{p}_{2}^{\delta}\tilde{p}_{1}^{\delta}\right]\subset\partial\underline{\Omega}^{\delta}$.
Standard estimates show (see \cite[Fig. 10(B)]{chsm2012}) that there
exists some $\smallc\left(\underline{\Omega}\right)>0$ such that
$V\left(\underline{z}_{\text{int}}\right)\geq\smallc\left(\underline{\Omega}\right)\delta$
for any boundary edge $\underline{z}_{\text{ext}}\in\left[\tilde{q}_{1}^{\delta}\tilde{q}_{2}^{\delta}\right]$.
Then, since $H\to h$ locally uniformly and $h<1$ on $\left[\tilde{q}_{1}\tilde{q}_{2}\right]$,
for small enough $\delta$ we have $\sum_{\underline{z}_{\text{ext}}\in\left[\tilde{q}_{1}^{\delta}\tilde{q}_{2}^{\delta}\right]}t_{z}\cdot\left[V(\underline{z}_{\text{int}})\left(H(\underline{z}_{\text{ext}})-1\right)\right]\leq-\smallc'(\underline{\Omega})<0$.
Also, there exists some $\const''(\eta)>0$ such that $\Delta^{\delta}H^{N}\geq-\const''\left(\eta\right)m^{2}\left(H^{\text{min}}\right)^{N}$
for any $N\geq2$ by Lemma \ref{lem:Nthpower}, as soon as $\delta$
is small enough so that $H>0$ globally on $\underline{\Omega}^{\delta}$.

Fix $N$ large enough so that $\int_{D}h^{N+1}d^{2}z\leq\frac{\smallc'(\underline{\Omega})}{4m^{2}\const''(\eta)}$.
Then we can again restrict to $\delta$ small enough so that $\sum_{u\in D^{\delta}}H^{N}(u)\mu_{\Gamma}^{\delta}\left(u\right)\leq\frac{\smallc'(\underline{\Omega})}{2m^{2}\const''(\eta)}$
(since $N$ is already fixed and $H^{N}$ may be bounded near $\partial\Omega^{\delta}$
by Proposition \ref{prop:massive-beurling}). Given this, we let $\tilde{H}:=H^{N}-1$
and apply the discrete Green's formula (\ref{eq:discgreens}),
\[
\sum_{u\in\Gamma(\Omega^{\delta})}\left[V\Delta^{\delta}\tilde{H}\right](u)\mu_{\Gamma}^{\delta}\left(u\right)=\sum_{z_{\text{ext}}\in C^{\delta}}t_{z}\cdot\left[V(\underline{z}_{\text{int}})\tilde{H}(\underline{z}_{\text{ext}})\right]-\sum_{z_{\text{ext}}\in\left[\tilde{p}_{2}^{\delta}\tilde{p}_{1}^{\delta}\right]}t_{z}\cdot\left[\tilde{H}(\underline{z}_{\text{int}})\right].
\]

Note that $\sum_{z_{\text{ext}}\in\left[\tilde{q}_{1}^{\delta}\tilde{q}_{2}^{\delta}\right]}t_{z}\cdot\left[V(\underline{z}_{\text{int}})\tilde{H}(\underline{z}_{\text{ext}})\right]\leq-\smallc'(\underline{\Omega})$
if $\delta$ is small enough so that $H(\underline{z}_{\text{ext}})<1$
and 
\[
\sum_{u\in\Gamma(\Omega^{\delta})}\left[V\Delta^{\delta}\tilde{H}\right](u)\mu_{\Gamma}^{\delta}\left(u\right)\geq-\frac{c'(\underline{\Omega})}{2}.
\]
That is,
\[
\sum_{z_{\text{ext}}\in\left[\tilde{p}_{2}^{\delta}\tilde{p}_{1}^{\delta}\right]}t_{z}\cdot\left[\tilde{H}(\underline{z}_{\text{int}})\right]\leq-\frac{\smallc'(\underline{\Omega})}{2}+\sum_{z_{\text{ext}}\in\left[p_{1}^{\delta}\tilde{q}_{1}^{\delta}\right]\cup\left[\tilde{q}_{2}^{\delta}p_{2}^{\delta}\right]}t_{z}\cdot\left[V(\underline{z}_{\text{int}})\tilde{H}(\underline{z}_{\text{ext}})\right].
\]
However, the summands in the second sum are also asymptotically negative
in bulk (since $h^{N}-1<0$). The only terms to control are the ones
near $p_{1,2}^{\delta}$. This may be done as in \cite[(6.10)]{chsm2012}.
\end{proof}
\begin{figure}
\centering
\includegraphics[width=0.5\paperwidth]{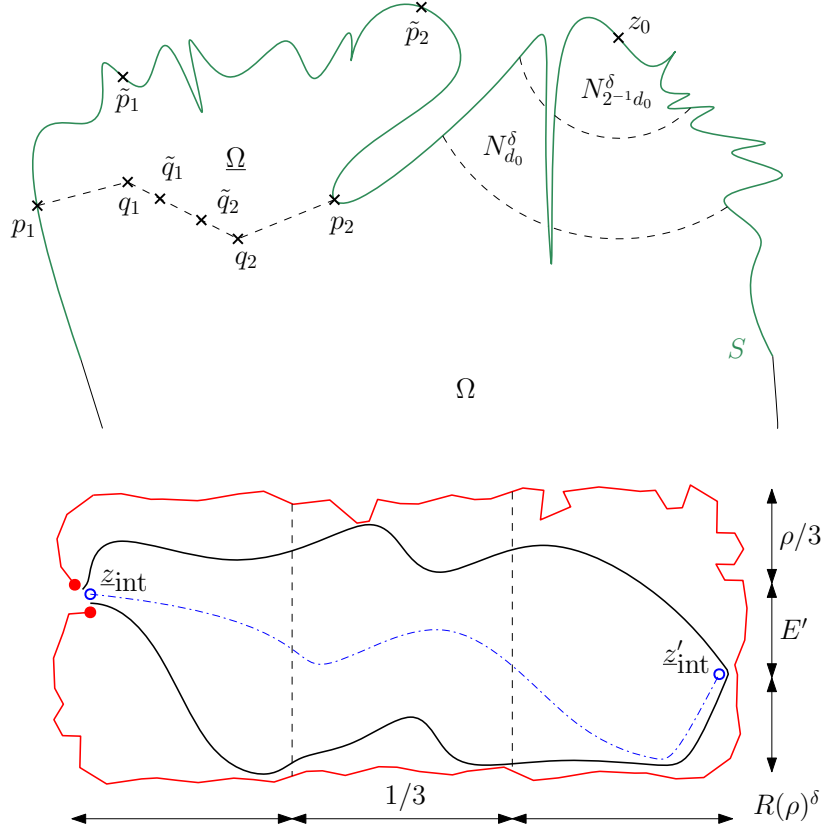}

\caption{(Top) Settings for the two estimates in Section \ref{subsec:Boundary-Estimates}.
(Bottom) Setting for the analysis of Section \ref{subsec:Analysis-of-the}.}
\label{fig:boundary}
\end{figure}

\subsection{Analysis of the Degenerate Observable\label{subsec:Analysis-of-the}}

Applying the estimates from previous sections, we derive estimates
of the 'degenerate' $2$-point observable which are crucial in proving
Theorem \ref{thm:rsw}. We undertake it here since it may be done
without identifying a unique continuum scaling limit.

Fix a rectangle $R=R(\rho):=\left(0,1\right)+\left(0,\rho i\right)$
which we discretise and give the FK-Ising measure with entirely wired
boundary condition. Then for any boundary rhombus $\left\langle \underline{z}_{\text{int}}u_{1}\underline{z}_{\text{ext}}u_{2}\right\rangle $
on the segment $\left[\left(\frac{2\rho i}{3}\right)^{\delta},\left(\frac{\rho i}{3}\right)^{\delta}\right]$
with inner dual vertex $\underline{z}_{\text{int}}$, we can consider
the degenerate $2$-point observable $F_{\underline{z}_{\text{int}}}^{\delta}:=F_{\left(R^{\delta},a^{\delta},b^{\delta}\right)}^{\delta}$
where $a^{\delta}=\left\langle u_{1}\underline{z}_{\text{int}}\right\rangle $
and $b^{\delta}=\left\langle u_{2}\underline{z}_{\text{int}}\right\rangle $.
Clearly, this 2-point Dobrushin model on $\left(R^{\delta},a^{\delta},b^{\delta}\right)$
is identical to the fully wired one (see Figure \ref{fig:boundary}B).
By definition, for any other boundary rhombus $\left\langle \underline{z}_{\text{int}}'u_{1}'\underline{z}_{\text{ext}}'u_{2}'\right\rangle $,
evaluating $F^{\delta}$ at the boundary corners $\left\langle u_{1,2}'\underline{z}_{\text{int}}'\right\rangle $
simply gives the dual-crossing probability 
\begin{equation}
\left\vert F_{\underline{z}_{\text{int}}}^{\delta}\left(\left\langle u_{1,2}'\underline{z}_{\text{ext}}'\right\rangle \right)\right\vert =\mathbb{P}\left[\underline{z}{}_{\text{int}}\xleftrightarrow{\text{d}}\underline{z}'_{\text{int}}\right].\label{eq:ftoprob-1}
\end{equation}

Because $a^{\delta}$ and $b^{\delta}$ are $O(\delta)$-apart, the
function $F_{\underline{z}_{\text{int}}}^{\delta}$ scales differently
from the usual non-degenerate 2-point observable. In this setting,
the boundary values of the square integral $H_{\underline{z}_{\text{int}}}^{\delta}$
may be chosen to be all zero except at the dual vertex $\underline{z}_{\text{int}}$,
where $H_{\underline{z}_{\text{int}}}^{\delta}(\underline{z}_{\text{int}})=-2\delta$.
Thinking of continuum Poisson kernel (i.e. in the massless case),
it makes sense to renormalise $H_{\underline{z}}^{\delta}$ by $\delta^{-2}$
to obtain a nontrivial limit, i.e. $F_{\underline{z}_{\text{int}}}^{\delta}$
scales like $\delta$ away from $\underline{z}_{\text{int}}$. We
show below that this intuition yields some correct bounds (in terms
of the magnitude) at least for small mass; we will carry out more
general analysis of $F_{\underline{z}_{\text{int}}}^{\delta}$ (which
coincides with the martingale observable for the spin-Ising interface
started at $\underline{z}_{\text{int}}$) in \cite{msle3}.

First we show an upper bound for the discrete $L^{1}$-norm of $H_{\underline{z}_{\text{int}}}^{\delta}$.
\begin{lem}
There exists some $m_{1}=m_{1}(\eta,\rho),\const(\rho,\eta)>0$ such
that, at $m\leq m_{1}$, for any $\underline{z}_{\text{int}}$ on
$\left[\left(\frac{2\rho i}{3}\right)^{\delta},\left(\frac{\rho i}{3}\right)^{\delta}\right]$,
\begin{equation}
0\leq-\sum_{w\in\Gamma^{*}\left(R^{\delta}\right)}\mu_{\Gamma^{*}}^{\delta}(w)H_{\underline{z}_{\text{int}}}^{\delta}(w)\leq\const(\rho,\eta)\delta^{2}.\label{eq:hupperbound}
\end{equation}
\end{lem}

\begin{proof}
We decompose $H_{\underline{z}_{\text{int}}}^{\delta}$ (and not its
square, as we have done throughout most of this paper) on the dual
graph $\Gamma^{*}$ into its harmonic and non-harmonic parts:
\[
H_{\underline{z}_{\text{int}}}^{\delta}(w)=-2\delta\cdot\omega^{\delta}\left(w,\left\{ \underline{z}_{\text{int}}\right\} ,\Gamma^{*}\left(R^{\delta}\right)\right)+\sum_{w'\in\Gamma^{*}\left(R^{\delta}\right)}\mu_{\Gamma^{*}}^{\delta}(w')G_{\Gamma^{*}\left(R^{\delta}\right)}^{\delta}(w,w')\Delta^{\delta}H_{\underline{z}_{\text{int}}}^{\delta}(w').
\]
We analyse the two terms separately.

First, we have the standard estimate $\omega^{\delta}\left(w,\left\{ \underline{z}_{\text{int}}\right\} ,\Gamma^{*}\left(R^{\delta}\right)\right)\apprle\frac{\delta}{\left\vert w-\underline{z}_{\text{ext}}\right\vert }$.
A simple derivation using discrete comparison principle goes as follows:
RHS is (continuum) subharmonic, so in view of \cite[Lemma 2.2(ii)]{chsm2011},
it is discrete subharmonic on points $w$ such that $\left\vert w-\underline{z}_{\text{ext}}\right\vert \geq\const(\eta)\delta$
for some $\const(\eta)>0$; LHS is zero on $\partial\Gamma^{*}\left(R^{\delta}\right)\setminus\left\{ \underline{z}_{\text{int}}\right\} $
and may be bounded by a constant multiple of RHS if $\left\vert w-\underline{z}_{\text{ext}}\right\vert \leq\const(\eta)\delta$
in any case since it is globally bounded by $1$. Given this pointwise
estimate, we get by integration
\[
\sum_{w\in\Gamma^{*}\left(R^{\delta}\right)}\mu_{\Gamma^{*}}^{\delta}(w)\omega^{\delta}\left(w,\left\{ \underline{z}_{\text{int}}\right\} ,\Gamma^{*}\left(R^{\delta}\right)\right)\leq\const(\rho,\eta)\delta,
\]
as desired.

For the second term,
\begin{alignat*}{1}
 & 0\leq-\sum_{w\in\Gamma^{*}\left(R^{\delta}\right)}\mu_{\Gamma^{*}}^{\delta}(w)\sum_{w'\in\Gamma^{*}\left(R^{\delta}\right)}\mu_{\Gamma^{*}}^{\delta}(w')G_{\Gamma^{*}\left(R^{\delta}\right)}^{\delta}(w,w')\Delta^{\delta}H_{\underline{z}_{\text{int}}}^{\delta}(w')\\
 & \leq\sum_{w'\in\Gamma^{*}\left(R^{\delta}\right)}\mu_{\Gamma^{*}}^{\delta}(w')\sum_{w\in\Gamma^{*}\left(R^{\delta}\right)}\mu_{\Gamma^{*}}^{\delta}(w)\left\vert G_{\Gamma^{*}\left(R^{\delta}\right)}^{\delta}(w,w')\right\vert \max\left(0,\Delta^{\delta}H_{\underline{z}_{\text{int}}}^{\delta}(w')\right)\\
 & \leq(1+\rho)\sum_{w'\in\Gamma^{*}\left(R^{\delta}\right)}\mu_{\Gamma^{*}}^{\delta}(w')\dist\left(w',\partial R\right)\max\left(0,\Delta^{\delta}H_{\underline{z}_{\text{int}}}^{\delta}(w')\right),
\end{alignat*}
by (\ref{eq:gboundary}). Write $d(w'):=\dist\left(w',\partial R\right)$,
and let us now show
\begin{align}
&\sum_{w'\in\Gamma^{*}\left(R^{\delta}\right)}\mu_{\Gamma^{*}}^{\delta}(w')d(w')\max\left(0,\Delta^{\delta}H_{\underline{z}_{\text{int}}}^{\delta}(w')\right)\nonumber \\
&\apprle m\left(m(1+\rho)\delta^{2}+(1+m^{2}(1+\rho^{2}))^{3}\sum_{w'\in\Gamma^{*}\left(R^{\delta}\right)}\mu_{\Gamma^{*}}^{\delta}(w')\left\vert H_{\underline{z}_{\text{int}}}^{\delta}(w')\right\vert \right),\label{eq:nonharmonicpartbound}
\end{align}
given which we may clearly restrict to small enough $m$ to obtain
(\ref{eq:hupperbound}). We estimate the Laplacian using (\ref{eq:hlap_dual}):
for $d(w')<2\const\delta$ as in Proposition \ref{prop:dreg-bulk},
there are $(1+\rho)O(\delta^{-1})$ such terms, so the trivial bound
from (\ref{eq:inth}) and $-2\delta\leq H_{\underline{z}_{\text{int}}}^{\delta}(w)\leq0$
yields the first term in the bound. If $d(w')\geq2\const\delta$,
we have again from (\ref{eq:hlap_dual}) and Proposition \ref{prop:dreg-bulk}
\begin{alignat*}{1}
\frac{\dist\left(w',\partial R\right)\Delta^{\delta}H_{\underline{z}_{\text{int}}}^{\delta}(w')}{m(1+m^{2}(1+\rho^{2}))^{3}} & \apprle\frac{\sum_{w\in\Gamma^{*}\left(B_{d(w')}^{\delta}(w')\right)}\mu_{\Gamma^{*}}^{\delta}(w)\left\vert H_{\underline{z}_{\text{int}}}^{\delta}(w)\right\vert }{d(w')^{2}},
\end{alignat*}
and therefore the area integral of LHS is bounded, up to a universal
factor, by the integral of $\left\vert H_{\underline{z}_{\text{int}}}^{\delta}(w)\right\vert =-H_{\underline{z}_{\text{int}}}^{\delta}(w)$
by Hardy-Littlewood maximal theorem \cite{stein} (say, apply to the
piecewise constant extension of $H_{\underline{z}_{\text{int}}}^{\delta}$
to each face in $\Gamma^{*}$). Therefore the proof is finished.
\end{proof}
Preceding integral estimate can easily be improved to the following
pointwise bound.
\begin{lem}
\label{lem:h_pointwise}There exists $0<m_{2}=m_{2}(\rho,\eta)\leq m_{1}$,
such that at $m\leq m_{2}$, for any fixed $\beta\in\left(0,1\right)$,
\[
-\const\left(\beta,\rho\right)\delta^{2}\apprle\min\left\{ H_{\underline{z}_{\text{int}}}^{\delta}(w):\re w\in\left(\beta-2\delta,\beta+2\delta\right)\right\} ,
\]
for some $\const\left(\beta,\rho\right)>0$.
\end{lem}

\begin{proof}
Write $\min(\beta):=\min\left\{ H_{\underline{z}_{\text{int}}}^{\delta}(w):\re w\in\left(\beta-2\delta,\beta+2\delta\right)\right\} $.
By minimum principle, there is a nearest-neighbour path $l^{\delta}$
of dual vertices from some $w_{0}$ with $\re w_{0}\in\left(\beta-2\delta,\beta+2\delta\right)$
which ends at $\underline{z}_{\text{int}}$ along which $H_{\underline{z}_{\text{int}}}^{\delta}\leq\min(\beta)$.
Then on $\Gamma^{*}\left(R^{\delta}\right)\setminus l^{\delta}$,
we again have the decomposition
\begin{align*}
H_{\underline{z}_{\text{int}}}^{\delta}(w)&\leq\min(\beta)\cdot\omega^{\delta}\left(w,l^{\delta},\Gamma^{*}\left(R^{\delta}\right)\setminus l^{\delta}\right)\\
&+\sum_{w'\in\Gamma^{*}\left(R^{\delta}\right)\setminus l^{\delta}}\mu_{\Gamma^{*}}^{\delta}(w')G_{\Gamma^{*}\left(R^{\delta}\right)\setminus\omega^{\delta}}^{\delta}(w,w')\Delta^{\delta}H_{\underline{z}_{\text{int}}}^{\delta}(w').
\end{align*}
Again let us consider the area integral of the involved functions:
\begin{itemize}
\item The integral of $H_{\underline{z}_{\text{int}}}^{\delta}(w)$ is bounded
below by $-\const(\rho,\eta)\delta^{2}$ by (\ref{eq:hupperbound}).
\item It is clear from standard harmonic function estimates \cite[Proposition 2.11]{chsm2011}
that the integral of $\omega^{\delta}\left(w,l^{\delta},\Gamma^{*}\left(R^{\delta}\right)\setminus l^{\delta}\right)$
is bounded below by some $\smallc(\beta,\rho,\eta)>0$.
\item Since $G_{\Gamma^{*}\left(R^{\delta}\right)\setminus l^{\delta}}^{\delta}(w,w')\geq G_{\Gamma^{*}\left(R^{\delta}\right)}^{\delta}(w,w')$,
we have
\begin{align*}
&\sum_{w'\in\Gamma^{*}\left(R^{\delta}\right)\setminus\omega^{\delta}}\mu_{\Gamma^{*}}^{\delta}(w')G_{\Gamma^{*}\left(R^{\delta}\right)\setminus\omega^{\delta}}^{\delta}(w,w')\Delta^{\delta}H_{\underline{z}_{\text{int}}}^{\delta}(w')\\
&\geq\sum_{w'\in\Gamma^{*}\left(R^{\delta}\right)}\mu_{\Gamma^{*}}^{\delta}(w')G_{\Gamma^{*}\left(R^{\delta}\right)}^{\delta}(w,w')\max\left(0,\Delta^{\delta}H_{\underline{z}_{\text{int}}}^{\delta}(w')\right),
\end{align*}
and the area integral of RHS is bounded below by some other $-m\const(\rho,\eta)\delta^{2}$
by (\ref{eq:nonharmonicpartbound}) and (\ref{eq:hupperbound}).
\end{itemize}
Given the above three bounds, we clearly have $\min(\beta)\geq-\const\left(\beta,\rho\right)\delta^{2}$
if $m$ is small enough.
\end{proof}
We are ready to give the needed lower bound on the crossing expectation:
\begin{cor}
\label{cor:rsw-num}There exists some $m_{0}=m_{0}(\rho,\eta),\smallc(\rho,\eta)>0$
such that, at $m\leq m_{0}$, for any $\underline{z}_{\text{int}}$
on $\left[\left(\frac{2\rho i}{3}\right)^{\delta},\left(\frac{\rho i}{3}\right)^{\delta}\right]$,
\begin{equation}
\sum_{\underline{z}'\in\left[\left(1+\frac{2\rho i}{3}\right)^{\delta},\left(1+\frac{\rho i}{3}\right)^{\delta}\right]}\left\vert F_{\underline{z}_{\text{int}}}^{\delta}(\underline{z}')\right\vert ^{2}\geq\smallc(\rho,\eta)\delta.\label{eq:fsqtheta}
\end{equation}
As a result, there exists a constant $\smallc(\rho,\eta)>0$ such
that at $m=m_{0}$ we have
\begin{equation}
\sum\mathbb{P}\left[\underline{z}_{\text{int}}\xleftrightarrow{\text{d}}\underline{z}'{}_{\text{int}}\right]\geq\smallc\left(\rho,\eta\right)\delta^{-1},\label{eq:Ptheta}
\end{equation}
summing over $\underline{z}\in\left[\left(\frac{2\rho i}{3}\right)^{\delta},\left(\frac{\rho i}{3}\right)^{\delta}\right],\underline{z}'\in\left[\left(1+\frac{2\rho i}{3}\right)^{\delta},\left(1+\frac{\rho i}{3}\right)^{\delta}\right]$.
\end{cor}

\begin{proof}
Since $\left(F_{\underline{z}_{\text{int}}}^{\delta}\right)^{2}$
is the discrete (normal) derivative of $H_{\underline{z}_{\text{int}}}^{\delta}$
on the boundary (with modified but uniformly positive weights as in
Lemma \ref{lem:dbvp}) and $H_{\underline{z}_{\text{int}}}^{\delta}\left(\underline{z}'_{\text{ext}}\right)=0$,
to get (\ref{eq:fsqtheta}) it suffices to show
\[
\sum_{\underline{z}'\in\left[\left(1+\frac{2\rho i}{3}\right)^{\delta},\left(1+\frac{\rho i}{3}\right)^{\delta}\right]}H_{\underline{z}_{\text{int}}}^{\delta}(\underline{z}'_{\text{int}})\leq-\smallc(\rho,\eta)\delta^{2}.
\]
By (\ref{eq:defh}), we may instead show this on the primal vertices,
i.e. write $E'$ for the set of the boundary primal vertices on $\left[\left(1+\frac{2\rho i}{3}\right)^{\delta},\left(1+\frac{\rho i}{3}\right)^{\delta}\right]$,
and show
\begin{equation}
\sum_{u\in\Gamma\left(R^{\delta}\right),u\sim E'}H_{\underline{z}_{\text{int}}}^{\delta}(u)\leq-\smallc(\rho,\eta)\delta^{2}.\label{eq:htheta}
\end{equation}

While the value of $H_{\underline{z}_{\text{int}}}^{\delta}$ on $\partial\Gamma\left(R^{\delta}\right)$
is identically zero, there is always some $u_{0}\subset\Gamma\left(R^{\delta}\right)$
such that $\left\vert u_{0}-\underline{z}_{\text{int}}\right\vert \leq4\delta$
where $H_{\underline{z}_{\text{int}}}^{\delta}(u_{0})\leq-\smallc\left(\rho,\eta\right)\delta$:
if not, for any primal vertex $u$ incident to $\underline{z}_{\text{int}}$,
we would have $\sum_{z\sim u}\left\vert \partial^{\delta}H_{\underline{z}_{\text{int}}}^{\delta}\vert_{\Gamma}(z)\right\vert =o(1)$,
which contradicts (\ref{eq:delh_fsq}) since the corner value $\left\vert F_{\underline{z}_{\text{int}}}^{\delta}\right\vert ^{2}\left(u\underline{z}_{\text{int}}\right)=1-o(1)$
by assumption. Applying discrete Green's formula (\ref{eq:discgreens})
on $H_{\underline{z}_{\text{int}}}^{\delta}$ and the harmonic measure
$V:=\omega^\delta\left(\cdot,E',\Gamma\left(R^{\delta}\setminus\left\{ u_{0}\right\} \right)\right)$
on $\Gamma\left(R^{\delta}\setminus\left\{ u_{0}\right\} \right)$,
\begin{align*}
&\sum_{u\in\Gamma(R^{\delta})\setminus\left\{ u_{0}\right\} }\left[-V\Delta^{\delta}H_{\underline{z}_{\text{int}}}^{\delta}\right](u)\mu_{\Gamma}^{\delta}\left(u\right)\\
&=\sum_{u\sim u_{\text{ext}}\in E'}t_{(uu_{\text{ext}})}\cdot\left[H_{\underline{z}_{\text{int}}}^{\delta}(u)\right]-H_{\underline{z}_{\text{int}}}^{\delta}(u_{0})\sum_{u\sim u_{0}}t_{\left(uu_{0}\right)}V(u).
\end{align*}
Again by standard estimates \cite[Lemma 3.17]{chsm2011} and since
$t\apprge1$, we have $\sum_{u\sim u_{0}}t_{\left(uu_{0}\right)}V(u)\geq\smallc(\rho,\eta)\delta$.
We will now show that LHS is bounded above by $m\const(\rho,\eta)\delta^{2}$,
so that we may set small enough $m$ to get (\ref{eq:htheta}).

We analyse the LHS by dividing $\left(0,1\right)\ni\re u$ into two
pieces. Again it is easy to deduce from \cite[Lemma 3.17]{chsm2011}
that $V(u)\leq\const\left(\rho,\eta\right)\dist\left(u,\partial R\right)$
for $\re u<\frac{2}{3}$. Therefore, for small $m$,
\[
\sum_{\re u<\frac{2}{3}}\left[-V\Delta^{\delta}H_{\underline{z}_{\text{int}}}^{\delta}\right](u)\mu_{\Gamma}^{\delta}\left(u\right)\leq m\const(\rho,\epsilon)\delta^{2},
\]
where we essentially repeat the proof of (\ref{eq:nonharmonicpartbound}) but on $\Gamma$, with the lower bound (\ref{eq:hlap_primal}). The required $O(\delta ^2)$ area integral bound directly comes from \eqref{eq:hupperbound}, since the bound on $\Gamma^*$ bounds the area integral on $\Gamma$ from (\ref{eq:defh}).

For the bound for $\re u\geq\frac{2}{3}$, note that $\osc_{\re u\geq\frac{1}{3}}H_{\underline{z}_{\text{int}}}^{\delta}\leq\const\left(\beta=\frac{1}{3},\rho,\epsilon\right)\delta^{2}$
from Lemma \ref{lem:h_pointwise}. By Remark \ref{rem:massive-beurling-gen}
and Proposition \ref{prop:dreg-bulk}, $\left\vert F_{\underline{z}_{\text{int}}}^{\delta}(z)\right\vert ^{2}\leq\const\left(\rho,\epsilon\right)\delta^{2}$
if $\re z\geq\frac{2}{3}$. Again using (\ref{eq:hlap_primal}) and
bounding $V\leq1$, we have $\sum_{\re u\geq\frac{2}{3}}\left[-V\Delta^{\delta}H_{\underline{z}_{\text{int}}}^{\delta}\right](u)\mu_{\Gamma}^{\delta}\left(u\right)\leq m\const(\rho,\epsilon)\delta^{2}$.

Therefore (\ref{eq:htheta}) holds for small $m$. Since $\left\vert F_{\underline{z}_{\text{int}}}^{\delta}(\underline{z}')\right\vert \leq\sqrt{\const\left(\rho,\epsilon\right)}\delta$
as noted above, we have (\ref{eq:Ptheta}) given (\ref{eq:ftoprob-1}).
\end{proof}

\section{Continuum Observable\label{sec:Continuum-Observable}}

\subsection{\label{subsec:contobs_intro}How do the Continuum Observables Look?}

In this descriptive section, we define and illustrate the continuous
limits of the discrete observables defined in Section \ref{subsec:Fermionic-Observables}
and their square integrals; the proofs of their uniqueness (and therefore
convergence and existence) will be given in the next section. We will
consider observables corresponding to two masses: massive $m>0$ (assumed
to be fixed and implicit) and the massless case $m=0$ (marked by
an explicit superscript). They will turn out to be related by (up
to constant factors) exactly the factorisation in Definition \ref{def:holpart}.

In the massless case, the square integrals are harmonic and thus may
be identified by their boundary values, which are locally constant.
Therefore, they are a linear combination of \emph{harmonic measures}
(cf., e.g. \cite{harmonic-measure})
\begin{thm}[{\cite[Theorem 4.3]{chsm2012}}]
\label{thm:crit2pt}In the massless case, the $2$-point observable
(Definition \ref{def:2ptc} with $q=0$) converges to a holomorphic
function $f_{\left(\Omega,a,b\right)}^{m=0}$, unique up to a sign,
satisfying the following:
\begin{itemize}
\item for any conformal pullback $D\xrightarrow{\varphi}\Omega$, pullback
$f_{\left(\Omega,a,b\right)}^{\text{pb};m=0}=\left(f_{\left(\Omega,a,b\right)}^{m=0}\circ\varphi\right)\cdot\left(\varphi'\right)^{1/2}$
coincides with the observable $f_{\left(D,\varphi^{-1}(a),\varphi^{-1}(b)\right)}^{m=0}$
, i.e. the observable is conformally covariant;
\item $f_{\left(D,\varphi^{-1}(a),\varphi^{-1}(b)\right)}^{m=0}$ satisfies
$(rh)_{f}$, and in particular is smooth away from $\varphi^{-1}(a),\varphi^{-1}(b)$;
\item square integral $h_{\left(\Omega,a,b\right)}^{m=0}:=\imm\int\left(f_{\left(\Omega,a,b\right)}^{m=0}\right)^{2}dz$
is conformally invariant, i.e. $h_{\left(\Omega,a,b\right)}^{m=0}=h_{\left(D,\varphi^{-1}(a),\varphi^{-1}(b)\right)}^{m=0}\circ\varphi$;
\item the harmonic function $h_{\left(\Omega,a,b\right)}^{m=0}(z)$ coincides
with $\omega\left(z,\left(ba\right),\Omega\right)$, the \emph{harmonic
measure} of $(ba)\subset\partial\Omega$ seen from $z$, characterised
by
\[
h_{\left(\Omega,a,b\right)}^{m=0}=0\text{ on }\left(ab\right)\text{, }h_{\left(\Omega,a,b\right)}^{m=0}=1\text{ on }\left(ba\right);
\]
\item on the strip $\mathbb{S}:=\mathbb{R}+\left(0,i\right)$, we have explicit
observables $f_{\left(\mathbb{S},-\infty,\infty\right)}^{m=0}\equiv1$
and $h_{\left(\mathbb{S},-\infty,\infty\right)}^{m=0}(x+yi)=y$.
\end{itemize}
\end{thm}

Accordingly, we define the continuous massive observable as the function
satisfying the following, shown to be unique in the next subsection.
\begin{defn}
\label{def:c2pt}Given the $2$-point marked domain $\left(\Omega,a,b\right)$,
the continuous $2$-point observable $f_{\left(\Omega,a,b\right)}$
is the massive holomorphic function, unique up to a sign, having the
following properties:
\begin{itemize}
\item $f_{\left(\Omega,a,b\right)}$ satisfies $(rh)_{f}$;
\item square integral $h_{\left(\Omega,a,b\right)}:=\imm\int\left(f_{\left(\Omega,a,b\right)}\right)^{2}dz$
is a solution of $\Delta h_{\left(\Omega,a,b\right)}=-4m\left\vert \nabla h_{\left(\Omega,a,b\right)}\right\vert $;
\item $h_{\left(\Omega,a,b\right)}(z)$ has the following boundary values
\[
h_{\left(\Omega,a,b\right)}\equiv0\text{ on }\left(ab\right)\text{, }h_{\left(\Omega,a,b\right)}\equiv1\text{ on }\left(ba\right);
\]
\item holomorphic parts $\underline{f}_{\left(\Omega,a,b\right)},\underline{f}_{\left(\Omega,a,b\right)}^{\text{pb}}$
respectively coincide with $f_{\left(\Omega,a,b\right)}^{m=0},f_{\left(D,\varphi^{-1}(a),\varphi^{-1}(b)\right)}^{m=0}$
up to real multiplicative constants.
\end{itemize}
\end{defn}

\begin{rem}
Near the points $a,b$, where $h_{\left(\Omega,a,b\right)}^{m=0}$
has a jump discontinuity, the conformal invariance allows us to deduce
$f_{\left(\Omega,a,b\right)}^{m=0}(z)$ has series expansions in half-integers
with leading inverse square root poles, in \emph{any} simply connected
domain $\Omega$ in the case where the prime ends $a,b$ are single
accessible points. The fact that in the massive case there has to
be \emph{some} blow-up in inverse square root rate may be deduced
from the mean value theorem (from below) and Proposition \ref{prop:creg-bulk}
(from above).
\end{rem}

\begin{thm}[{\cite[Proof of Theorem 6.1]{chsm2012}}]
\label{thm:crit4pt}In the massless case, the $4$-point observable
(Definition \ref{def:d4pt} with $q=0$) converges to a holomorphic
function $f_{\left(\Omega,a,b,c,d\right)}^{m=0}$, unique up to a
sign, satisfying the following:
\begin{itemize}
\item analogues of the first three properties in Theorem \ref{thm:crit2pt};
\item the harmonic function $h_{\left(\Omega,a,b,c,d\right)}^{m=0}$ has
the boundary values
\begin{equation}
h_{\left(\Omega,a,b,c,d\right)}^{m=0}\equiv0\text{ on }\left(ab\right)\text{, }h_{\left(\Omega,a,b,c,d\right)}^{m=0}\equiv1\text{ on }\left(bc\right),h_{\left(\Omega,a,b,c,d\right)}^{m=0}\equiv\chi_{m=0}\text{ on }\left(ca\right),\label{eq:4ptcbvp}
\end{equation}
where $\chi_{m=0}\in\left(0,1\right)$ is the conformally invariant
unique value which realises $(rh)_{h}$ for $h_{\left(\Omega,a,b,c,d\right)}^{m=0}$;
\item on the slit-strip $\mathbb{S}_{\chi_{m=0}}:=\mathbb{S}\setminus\left(\mathbb{R}_{<0}+\chi_{m=0}i\right)$,
we have the explicit observables
\[
f_{\left(\mathbb{S}_{\chi_{m=0}},-\infty+0\cdot i,+\infty,-\infty+i,\chi_{m=0}i\right)}^{m=0}\equiv1,h_{\left(\mathbb{S}_{\chi_{m=0}},-\infty+0\cdot i,+\infty,-\infty+i,\chi_{m=0}i\right)}^{m=0}(x+yi)=y.
\]
\end{itemize}
\end{thm}

And we have
\begin{defn}
\label{def:c4pt}Given the $4$-point marked domain $\left(\Omega,a,b,c,d\right)$,
the continuous $2$-point observable $f_{\left(\Omega,a,b,c,d\right)}$
is the massive holomorphic function, unique up to a sign, having the
following properties:
\begin{itemize}
\item pullback $f_{\left(\Omega,a,b,c,d\right)}^{\text{pb}}$ on any $D$
satisfies $(rh)_{f}$;
\item square integral $h_{\left(\Omega,a,b,c,d\right)}:=\imm\int\left(f_{\left(\Omega,a,b,c,d\right)}\right)^{2}dz$
is a solution of $\Delta h_{\left(\Omega,a,b,c,d\right)}=-4m\left\vert \nabla h_{\left(\Omega,a,b,c,d\right)}\right\vert $;
\item $h_{\left(\Omega,a,b,c,d\right)}(z)$ has the following boundary values
\[
h_{\left(\Omega,a,b,c,d\right)}\equiv0\text{ on }\left(ab\right)\text{, }h_{\left(\Omega,a,b,c,d\right)}\equiv1\text{ on }\left(bc\right),h_{\left(\Omega,a,b,c,d\right)}\equiv\chi_{m}\text{ on }\left(ca\right),
\]
where $\chi_{m}\in\left(0,1\right)$ is the unique value which realises
$(rh)_{h}$ for $h_{\left(\Omega,a,b,c,d\right)}$;
\item holomorphic parts $\underline{f}_{\left(\Omega,a,b,c,d\right)},\underline{f}_{\left(\Omega,a,b,c,d\right)}^{\text{pb}}$
respectively coincide with $f_{\left(\Omega,a,b,c,d\right)}^{m=0},f_{\left(D,\varphi^{-1}(a),\varphi^{-1}(b),\varphi^{-1}(c),\varphi^{-1}(d)\right)}^{m=0}$
up to real multiplicative constants.
\end{itemize}
\end{defn}

\begin{rem}
\label{rem:fc_zero}If the boundary arc near $d$ is smooth, it is
easy to see that massless $\left(f_{\left(\Omega,a,b,c,d\right)}^{m=0}\right)^{2}$
has a simple zero at $d$. Unlike the jump discontinuities at $a,b$,
regularity of the boundary is important; e.g. there is no zero if $d$
is the endpoint of an inward slit.
\end{rem}

\subsection{$2$-Point Observable and Improved Regularity}

We now identify the limit $f_{\left(\Omega,a,b\right)}$ of the $2$-point
observables as the unique function satisfying the conditions set out
in Definition \ref{def:c2pt}. First, we will prove uniqueness of
the solution of the PDE given the boundary values, which any subsequential
limit of the two-point observable (or rather, the square integral
thereof) satisfies. Then, given uniqueness and therefore convergence,
we will be able to also characterise the function in terms of the
factorisation of Definition \ref{def:holpart}, which improves the
boundary regularity.

We first state the following standard lemma.
\begin{lem}
\label{lem:chelkak}Let $G_{\Omega}$ be the Green's function for
the Dirichlet Laplacian on $\Omega$. If for some $\alpha>0$, a
locally H\"older function $g$ on $\Omega$ satisfies the estimate $\left\vert g(z)\right\vert =O({\dist(z,\partial\Omega)^{2-\alpha}})$,
then
\[
\Phi(z):=\iint_{\Omega}G_{\Omega}(z,z')g(z')d^{2}z'
\]
is twice differentiable, solves the Poisson's equation $\Delta\Phi=g$,
and takes the boundary value $0$ continuously.
\end{lem}

\begin{proof}
The proof of \cite[Lemma A.2]{s-emb} proves the boundedness and (H\"older)
continuity up to boundary of the RHS. The (local) twice differentiability
and the resulting Laplacian is a standard result, e.g. using \cite[Lemma 4.2]{GiTr}.
\end{proof}
The following proposition shows that any subsequential limit of $F_{\left(\Omega^{\delta},a^{\delta},b^{\delta}\right)}^{\delta}$
obtained by Remark \ref{rem:subsequence} has to be unique, therefore
completing the proof of convergence.
\begin{prop}
\label{prop:2pt-uniqueness}Any limit $f_{\left(\Omega,a,b\right)}$
of $F_{\left(\Omega^{\delta},a^{\delta},b^{\delta}\right)}^{\delta}$
has a square integral $h_{\left(\Omega,a,b\right)}$ which is the
unique solution to the boundary value problem in Definition \ref{def:c2pt},
and is therefore unique up to a sign.
\end{prop}

\begin{proof}
Suppose there are two solutions $f_{\left(\Omega,a,b\right)}^{1,2}$
with two square integrals $h_{\left(\Omega,a,b\right)}^{1,2}$ continuously
taking the boundary value on each open boundary arc by Proposition
\ref{prop:massive-beurling}. We have that
\[
\left\vert f_{\left(\Omega,a,b\right)}^{1,2}(z)\right\vert ^{2}=O({\dist(z,\partial\Omega)}^{-1}),
\]
by Proposition \ref{prop:creg-bulk} (or, simply by the fact that
the discrete estimate from Proposition \ref{prop:dreg-bulk} used
for precompactness is inherited). Then define
\[
\Phi(z):=\iint_{\Omega}-4m\left(\left\vert f_{\left(\Omega,a,b\right)}^{1}(w)\right\vert ^{2}-\left\vert f_{\left(\Omega,a,b\right)}^{2}(w)\right\vert ^{2}\right)G_{\Omega}(z,w)d^{2}w,
\]
which has the same Laplacian as $h_{\left(\Omega,a,b\right)}^{1}-h_{\left(\Omega,a,b\right)}^{2}$
and takes zero boundary value everywhere on $\partial\Omega$ by Lemma
\ref{lem:chelkak}. Therefore, $h_{\left(\Omega,a,b\right)}^{1}-h_{\left(\Omega,a,b\right)}^{2}-\Phi$
is a bounded harmonic function continuously taking zero boundary value
on $\partial\Omega\setminus\{a,b\}$; since $a,b$ are isolated prime
ends, there is no such nonzero harmonic function (e.g. \cite[Lemma 1.1]{harmonic-measure}).
So $h_{\left(\Omega,a,b\right)}^{1}-h_{\left(\Omega,a,b\right)}^{2}$
continuously takes zero boundary value everywhere on $\partial\Omega$,
and by comparison principle (Lemma \ref{lem:comparison}) $h_{\left(\Omega,a,b\right)}^{1}=h_{\left(\Omega,a,b\right)}^{2}$
everywhere.
\end{proof}
\begin{rem}
\label{rem:capacity}The above proof illustrates how Lemma \ref{lem:chelkak}
implies that, analogously to the harmonic case, bounded massive holomorphic
integrals $h$ cannot be supported on discrete prime ends, since they
do not have enough capacity. This in particular justifies only specifying
boundary values of square integrals on open boundary arcs, missing
a finite number of prime ends.
\end{rem}

Now we identify $f_{\left(\Omega,a,b\right)}$ as precisely the function
whose holomorphic part comes from the corresponding boundary value
problem in the massless case.
\begin{cor}
\label{cor:2ptrhf}The holomorphic part $\underline{f}_{\left(\Omega,a,b\right)}$
of $f_{\left(\Omega,a,b\right)}$ as defined in Definition \ref{def:holpart}
coincides with $f_{\left(\Omega,a,b\right)}^{m=0}$ up to a real multiplicative
constant, which satisfies
\[
\smallc\left(m\diam\Omega\right)\leq\left(f_{\left(\Omega,a,b\right)}^{m=0}/\underline{f}_{\left(\Omega,a,b\right)}\right)^{2}\leq\const\left(m\diam\Omega\right),
\]
i.e. positive constants only depending on $m\diam\Omega$. In particular,
$\underline{f}_{\left(\Omega,a,b\right)}$ satisfies $(rh)_{f}$ on
$\partial\Omega\setminus\left\{ a,b\right\} $.
\end{cor}

\begin{proof}
Exploiting uniqueness of the factorisation (\ref{eq:omega_factorisation}),
we may show this in the opposite order: i.e. $f_{\left(\Omega,a,b\right)}^{m=0}$
has a massive holomorphic counterpart $g$ as in Definition \ref{def:holpart}
such that $\underline{g}=f_{\left(\Omega,a,b\right)}^{m=0}$. But
by Propositions \ref{prop:rhh1} and \ref{prop:rhh2}, the square
integral $i(z):=\imm\int^{z}g^{2}dz$ satisfies $(rh)_{h}$ on $(ab)$
and $(ba)$: it is easy then to see that it has to coincide with $h_{\left(\Omega,a,b\right)}$
up to additive and positive multiplicative constants.

Consequently, 
\begin{equation}
g=kf_{\left(\Omega,a,b\right)}\text{ for some }k\in\mathbb{R}.\label{eq:2pt_prop}
\end{equation}
 Also note that 

\[
i\left(ba\right)-i\left(ab\right)=k^{2}\left(h_{\left(\Omega,a,b\right)}\left(ba\right)-h_{\left(\Omega,a,b\right)}\left(ab\right)\right)=k^{2}.
\]
To estimate $k^{2}$, we pullback to the unit disc: fix a map $\varphi_{\mathbb{D}}:\mathbb{D}\to\Omega$
such that $\varphi_{\mathbb{S}}(i)=a,\varphi_{\mathbb{S}}(-i)=b$.
On the truncated disc $\mathbb{D}\cap\left\{ z:\left\vert \imm z\right\vert \leq\frac{1}{2}\right\} $,
we have the universal bounds $0<\smallc\leq\left\vert f_{\left(\mathbb{D},i,-i\right)}^{m=0}\right\vert \leq\const$
: pullback to the strip is identically $1$ from Theorem \ref{thm:crit2pt},
and any fixed map from the disc to the strip mapping $\mp\infty$
to $\pm i$ also satisfies the same lower/upper bounds.

Then as usual the pullback of $g$ and $i$ becomes simply 
\[
g_{\mathbb{D}}:=\left(g\circ\varphi_{\mathbb{D}}\right)\cdot\left(\varphi_{\mathbb{D}}'\right)^{1/2}\text{, and }i\circ\varphi_{\mathbb{D}}^{-1}=\imm\int g_{\mathbb{D}}^{2}dz.
\]
 Recall the massless observable is conformally covariant, so that
$g_{\mathbb{D}}=e^{s_{g_{\mathbb{D}}}^{\mathbb{D}}}f_{\left(\mathbb{D},i,-i\right)}^{m=0}$
with $\left\Vert s_{g_{\mathbb{D}}}^{\mathbb{D}}\right\Vert _{W^{1,2}\left(\mathbb{D}\right)}\leq\const m\diam\Omega$.
So we apply Lemma \ref{lem:general_sobolev} to bound $k^{2}=\imm\int_{-1}^{1}\left(e^{s_{g_{\mathbb{D}}}^{\mathbb{D}}}f_{\left(\mathbb{D},i,-i\right)}^{m=0}\right)^{2}dz\leq\const\left(m\diam\Omega\right)$. 

For the lower bound, take the holomorphic parts of both sides of (\ref{eq:2pt_prop})
and pullback to $\mathbb{S}$. We get $\left(\underline{f}_{\left(\Omega,a,b\right)}\circ\varphi_{\mathbb{S}}\right)\cdot\left(\varphi_{\mathbb{S}}'\right)^{1/2}=k^{-1}f_{\left(\mathbb{D},i,-i\right)}^{m=0}$.
Applying Proposition \ref{prop:creg-bulk} on the massive holomorphic
pullback $f_{\left(\Omega,a,b\right)}^{\text{pb}}$ and its square
integral $h_{\left(\Omega,a,b\right)}^{\text{pb }}$ (whose oscillation
is bounded above by $1$) on $\mathbb{D}$, we have on $B_{1/4}(0)\subset\mathbb{D}$
\[
\left\vert f_{\left(\Omega,a,b\right)}^{\text{pb}}\right\vert =\left\vert e^{s_{f_{\left(\Omega,a,b\right)}^{\text{pb}}}^{\mathbb{D}}}\underline{f}_{\left(\Omega,a,b\right)}^{\text{pb}}\right\vert =\left\vert k^{-1}e^{s_{f_{\left(\Omega,a,b\right)}^{\text{pb}}}^{\mathbb{D}}}f_{\left(\mathbb{D},i,-i\right)}^{m=0}\right\vert \leq\const.
\]
Therefore, taking the $L^{2}$-norm on $B_{1/4}(0)$, we have $k^{-2}\leq\const\left\Vert e^{s_{f_{\left(\Omega,a,b\right)}^{\text{pb}}}^{\mathbb{D}}}\right\Vert _{L^{2}\left(B_{1/4}(0)\right)}^{2}\leq\const\left(m\diam\Omega\right)$,
applying Lemma \ref{lem:general_sobolev}.
\end{proof}
\begin{figure}
\centering
\includegraphics[width=0.7\paperwidth]{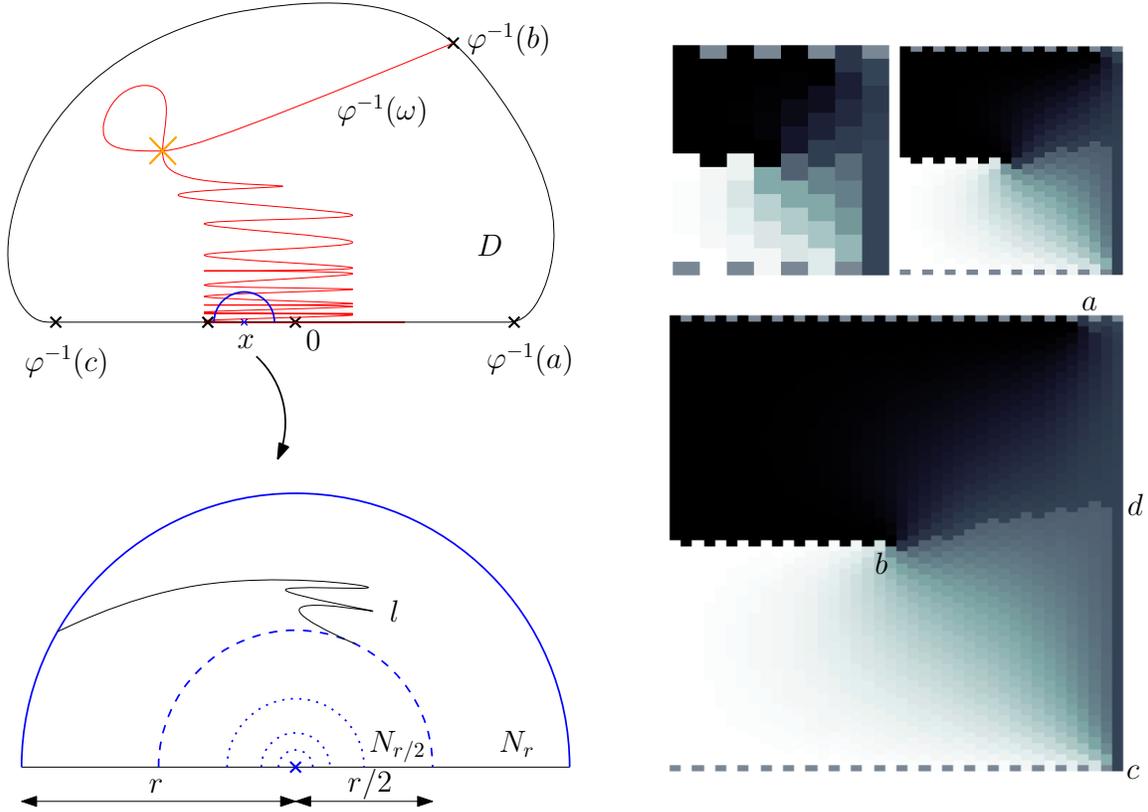}

\caption{(Left-top) Setting for the analysis of Theorem \ref{thm:levelset}. Top
shows some pathologies of the level line (pulled back) $\varphi^{-1}(\omega)$
to be ruled out: self-intersection (rays like the ones in orange are
used for the proof of contradiction) and accumulation near an open
interval containing $0$. (Left-bottom) Half-balls used in the non-existence
proof of the latter. (Right) Simulations of the 4-point square integral
$\left.H_{\left(\Omega^{\delta},a,b,c,d\right)}^{\delta}\right\vert _{\Gamma^*}$
on the slit unit square $\Omega:=R(1)\setminus\left(\frac{i}{2},\frac{1+i}{2}\right)$
discretised by hexagonal faces ($8\times8$, $20\times20$,
$40\times40$), here rendered as squares. The hue is lighter for greater
values, and points in $\left\{ H_{\left(\Omega^{\delta},a,b,c,d\right)}^{\delta}\protect\leq\chi^{\delta}\right\} $
have slight transparency to delineate the level
lines from $b^{\delta}$ to $d^{\delta}$. The mass parameter is set $m=\frac{2\sqrt{3}}{5}\simeq0.69$.}
\label{fig:sec5}
\end{figure}

\subsection{\label{subsec:4pointconvergence}Level Set Decomposition and $4$-Point
Observable}

Again, by Remark \ref{rem:subsequence} and Proposition \ref{prop:massive-beurling},
we assume some continuous function which has the properties defined
in Definition \ref{def:c4pt} is given. We need to show its uniqueness.
\begin{thm}
\label{thm:levelset}Suppose any limit $f_{\left(\Omega,a,b,c,d\right)}$
of $F_{\left(\Omega^{\delta},a^{\delta},b^{\delta},c^{\delta},d^{\delta}\right)}^{\delta}$,
such that the square integral $h_{\left(\Omega,a,b,c,d\right)}$ continuously
takes the boundary values (\ref{eq:4ptcbvp}) for some $\chi_{m}\in\left[0,1\right]$,
is given. Then $\chi_{m}\in\left(0,1\right)$, and there are two disjoint
simply connected domains $\Omega_{-},\Omega_{+}$ and the image of
a locally smooth curve $\omega$ partitioning $\Omega$, defined by
\[
\Omega_{-}:=\left\{ h_{\left(\Omega,a,b,c,d\right)}<\chi_{m}\right\} ,\Omega_{+}:=\left\{ h_{\left(\Omega,a,b,c,d\right)}>\chi_{m}\right\} ,\omega:=\left\{ h_{\left(\Omega,a,b,c,d\right)}=\chi_{m}\right\} .
\]

We have $\partial\Omega_{-}=\left[ab\right]\cup\left[da\right]\cup\omega,\partial\Omega_{+}=\left[bc\right]\cup\left[cd\right]\cup\omega$.
That is, the only limit points (therefore endpoints) of $\omega$
in $\partial\Omega$ are $b$ and $d$.
\end{thm}

\begin{proof}
Note that by the strong maximum principle, each connected component
of $\Omega_{\pm}$ is simply connected. Then following lemmas will
together imply the result.
\begin{lem}
$\chi_{m}\in(0,1)$.
\end{lem}

\begin{proof}
Without loss of generality, suppose $\chi_{m}=0$. By Proposition
\ref{prop:drhh_crhh}, any prime end in $\left(da\right)$ as a sequence
$z_{j}$ converging to it such that $h_{\left(\Omega,a,b,c,d\right)}(z_{j})\leq0$.
Then by the strong maximum principle, $h_{\left(\Omega,a,b,c,d\right)}$
is a constant, which is impossible since $h_{\left(\Omega,a,b,c,d\right)}$
is continuous up to $\left(bc\right)$ where it takes the value $1$. 
\end{proof}
\begin{lem}
We have
\begin{alignat}{1}
\overline{\Omega_{+}}\cap\left(ab\right),\overline{\Omega_{-}}\cap\left(bc\right) & =\emptyset,\label{eq:disjoint1}
\end{alignat}
and $\Omega_{-}$ and $\Omega_{+}$ are connected.
\end{lem}

\begin{proof}
Proposition \ref{prop:massive-beurling} and continuity of $h_{\left(\Omega,a,b,c,d\right)}$
up to boundary segments immediately give (\ref{eq:disjoint1}).

Let us now prove connectedness. Given any two connected components
$\Omega_{+}^{1,2}$ of $\Omega_{+}$, the intersections $\partial\Omega_{+}^{1,2}\cap\left(bc\right)$
cannot be empty: indeed, then, say, the boundary $\partial\Omega_{+}^{1}$
will be a subset of $\left[cb\right]\cup\omega$, on which $h_{\left(\Omega,a,b,c,d\right)}\leq\chi_{m}$
(except possibly at $b,c$), and the maximum principle says that $\Omega_{+}^{1}$
is empty. Fix $z_{1}\in\partial\Omega_{+}^{1}\cap\left(bc\right),z_{2}\in\partial\Omega_{+}^{2}\cap\left(bc\right)$,
say counterclockwise along $\left(bc\right)$. By Proposition \ref{prop:massive-beurling},
there is an open cover of $\left[z_{1}z_{2}\right]$ consisting of
$B_{d_{z}}\left(z\right)$ for $z\subset\left[z_{1}z_{2}\right]$
such that $B_{d_{z}}\left(z\right)\cap\Omega\subset\Omega_{+}$. Then
\[
\bigcup_{z\in\left[z_{1}z_{2}\right]}B_{d_{z}}\left(z\right)\cap\Omega,
\]
is an open connected set (easily seen, e.g., by pulling back to the
unit disc $\mathbb{D}$), which is itself connected to $\Omega_{+}^{1},\Omega_{+}^{2}$:
they are thus the same connected component.
\end{proof}
\begin{lem}
The set $\omega$ is locally the image of a smooth simple curve in
$\Omega$.
\end{lem}

\begin{proof}
Both smoothness and simpleness will come from the fact that $f_{\left(\Omega,a,b,c,d\right)}$
does not vanish on $\omega$, since then $\omega$ is locally the
image of an integral curve of $\left(f_{\left(\Omega,a,b,c,d\right)}\right)^{-2}$.
Suppose $f_{\left(\Omega,a,b,c,d\right)}$ has an interior zero (say
at $z=0\in\omega$). 

Then by Remark \ref{rem:similarity} and smoothness, it is locally
of the form $f_{\left(\Omega,a,b,c,d\right)}(z)=t(z)z^{n}$ for some
nonzero smooth function $t$ and an integer $n>0$. By rotation and
rescaling, assume $t(z)=1+O(\left\vert z\right\vert )$. We have, for small
$r>0$,
\[
h_{\left(\Omega,a,b,c,d\right)}(z=re^{\frac{ik\pi}{4n}})-h_{\left(\Omega,a,b,c,d\right)}(0)=\imm\int_{0}^{z}\left(f_{\left(\Omega,a,b,c,d\right)}(z)\right)^{2}dz
\]
 is strictly positive (i.e. in $\Omega_{+}$) for $k\equiv1\mod4$,
and strictly negative (i.e. in $\Omega_{-}$) for $k\equiv3\mod4$.
Consider for small $0<r\leq r_{0}$, four rays 
\[
re^{\frac{i\pi}{4n}},re^{\frac{3i\pi}{4n}},re^{\frac{5i\pi}{4n}},re^{\frac{7i\pi}{4n}},
\]
which are respectively subsets of $\Omega_{+},\Omega_{-},\Omega_{+},\Omega_{-}$.
Then there is a simple curve within $\Omega_{-}$ connecting $r_{0}e^{\frac{3i\pi}{4n}},r_{0}e^{\frac{7i\pi}{4n}}$
bypassing the rays. Concatenating the curve with the line segment
connecting the two points, we see that the resulting Jordan curve,
on which $h_{\left(\Omega,a,b,c,d\right)}\leq\chi_{m}$, envelopes
one of $r_{0}e^{\frac{i\pi}{4n}},r_{0}e^{\frac{5i\pi}{4n}}\in\Omega_{+}$:
by the maximum principle, this is a contradiction.
\end{proof}
\begin{lem}
We have $\overline{\Omega_{+}}\cap\left(da\right),\overline{\Omega_{-}}\cap\left(cd\right)=\emptyset$.
Therefore, the only limit points of $\omega$ in $\partial\Omega$
are $b$ and $d$.
\end{lem}

\begin{proof}
The first statement implies the second, since by Proposition \ref{prop:drhh_crhh}
$\left(da\right)$ and $\left(cd\right)$ are respectively part of
$\partial\Omega_{-}$ and $\partial\Omega_{+}$: the only way the
intersection can be nonempty is by there being infinitely many values
both below and above $\chi_{m}$ near those points. On the other hand,
by precisely that reason, $b$ and $d$ are already limit points of
$\omega$.

Suppose there is a limit point $c'$ of $\Omega_{-}$ along $(cd)$:
one cannot lie on $(bc)$. Then $c'$ is in the intersection of $\partial\Omega_{-}$
and $\partial\Omega_{+}$, and it is easy to see that the entire arc
$\left[c'd\right]$ must be limit points of $\Omega_{-}$ by connectedness
of $\Omega_{+}$. 

By pulling back to $D\subset\mathbb{H}$ with $\varphi^{-1}\left(ca\right)\subset\partial D\cap\mathbb{R}$,
suppose $\varphi^{-1}\left(c'd\right)=(-1,0)$. We will derive contradiction
by showing that $\underline{f}_{\left(\Omega,a,b,c,d\right)}^{\text{pb}}$
extends by zero to $\left[-\frac{1}{3},-\frac{1}{4}\right]$. To use
Propositions \ref{prop:rhh1} and \ref{prop:rhh2}, we need to show
that $f_{\left(\Omega,a,b,c,d\right)}^{\text{pb}}$ is bounded near
$\left[-\frac{1}{3},-\frac{1}{4}\right]$, for which it suffices (by
Proposition \ref{prop:creg-bulk}) to show that 
\begin{equation}
\osc_{B_{r}(x)\cap D}h_{\left(\Omega,a,b,c,d\right)}^{\text{pb}}=O(r)\text{ as }r\downarrow0,\label{eq:4pt_or}
\end{equation}
uniformly in $x\in\left[-\frac{1}{2},-\frac{1}{5}\right]$. We will
assume $r$ is small enough such that $B_{r}(x)\cap D=:N_{r}=B_{r}(x)\cap\mathbb{H}$.

Split $h_{+}:=\max\left(h_{\left(\Omega,a,b,c,d\right)}^{\text{pb}},\chi_{m}\right)-\chi_{m}$
and $h_{-}:=\min\left(h_{\left(\Omega,a,b,c,d\right)}^{\text{pb}},\chi_{m}\right)-\chi_{m}$.
We will show that $\max_{N_{r/2}}\left\vert h_{+}\right\vert \leq\const\max_{N_{r}(x)}\left\vert h_{-}\right\vert $.
By the maximum principle and bulk smoothness, there is a curve $l$
in the half-annulus $N_{r}\setminus N_{r/2}$ from the inner boundary
to the outer boundary where $h_{\left(\Omega,a,b,c,d\right)}^{\text{pb}}$
is greater than $\max_{N_{r/2}}\left\vert h_{+}\right\vert =\max_{N_{r/2}}h_{\left(\Omega,a,b,c,d\right)}^{\text{pb}}-\chi_{m}$.
Now, by superharmonicity, $h_{\left(\Omega,a,b,c,d\right)}^{\text{pb}}-\chi_{m}$
is bounded below by the linear combination $\max_{N_{r/2}}\left\vert h_{+}\right\vert \cdot h_{1}-\max_{N_{r/2}}\left\vert h_{-}\right\vert \cdot h_{2}$
of two harmonic functions $h_{1,2}$ on $N_{r}\setminus l$, with
boundary values
\begin{alignat*}{1}
\begin{cases}
h_{1}=1 & \text{on }l;\\
h_{1}=0 & \text{on }\partial N_{r}\setminus l,
\end{cases} & \begin{cases}
h_{2}=0 & \text{on }l\cup\left(\partial N_{r}\cap\mathbb{R}\right);\\
h_{2}=1 & \text{on }\partial N_{r}\setminus\left(l\cup\mathbb{R}\right).
\end{cases}
\end{alignat*}
By standard methods, the inner normal derivatives of $h_{1}$ on $\left[x-\frac{r}{4},x+\frac{r}{4}\right]$
are bounded below by universal $\frac{\const^{-1}}{r}$ (say by conformal
map and Beurling projection, e.g., \cite[Theorem 9.2]{harmonic-measure}),
whereas similarly those for $h_{2}$ are bounded above by $\frac{\const}{r}$.
Therefore, unless $\max_{N_{r/2}}\left\vert h_{+}\right\vert \leq\const\max_{N_{r/2}}\left\vert h_{-}\right\vert $
for a universal $\const>0$, the values of $h_{\left(\Omega,a,b,c,d\right)}^{\text{pb}}$
are strictly greater than $\chi_{m}$ near $x$, contradicting the
limit point assumption. Then, (\ref{eq:4pt_or}) follows since $\left\vert h_{-}\right\vert =-h_{-}$
is a bounded nonnegative subharmonic function vanishing on $(-1,0)$
and thus $\max_{N_{r}(x)}\left\vert h_{-}\right\vert =O(r)$ uniformly in $x\in\left[-\frac{1}{2},-\frac{1}{5}\right]$.

Then, by Proposition \ref{prop:rhh1}, $\underline{h}_{\left(\Omega,a,b,c,d\right)}^{\text{pb}}$
extends continuously (and thus smoothly) as a constant to $\left[-\frac{1}{3},-\frac{1}{4}\right]$,
and by Proposition \ref{prop:rhh2} it must have zero normal derivative
there. Therefore $\underline{f}_{\left(\Omega,a,b,c,d\right)}^{\text{pb}}$
extends by zero to $\left[-\frac{1}{3},-\frac{1}{4}\right]$, contradiction.
\end{proof}
\begin{cor}
The set $\omega$ is path-connected, i.e. it is the image of a single
curve.
\end{cor}

\begin{proof}
Any two distinct path-connected components are images of smooth nonintersecting
curves tending to $b$ and $d$: then the boundary value of $h_{\left(\Omega,a,b,c,d\right)}$
in the interior of the loop formed by concatenation of the two curves
is $\chi_{m}$ with the exception of prime ends $b,d$. So, in view
of Remark \ref{rem:capacity}, $h_{\left(\Omega,a,b,c,d\right)}\equiv\chi_{m}$
there. So the components are connected: contradiction.
\end{proof}
\end{proof}
Then we finally show uniqueness of the limit. We exploit two uniqueness
theorems: the factorisation $f_{\left(\Omega,a,b,c,d\right)}=e^{s^{\Omega}}\underline{f}_{\left(\Omega,a,b,c,d\right)}$
may be done uniquely (Theorem \ref{thm:similarity}) and the critical
4-point observable $f_{\left(\Omega,a,b,c,d\right)}^{m=0}$ is unique
up to a sign \cite[(6.7)]{chsm2012}. It suffices to prove the following.
\begin{cor}
\label{cor:4pt_uniqueness}The holomorphic part $\underline{f}_{\left(\Omega,a,b,c,d\right)}$
coincides with $f_{\left(\Omega,a,b,c,d\right)}^{m=0}$ up to a real
multiplicative constant, which satisfies the bound
\begin{equation}
\smallc\left(m\diam\Omega\right)\leq\left(f_{\left(\Omega,a,b,c,d\right)}^{m=0}/\underline{f}_{\left(\Omega,a,b,c,d\right)}\right)^{2}\leq\const\left(m\diam\Omega\right).\label{eq:4ptprop}
\end{equation}
\end{cor}

\begin{proof}
It suffices to show that $\underline{h}_{\left(\Omega,a,b,c,d\right)}:=\imm\int\left(\underline{f}_{\left(\Omega,a,b,c,d\right)}\right)^{2}dz$
(or equivalently $\underline{h}_{\left(\Omega,a,b,c,d\right)}^{\text{pb}}$
by (\ref{eq:omega_factorisation})) satisfies $(rh)_{h}$ on the four
boundary segments, and the boundary values on $\left(cd\right)$ and
$\left(da\right)$ coincide. But by Theorem \ref{thm:levelset} and
the uniqueness of the 2-point observable (Proposition \ref{prop:2pt-uniqueness}),
$f_{\left(\Omega,a,b,c,d\right)}$ restricted to $\Omega_{\pm}$ are
up to real constant factors respectively equal to the 2-point observables
$f_{\left(\Omega_{+},c,b\right)}$ and $f_{\left(\Omega_{-},a,b\right)}$,
both of which satisfy $(rh)_{f}$ on their respective boundary segments
by Corollary \ref{cor:2ptrhf}.

For example, suppose we pullback $f_{\left(\Omega_{-},a,b\right)}$
to $\left(D_{-},\varphi_{-}^{-1}(a),\varphi_{-}^{-1}(b)\right)$ by
using a conformal map $D_{-}\xrightarrow{\varphi_{-}}\Omega_{-}$.
Then we have the corresponding factorisation $f_{\left(\Omega_{-},a,b\right)}^{\text{pb}}=e^{s_{f_{\left(\Omega_{-},a,b\right)}^{\text{pb}}}^{D_{-}}}\underline{f}_{\left(\Omega_{-},a,b\right)}^{\text{pb}}$
on $D_{-}$. On the other hand, we pullback the original $f_{\left(\Omega,a,b,c,d\right)}$
using another map $D\xrightarrow{\varphi}\Omega$, where we have the
factorisation $f_{\left(\Omega,a,b,c,d\right)}^{\text{pb}}=e^{s_{f_{\left(\Omega,a,b,c,d\right)}^{\text{pb}}}^{D}}\underline{f}_{\left(\Omega,a,b,c,d\right)}^{\text{pb}}$
on $D$. Therefore, in $\varphi^{-1}(\Omega_{-})\subset D$, we have
\[
f_{\left(\Omega,a,b,c,d\right)}^{\text{pb}}=f_{\left(\Omega_{-},a,b\right)}^{\text{pb}}\circ\varphi_{-}^{-1}\circ\varphi=e^{s_{f_{\left(\Omega_{-},a,b\right)}^{\text{pb}}}^{D_{-}}\circ\varphi_{-}^{-1}\circ\varphi}\underline{f}_{\left(\Omega_{-},a,b\right)}^{\text{pb}}\circ\varphi_{-}^{-1}\circ\varphi,
\]
where $\underline{f}_{\left(\Omega_{-},a,b\right)}^{\text{pb}}\circ\varphi_{-}^{-1}\circ\varphi$
is smooth up to any boundary segment $S'\Subset\varphi^{-1}(ab)$
and $s_{f_{\left(\Omega_{-},a,b\right)}^{\text{pb}}}^{D_{-}}\circ\varphi_{-}^{-1}\circ\varphi$
is in $W^{1,2}\left(\varphi^{-1}(\Omega_{-})\right)$-bounded (since
$s_{f_{\left(\Omega_{-},a,b\right)}^{\text{pb}}}^{D_{-}}\in W^{1,2}\left(D_{-}\right)$).
Therefore, by Propositions \ref{prop:rhh1} and \ref{prop:rhh2},
the integral $\underline{h}_{\left(\Omega,a,b,c,d\right)}^{\text{pb}}$
satisfies $(rh)_{h}^{\text{pb}}$ on $\varphi^{-1}(ab)$. We may repeat
this argument on the other three segments.

Now, note that for any $z\in\omega$ as in Theorem \ref{thm:levelset},
we have the bound
\begin{alignat*}{1}
\left\vert \underline{h}_{\left(\Omega,a,b,c,d\right)}(z)-\underline{h}_{\left(\Omega,a,b,c,d\right)}\left(da\right)\right\vert  & \leq\int_{C}\left\vert \underline{f}_{\left(\Omega,a,b,c,d\right)}\circ\varphi_{-}\right\vert ^{2}\left\vert dz\right\vert \\
 & =\int_{C}\left\vert e^{s_{f_{\left(\Omega,a,b,c,d\right)}}^{\Omega}\circ\varphi_{-}}f_{\left(\Omega,a,b,c,d\right)}\circ\varphi_{-}\right\vert ^{2}\left\vert dz\right\vert \\
 & =\int_{C}\left\vert e^{s_{f_{\left(\Omega,a,b,c,d\right)}}^{\Omega}\circ\varphi_{-}}f_{\left(\Omega_{-},a,b\right)}^{\text{pb}}\right\vert ^{2}\left\vert dz\right\vert \\
 & =\int_{C}\left\vert e^{s_{f_{\left(\Omega,a,b,c,d\right)}}^{\Omega}\circ\varphi_{-}+s_{f_{\left(\Omega_{-},a,b\right)}^{\text{pb}}}^{D_{-}}}\underline{f}_{\left(\Omega_{-},a,b\right)}^{\text{pb}}\right\vert ^{2}\left\vert dz\right\vert ,
\end{alignat*}
where $C\subset D_{-}$ is any rectifiable curve from $\varphi_{-}^{-1}(z)\in\partial D_{-}$
to $\varphi_{-}^{-1}\left(da\right)\subset\partial D_{-}$. $\underline{f}_{\left(\Omega_{-},a,b\right)}^{\text{pb}}$
is smooth near $\varphi_{-}^{-1}\left(d\right)$ by Corollary \ref{cor:2ptrhf},
and $s_{f_{\left(\Omega,a,b,c,d\right)}}^{\Omega}\circ\varphi_{-}+s_{f_{\left(\Omega_{-},a,b\right)}^{\text{pb}}}^{D_{-}}\in W^{1,2}\left(D_{-}\right)$.
By applying trace inequality as in (\ref{eq:sobolevholder}), we may
send the integral above to zero by letting $z\to d\in\partial\Omega$
and $\left\vert C\right\vert \to0$. Repeating this argument for $\left\vert \underline{h}_{\left(\Omega,a,b,c,d\right)}(z)-\underline{h}_{\left(\Omega,a,b,c,d\right)}\left(cd\right)\right\vert $
in $\Omega_{+}$, we obtain $\underline{h}_{\left(\Omega,a,b,c,d\right)}\left(da\right)=\underline{h}_{\left(\Omega,a,b,c,d\right)}\left(cd\right)$
as desired.

Given the existence of $k\in\mathbb{R}$ such that $f_{\left(\Omega,a,b,c,d\right)}^{m=0}=k\underline{f}_{\left(\Omega,a,b,c,d\right)}$,
estimation of $k$ proceeds exactly as in the proof of Corollary \ref{cor:2ptrhf},
this time deriving uniform bounds from the slit-strip $\mathbb{S}_{\chi_{m=0}}:=\left(\mathbb{R}+\left(0,i\right)\right)\setminus\left(\mathbb{R}_{<0}+\chi_{m=0}i\right)$
with, where $f_{\left(\Omega,a,b,c,d\right)}^{m=0}$ pulls back to
the constant function $1$ by Theorem \ref{thm:crit4pt}.
\end{proof}

\section{Asymptotic Analysis of $\chi_{m}$ \label{sec:Asymptotic-Analysis-of}}

Here, we carry out analysis of the behaviour of $\chi_{m}$ using
function theory in continuum.

\subsection{$\chi_{m}-\chi\asymp m$ as $m\to0$}

We start with the following observation. Given the square integral
$h_{\left(\Omega,a,b,c,d\right)}^{\text{pb}}$ on $D$ with some $m>0$,
$\chi_{m}$ may be recovered by the following procedure:
\begin{enumerate}
\item First, subtract the positive \emph{superharmonic part} $g(z):=\iint_{D}G_{\Omega}\left(z,z'\right)\Delta h_{\left(\Omega,a,b,c,d\right)}^{\text{pb}}(z')d^{2}z'$
from $h_{\left(\Omega,a,b,c,d\right)}^{\text{pb}}$;
\item Look at the new endpoint $\varphi^{-1}\left(d'\right)\in\left(\varphi^{-1}(c)\varphi^{-1}\left(a\right)\right)$
of the level line of $\chi_{m}$ in the resulting \emph{harmonic part},
which has moved towards $\varphi^{-1}(c)$;
\item given $\left(D,\varphi^{-1}(a),\varphi^{-1}(b),\varphi^{-1}(c),\varphi^{-1}(d')\right)$,
such harmonic function (thus $\chi_{m}$) is determined uniquely,
namely as the imaginary part of the unique map sending the domain
to a slit strip \cite[(6.7)]{chsm2012}
\end{enumerate}
Therefore, we need to study the location of $\varphi^{-1}(d')$: given
the locations of $\varphi^{-1}(a),\varphi^{-1}(b),\varphi^{-1}(c)$,
this is equivalent to studying the normal derivative of the harmonic
part at $\varphi^{-1}(d)$.

Concretely, assume as before we fix $D\subset\mathbb{H}$ with $\left(\varphi^{-1}(c)\varphi^{-1}\left(a\right)\right)\subset\mathbb{R}$.
In fact, we will assume that $\varphi^{-1}(d)=0$, $\left[-1,1\right]\subset\left(\varphi^{-1}(c)\varphi^{-1}\left(a\right)\right)$,
and $B_{1}(0)\cap\mathbb{H}\subset D$. We may fix the locations of
these three boundary points independent of $\left(\Omega,a,b,c,d\right)$.
Define the harmonic function $h_{'}:=h_{\left(D,\varphi^{-1}(a),\varphi^{-1}(b),\varphi^{-1}(c),\varphi^{-1}(d')\right)}^{m=0}=h_{\left(\Omega,a,b,c,d\right)}^{\text{pb}}-g(z)$:
the location of $\varphi^{-1}(d')<0$, where $\partial_{y}h_{'}=0$,
may be determined by the value of $\partial_{y}h_{'}(0)=-\partial_{y}g(0)$
since $\partial_{y}h_{\left(\Omega,a,b,c,d\right)}^{\text{pb}}(0=\varphi^{-1}(d))=0$,
to be shown in the proof below (see Figure \ref{fig:sec6-1}LT).

The following proposition gives the asympotic of $\partial_{y}h_{'}(0)$.
\begin{prop}
\label{prop:mtozero}Given $\left(\Omega,a,b,c,d\right)$ and the
pullback $\Omega\xrightarrow{\varphi}D$, the superharmonic part $g(z)$
has a normal derivative at $\varphi^{-1}(d)=0$, and
\[
\smallc\left(\Omega,a,b,c,d\right)\cdot m\leq\partial_{y}g(0)=-\partial_{y}h_{'}(0)\leq\const\left(\diam\Omega\right)\cdot m,
\]
for constants independent of $m<1$.
\end{prop}

\begin{proof}
Note, we have the conformal covariance
\[
\Delta h_{\left(\Omega,a,b,c,d\right)}^{\text{pb}}(z')=-4m\left\vert \varphi'(z')\right\vert \left\vert f_{\left(\Omega,a,b,c,d\right)}^{\text{pb}}(z')\right\vert ^{2}=\left\vert \varphi'(z')\right\vert ^{2}\Delta h_{\left(\Omega,a,b,c,d\right)}(\varphi(z')),
\]
and we can write
\[
g(z)=\iint_{\Omega}G_{\Omega}(\varphi^{-1}(z),z')\Delta h_{\left(\Omega,a,b,c,d\right)}(z')d^{2}z'.
\]
Since $\Delta h_{\left(\Omega,a,b,c,d\right)}(z')=m\left\vert f_{\left(\Omega,a,b,c,d\right)}(z')\right\vert ^{2}$
is less than $\const m\dist\left(z',\partial\Omega\right)^{-1}$ by
Proposition \ref{prop:creg-bulk}, continuity of $g\circ\varphi$
up to $\partial\Omega$ is guaranteed by \cite[Lemma A.2]{s-emb}
as in Lemma \ref{lem:chelkak}. That is, $g$ takes zero boundary
value at $\partial D$. Also note that by the same lemma $\max_{D}g(z)\leq\const m\diam\left(\Omega\right)$.

Now let us derive the upper bound. First we need to control the influence
of the possible singularities at $\varphi^{-1}(a),\varphi^{-1}(b),\varphi^{-1}(c)$.
We split
\[
g(z):=\iint_{D\setminus B_{1}}G_{D}\left(z,z'\right)\Delta h_{\left(\Omega,a,b,c,d\right)}^{\text{pb}}(z')d^{2}z'+\iint_{D\cap B_{1}}G_{D}\left(z,z'\right)\Delta h_{\left(\Omega,a,b,c,d\right)}^{\text{pb}}(z')d^{2}z'.
\]
The first term is harmonic in the ball $B_{1}=B_{1}(0)$ and vanishes
on $\partial D\cap B_{1}$, while both terms are bounded above by
$\max_{D}g(z)$ by nonnegativity. Therefore, by usual arguments (e.g.
Schwarz reflection), its normal derivative at $0$ exists, and is
bounded above by $\const\max_{D}g(z)\leq\const m\diam\left(\Omega\right)$.
Now it remains to study the second term near $0$.

Note that $z\mapsto G_{D}(z,z')$ is a nonpositive function, harmonic
away from $z'$, bounded from below by $\frac{1}{2\pi}\log\frac{\left\vert z-z'\right\vert }{\left\vert z-\bar{z'}\right\vert }$
from comparison on $\partial D$. Therefore, again from usual arguments,
we have that for $\epsilon\leq\frac{\left\vert z'\right\vert }{4}$, 
\begin{alignat*}{1}
\left\vert \frac{G_{D}(\epsilon i,z')-G_{D}(0,z')}{\epsilon}\right\vert =\left\vert \frac{G_{D}(\epsilon i,z')}{\epsilon}\right\vert \text{ and }\left\vert \partial_{y}G_{D}(yi,z')\vert_{y=0}\right\vert  & \leq\frac{\const}{\left\vert z'\right\vert },
\end{alignat*}
while for $\frac{\left\vert z'\right\vert }{4}<\epsilon$,
\[
\left\vert \frac{G_{D}(\epsilon i,z')-G_{D}(0,z')}{\epsilon}\right\vert =\left\vert \frac{G_{D}(\epsilon i,z')}{\epsilon}\right\vert \leq\frac{\const\left\vert \log\left\vert \epsilon i-z'\right\vert \right\vert }{\epsilon},
\]
with universal constants. For $z'\in D_{1}:=D\cap B_{1}$, we claim
that
\begin{alignat*}{1}
\left\vert \Delta h_{\left(\Omega,a,b,c,d\right)}^{\text{pb}}(z')\right\vert  & =4m\left\vert \varphi'(z')\right\vert \left\vert f_{\left(\Omega,a,b,c,d\right)}^{\text{pb}}(z')\right\vert ^{2}\\
 & =4m\left\vert e^{2s(z')}\right\vert \left\vert \varphi'(z')\right\vert \left\vert \underline{f}_{\left(\Omega,a,b,c,d\right)}^{\text{pb}}(z')\right\vert ^{2}\\
 & \leq4\const\left(\diam\Omega\right)m\left\vert e^{2s_{f_{\left(\Omega,a,b,c,d\right)}^{\text{pb}}}^{D}(z')}\right\vert \left\vert \varphi'(z')\right\vert \left\vert z'\right\vert .
\end{alignat*}
Indeed, $\left(f_{\left(D,\varphi^{-1}(a),\varphi^{-1}(b),\varphi^{-1}(c),\varphi^{-1}(d)\right)}^{m=0}\right)^{2}$
has simple zero at $z'=0$ by Remark \ref{rem:fc_zero}. The coefficient
of $w$ is bounded simply because it may be obtained from the second
derivative of the massless square integral $h_{\left(D,\varphi^{-1}(a),\varphi^{-1}(b),\varphi^{-1}(c),\varphi^{-1}(d)\right)}^{m=0}$
which is (again by e.g. Schwarz reflection) harmonic and bounded in
$B_{1}$. Then the same holds with $\left(\underline{f}_{\left(\Omega,a,b,c,d\right)}^{\text{pb}}\right)^{2}$
up to a factor of $\const\left(\diam\Omega\right)$ by (\ref{eq:4ptprop}).

Therefore, we have the following bounds for the integral over $D_{1}=D\cap B_{1}$:
\begin{alignat*}{1}
&\iint_{D_{1}\setminus B_{4\epsilon}}\left\vert \frac{G_{D}\left(\epsilon i,z'\right)}{\epsilon}\Delta h_{\left(\Omega,a,b,c,d\right)}^{\text{pb}}(z')\right\vert d^{2}z' \\
& \leq\const\left(\diam\Omega\right)\cdot m\iint_{D}\mathbf{1}_{D_{1}\setminus B_{4\epsilon}}\left\vert e^{2s_{f_{\left(\Omega,a,b,c,d\right)}^{\text{pb}}}^{D}(z')}\right\vert \left\vert \varphi'(z')\right\vert d^{2}z',\\
&\iint_{D\cap B_{4\epsilon}}\left\vert \frac{G_{D}\left(\epsilon i,z'\right)}{\epsilon}\Delta h_{\left(\Omega,a,b,c,d\right)}^{\text{pb}}(z')\right\vert d^{2}z' \\
& \leq\const\left(\diam\Omega\right)\cdot m\iint_{D\cap B_{4\epsilon}}\left\vert \log\left\vert \epsilon i-z'\right\vert \right\vert \left\vert e^{2s_{f_{\left(\Omega,a,b,c,d\right)}^{\text{pb}}}^{D}(z')}\right\vert \left\vert \varphi'(z')\right\vert d^{2}z'.
\end{alignat*}
It follows easily by H\"older that the first line is bounded by $\const\left(\diam\Omega\right)\cdot m$
and the second is negligible as $\epsilon\downarrow0$: for example,
the more complicated latter integral may be bounded by
\begin{alignat*}{1}
 & \leq\left(\iint_{D\cap B_{4\epsilon}(0)}\left\vert \log\left\vert \epsilon i-z'\right\vert \right\vert ^{4}d^{2}z'\right)^{1/4}\cdot \left(\iint_{D\cap B_{4\epsilon}(0)}\left\vert e^{2s_{f_{\left(\Omega,a,b,c,d\right)}^{\text{pb}}}^{D}(z')}\right\vert ^{4}d^{2}z'\right)^{1/4}\\
& \cdot \left(\iint_{D\cap B_{4\epsilon}(0)}\left\vert \varphi'(z')\right\vert ^{2}d^{2}z'\right)^{1/2}\\
 & \leq\epsilon^{1/2}\cdot\const\left(\diam\Omega\right)\cdot\diam\left(\Omega\right)\to0\text{ as }\epsilon\downarrow0.
\end{alignat*}
As usual, we bounded the $L^{4}$-norm of the exponential by $\const\left(\diam\Omega\right)$
using (\ref{eq:delsl2norm}) and Lemma \ref{lem:general_sobolev}.

Therefore, by dominated convergence as $\epsilon\downarrow0$, the
derivative exists and
\begin{alignat*}{1}
\partial_{y}\iint_{D_{1}}G_{D}\left(yi,z'\right)\Delta h_{\left(\Omega,a,b,c,d\right)}^{\text{pb}}(z')d^{2}z' & \vert_{y=0}=\iint_{D_{1}}\partial_{y}G_{D}\left(yi,z'\right)\vert_{y=0}\Delta h_{\left(\Omega,a,b,c,d\right)}^{\text{pb}}(z')d^{2}z'\\
 & \leq\const\left(\diam\Omega\right)\cdot m.
\end{alignat*}

To conclude, we argue that $\partial_{y}h_{\left(\Omega,a,b,c,d\right)}^{\text{pb}}(0=\varphi^{-1}(d))=0$
and thus $\partial_{y}h_{'}(0)=-\partial_{y}g(0)$. The normal derivative
of $h_{\left(\Omega,a,b,c,d\right)}^{\text{pb}}$ (i.e. the sum of
$g$ and the harmonic function $h_{'}$) at $0$ exists, which has
to be zero: $h_{\left(\Omega,a,b,c,d\right)}^{\text{pb}}(\epsilon i)-h_{\left(\Omega,a,b,c,d\right)}^{\text{pb}}(0)$
is given by a line integral over the segment $\left(0,\epsilon i\right)$
of $\left(f_{\left(\Omega,a,b,c,d\right)}^{\text{pb}}\right)^{2}$,
which is a product of a simple zero at $0$ and $e^{2s}$ term as
above. Since the latter has a trace which is in $L^{q}\left(0,\epsilon i\right)$
for any $q\in\left(1,\infty\right)$ by Lemma \ref{lem:general_sobolev},
the line integral has to be $o(\epsilon)$.

For the lower bound, we look at the values of $g(z)$ on the middle
part $L:=\left[-\frac{1}{6}+\frac{i}{2},\frac{1}{6}+\frac{i}{2}\right]$
of the top side of the rectangle $\left(-\frac{1}{2},\frac{1}{2}\right)+\left(0,\frac{i}{2}\right)$
(see Figure \ref{fig:sec6-1}LB); then the harmonic function which
coincides with $g$ on $L$ with zero boundary values elsewhere on
the rectangular boundary will be a lower bound for $g$. Accordingly,
by conformal mapping or otherwise, we may derive $\partial_{y}g(0)\geq\smallc\int_{L}g(z)\left\vert dz\right\vert $
for some universal constant $\smallc>0$. We will bound this line
integral of $g$ from below.

Consider the interior balls $B_{1/4}\left(i/2\right)\subset B_{1/2}\left(i/2\right)\subset D$.
Since $\left\vert \varphi'\right\vert $ is bounded below on $B_{1/4}(i/2)$
by some constant depending on $\left(\Omega,a,b,c,d\right)$, applying
H\"older inequality,
\begin{alignat*}{1}
\left\Vert \Delta h_{\left(\Omega,a,b,c,d\right)}^{\text{pb}}\right\Vert _{L^{1}(B_{1/4}\left(i/2\right))} & =\iint_{B_{1/4}\left(i/2\right)}4m\left\vert \varphi'(z')\right\vert \left\vert f_{\left(\Omega,a,b,c,d\right)}^{\text{pb}}\right\vert ^{2}(z')d^{2}z'\\
=\iint_{B_{1/4}(i/2)}4m\left\vert \varphi'(z')\right\vert  & \left\vert e^{2s_{f_{\left(\Omega,a,b,c,d\right)}^{\text{pb}}}^{D}(z')}\right\vert \left\vert f_{\left(D,\varphi^{-1}(a),\varphi^{-1}(a),\varphi^{-1}(a),\varphi^{-1}(a)\right)}^{m=0}\right\vert ^{2}(z')d^{2}z'\\
\geq\smallc(\Omega,a,b,c,d)\cdot m & \cdot\frac{\left(\iint_{B_{1/4}(i/2)}\left\vert f_{\left(D,\varphi^{-1}(a),\varphi^{-1}(a),\varphi^{-1}(a),\varphi^{-1}(a)\right)}^{m=0}\right\vert (z')d^{2}z'\right)^{2}}{\iint_{B_{1/4}(i/2)}\left\vert e^{-2s_{f_{\left(\Omega,a,b,c,d\right)}^{\text{pb}}}^{D}(z')}\right\vert d^{2}z'}.
\end{alignat*}
The fraction is bounded below by (some other) $\smallc(\Omega,a,b,c,d)>0$:
$\left\vert f_{\left(D,\varphi^{-1}(a),\varphi^{-1}(a),\varphi^{-1}(a),\varphi^{-1}(a)\right)}^{m=0}\right\vert (z')$
is bounded below in $B_{1/4}(i/2)$ by some $\smallc\left(\Omega,a,b,c,d\right)>0$
(recall from Theorem \ref{thm:crit2pt} that its pullback to the strip
is identically $1$) while the denominator is bounded above by some
$\const(\diam\Omega)>0$ by Lemma \ref{lem:general_sobolev} as above.
In addition, noting that $G_{B_{1/2}(i/2)}(z,z')<-\smallc<0$ if $z\in L$
and $z'\in B_{1/4}\left(i/2\right)$, we have
\[
g(z)=\iint_{B_{1/2}(i/2)}G_{B_{1/2}(i/2)}(z,z')\Delta h_{\left(\Omega,a,b,c,d\right)}^{\text{pb}}(z')d^{2}z'\geq\smallc\cdot\left\Vert \Delta h_{\left(\Omega,a,b,c,d\right)}^{\text{pb}}\right\Vert _{L^{1}(B_{1/4}\left(i/2\right))},
\]
so $\int_{L}g(z)\left\vert dz\right\vert \geq\smallc\left(\Omega,a,b,c,d\right)\cdot m$
as desired.
\end{proof}
\begin{cor}
\label{cor:mtozero}Suppose the marked domain $\left(\Omega,a,b,c,d\right)$
is given. We have
\begin{alignat*}{1}
\smallc\left(\Omega,a,b,c,d\right)\cdot m & \leq\chi_{m}-\chi_{m=0}\leq\const\left(\diam\Omega\right)\cdot m,\\
\smallc\left(\Omega,a,b,c,d\right)\cdot m & \leq\mathrm{p}_{m}-\mathrm{p}_{m=0}\leq\const\left(\diam\Omega,\chi_{0}\right)\cdot m,
\end{alignat*}
for constants independent of, say, $m\in\left[0,1\right]$.
\end{cor}

\begin{proof}
Thanks to Proposition \ref{prop:mtozero}, it suffices to show that
$\chi_{m}-\chi_{m=0}$ is within a bounded factor of $-\partial_{y}h_{'}(0)$.
Consider the boundary values of $h_{'}$ and $\omega:=h_{\left(D,\varphi^{-1}(a),\varphi^{-1}(c)\right)}^{m=0}$
(which coincides with the harmonic measure $\omega(\cdot,\left(\varphi^{-1}(c)\varphi^{-1}(a)\right),D)$):
\[
\begin{cases}
h_{'}=0,\omega=0 & \text{on }\left(\varphi^{-1}(a)\varphi^{-1}(b)\right);\\
h_{'}=1,\omega=0 & \text{on }\left(\varphi^{-1}(b)\varphi^{-1}(c)\right);\\
h_{'}=\chi_{m},\omega=1 & \text{on }\left(\varphi^{-1}(c)\varphi^{-1}(a)\right).
\end{cases}
\]
The normal derivative $\partial_{y}\omega(0)$ is strictly negative
by Hopf lemma, so $h_{'}-\frac{\partial_{y}h_{'}(0)}{\partial_{y}\omega(0)}\omega$
has zero normal derivative at $0$. So it is the harmonic function
$h_{0}:=h_{\left(D,\varphi^{-1}(a),\varphi^{-1}(b),\varphi^{-1}(c),\varphi^{-1}(d)\right)}^{m=0}$
that solves precisely the boundary value problem for the $4$-point
massless observable, and $\chi_{m}-\frac{\partial_{y}h_{'}(0)}{\partial_{y}\omega(0)}=\chi_{0}$.
Since $\omega$ is fixed independent of $\left(\Omega,a,b,c,d\right)$,
$\chi_{m}-\chi_{0}$ is within a (universal) bounded factor of $-\partial_{y}h_{'}(0)$.
Then the analogue for $\mathrm{p}_{m}-\mathrm{p}_{m=0}$ is straightforward
from the definition of $\chi_{m}$.

\begin{figure}
\centering
\includegraphics[width=0.7\paperwidth]{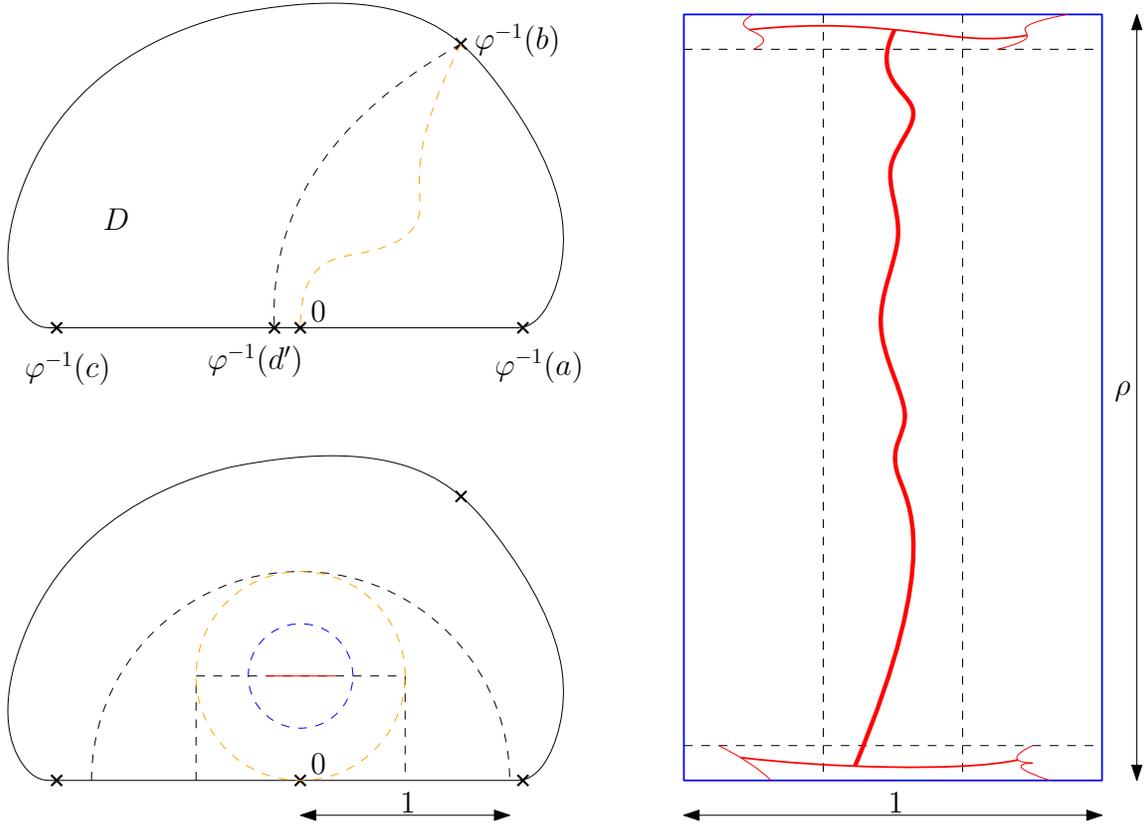}

\caption{(Left-top) Level lines for $h_{\left(\Omega,a,b,c,d\right)}^{\text{pb}}$
(dashed orange) and the harmonic part $h_{'}$ (dashed black), endpoint
of latter being $\varphi^{-1}(d')$. (Left-bottom) Concentric circles around
$i/2$ (dashed), and the line $L$ (solid red). (Right) Proof of the upper
bound in Corollary \ref{cor:corlength} by positive correlation of
crossing events.}
\label{fig:sec6-1}
\end{figure}
\end{proof}

\subsection{\label{subsec:mtoinfinity}$\chi_{m}\to1$ as $m\to\infty$}

We now treat the case where the mass tends to infinity. This is crucial
in our analysis of the characteristic length (\ref{eq:corlength}):
sending $m\to\infty$ is equivalent to taking larger and larger rectangles. 
\begin{prop}
\label{prop:mtoinfty}As $m\to\infty$, $\chi_{m}\to1$.
\end{prop}

\begin{proof}
For conciseness, we use a compactness argument: similar ideas to Proposition
\ref{prop:mtozero} may be used to derive more quantitative bounds.
Note that, on any compact subset of $\Omega$, $\left\{ h_{\left(\Omega,a,b,c,d\right)}\right\} _{m\geq0}$
is $C^{1}$ bounded by Proposition \ref{prop:creg-bulk}; therefore,
as in Remark \ref{rem:subsequence}, by diagonalising we may assume
that $h_{\left(\Omega,a,b,c,d\right)}\xrightarrow{m\to\infty}g$.
But then Proposition \ref{prop:mtozero} in fact implies that $g$
must be a constant. Given that $h_{\left(\Omega,a,b,c,d\right)}$
is bounded below by its harmonic part which is continuous up to $(bc)$
where the boundary value is $1$, the constant must be $g=1$.

Suppose, along some subsequence of $m\to\infty$, $\sup\chi_{m}\leq1-\epsilon<1$.
As before, take a pullback domain $D\subset\mathbb{H}$ with $\varphi^{-1}\left(d\right)=0$
and $\left[0,1\right]\subset\left(\varphi^{-1}(c)\varphi^{-1}(a)\right)$.
Take small enough $\rho=\rho(\epsilon)>0$ so that on the rectangle
$R(\rho):=\left(0,1\right)+\left(0,\rho i\right)\subset D$, the harmonic
function $h_{1}$ with boundary values
\[
\begin{cases}
h_{1}=1-\frac{\epsilon}{2} & \text{on }(\rho i,1+\rho i);\\
h_{1}=0 & \text{on }\left(0,\rho i\right)\cup\left(1,1+\rho i\right);\\
h_{1}=\chi_{m} & \text{on }\left(0,1\right);
\end{cases}
\]
has positive inner normal derivatives near $\frac{1}{2}$. Then by
taking large enough $m=m(\rho)$ so that $\min_{(\rho i,1+\rho i)}h_{\left(\Omega,a,b,c,d\right)}^{\text{pb}}>1-\frac{\epsilon}{2}$,
the superharmonic function $h_{\left(\Omega,a,b,c,d\right)}^{\text{pb}}$
is bounded below by $h_{1}$ and thus $h_{\left(\Omega,a,b,c,d\right)}^{\text{pb}}$
have values strictly greater than $\chi_{m}$ near $\frac{1}{2}$.
This contradicts $(rh)_{h}^{\text{pb}}$ on $\left(\varphi^{-1}(d)\varphi^{-1}(a)\right)$.
\end{proof}
The above proposition directly implies Corollary \ref{cor:corlength}
for the specific case of 4-point Dobrushin boundary condition. Specifically,
take the rectangle $\Omega:=R(\rho)$ where 4 marked boundary corners
are taken at the 4 corners of $R(\rho)$. Then, given $\epsilon$,
one may find $m$ large enough such that $\mathrm{p}_{m}\geq1-\frac{\epsilon}{2}$
by Proposition \ref{prop:mtoinfty}. If the upper bound doesn't hold,
we may extract a sequence $\delta_j\to 0, q_{j}\delta_{j}^{-1}\to\infty$;
in fact by monotonicity (in $q$) we may instead assume $q_{j}\delta_{j}^{-1}\to m$,
along which $\mathbb{P}\left[\stackrel{\leftrightarrow}{R(\rho)^{\delta}}\right]<1-\epsilon$ and thus is a contradiction to the $\mathrm{p}_m$ estimate.
Therefore we have the following:
\begin{cor}
\label{cor:dob-correlation}The upper bound of (\ref{eq:corlength})
is correct for the 4-point Dobrushin boundary condition.
\end{cor}

With the help of the RSW-type estimates (Theorem \ref{thm:rsw}) and
the FKG inequality, we now extend Corollary \ref{cor:dob-correlation}
to general boundary conditions.
\begin{proof}[Proof of the upper bound in Corollary \ref{cor:corlength}]
By monotonicity (with respect to the boundary condition) and arguments
as in above, for the upper bound it suffices to show that there is
some large $m=m(\epsilon)$ such that for small $\delta<\delta_{0}(m)$
we have $\mathbb{P}\left[\stackrel{\leftrightarrow}{R(\rho)^{\delta}}\right]\geq1-\frac{\epsilon}{2}$
with the dual-wired boundary condition. We show this by intersecting
multiple (but a finite number of) crossing events (positively correlated
due to FKG inequality), each of which, conditioning on predecessors,
has a probability which can be made arbitrarily close to $1$.

Specifically, consider the following events, where we drop the $\delta$
superscript for conciseness (see Figure \ref{fig:sec6-1}R):
\begin{enumerate}
\item For $0<\alpha\ll\rho$ and $j=0,1,2$, the rectangles $(\frac{j}{3},\frac{j+1}{3})+(0,\alpha i)$
in the bottom and $(\frac{j}{3},\frac{j+1}{3})+(i-\alpha i,i)$ on
top are crossed vertically;
\item The middles of the top/bottom three $(\frac{1}{3},\frac{2}{3})+(0,\alpha i)$
and $(\frac{1}{3},\frac{2}{3})+(i-\alpha i,i)$ are horizontally crossed;
\item The middle third $(\frac{1}{3},\frac{2}{3})+(0,\rho i)$ of $R(\rho)$
is vertically crossed.
\end{enumerate}
The probability of (1) may be made arbitrarily close to $1$ by setting
small enough aspect ratio $\alpha=\alpha(\epsilon)$, by the RSW-type
estimate at \emph{criticality} and then monotonicity in $q$. Conditionally
on (1), the probability of (2) may be made arbitrarily close to $1$
by setting large enough $m=m_{1}(\alpha,\epsilon)$ and small enough
$\delta<\delta_{0}(m_{1})$, from Proposition \ref{prop:mtoinfty}
and comparing with the 4-point Dobrushin boundary condition on the
rectangles $(0,1)+(0,\alpha i)$ and $\left(0,1\right)+\left(i-\alpha i,i\right)$.
Conditionally on (1) and (2), the probability of (3) may be made arbitrarily
close to $1$ by setting large enough $m=m_{2}(m_{1},\alpha,\epsilon)>m_{1}(\alpha,\epsilon)$
and small enough $\delta<\delta_{0}(m_{2})$, again from Proposition
\ref{prop:mtoinfty} and comparison with the 4-point Dobrushin boundary
condition.
\end{proof}

\appendix

\section{Complex Analytic Computations}

\subsection{Massive S-holomorphicity}

In this section, we study the main discrete relations which serve
as the counterpart to holomorphicity in the continuum: the (massive)
s-holomorphicity. Here we supply the calculations needed to verify
the consequences of massive s-holomorphicity: namely the properties
of the square integral and that s-holomorphicity is indeed a discretisation
of the continuous equation (\ref{eq:mchol}).

\subsubsection{S-holomorphicity on Half-Rhombi}

Recall that our discrete domains are composed of rhombi and boundary
half-rhombi; the half-rhombi (both in the interior and on the boundary)
provide a convenient setting on which s-holomorphicity may be rephrased
as intergrability conditions. Consider the dual half-rhombus $T=\left\langle uww_{z}\right\rangle $
around which an s-holomorphic function $F$ may be evaluated at $\xi^{z}:=\left\langle uw\right\rangle $,
$\xi_{z}:=\left\langle uw_{z}\right\rangle $, and $z=\left\langle ww_{z}\right\rangle \in\lozenge^{*}$
(see Figure \ref{fig:app}). We will be motivated by 'contour integrals'
around $T$: (weighted) sums of function values on $\xi^{z},z,\xi_{z}$
with direction factors $\nu(\xi^{z}),\nu_{T}(z):=\frac{w_{z}-w}{\left\vert w_{z}-w\right\vert },-\nu(\xi_{z})$.

Denote the angle surplus $\theta_{z}^{\dagger}=\hat{\theta}_{z}-\bar{\theta}_{z}$.
The s-holomorphic projection relations (\ref{eq:shol}) are
\begin{alignat*}{1}
F(z)-ie^{i\theta_{z}^{\dagger}}\nu^{-1}(\xi)\overline{F(z)} & =2e^{i\frac{\theta_{z}^{\dagger}}{2}}F(\xi^{z});\\
F(z)-ie^{-i\theta_{z}^{\dagger}}\nu^{-1}(\xi_{z})\overline{F(z)} & =2e^{-i\frac{\theta_{z}^{\dagger}}{2}}F(\xi_{z}).
\end{alignat*}

Immediately we have the 'contour integral' of $F$: (using $e^{-i\frac{\theta_{z}^{\dagger}}{2}}\nu(\xi^{z})-e^{i\frac{\theta_{z}^{\dagger}}{2}}\nu(\xi_{z})=-2\sin\frac{\hat{\theta}_{z}+\bar{\theta}_{z}}{2}\nu_{T}(z)$)
\begin{equation}
F(\xi^{z})\nu(\xi^{z})+\sin\frac{\hat{\theta}_{z}+\bar{\theta}_{z}}{2}F(z)\nu_{T}(z)-F(\xi_{z})\nu(\xi_{z})=\sin\frac{\theta_{z}^{\dagger}}{2}\overline{F(z)}.\label{eq:1}
\end{equation}

Alternatively, we may eliminate $\overline{F(z)}$: (using $e^{-i\theta_{z}^{\dagger}}\nu(\xi^{z})-e^{i\theta_{z}^{\dagger}}\nu(\xi_{z})=-2\sin\hat{\theta}_{z}\nu_{T}(z)$),
\begin{equation}
\sin\hat{\theta}_{z}F(z)\nu_{T}(z)  =e^{i\frac{\theta_{z}^{\dagger}}{2}}\nu(\xi_{z})F(\xi_{z})-e^{-i\frac{\theta_{z}^{\dagger}}{2}}\nu(\xi^{z})F(\xi^{z}),\label{eq:reconst}
\end{equation}
which, once squared, gives (recall $\nu(\xi_{z})=-ie^{i\bar{\theta}_{z}}\nu_{T}(z)$,
$\nu(\xi^{z})=-ie^{-i\bar{\theta}_{z}}\nu_{T}(z)$)
\begin{alignat*}{1}
 & \left(\sin\hat{\theta}_{z}F(z)\nu_{T}(z)\right)^{2}\\
= & e^{i\theta_{z}^{\dagger}}\nu(\xi_{z})^{2}F(\xi_{z})^{2}+e^{-i\theta_{z}^{\dagger}}\nu(\xi^{z})^{2}F(\xi^{z})^{2}-2\nu(\xi)\nu(\xi_{z})F(\xi_{z})F(\xi)\\
= & -\nu_{T}(z)\left(ie^{i\hat{\theta}_{z}}F(\xi_{z})^{2}\nu(\xi_{z})+ie^{-i\hat{\theta}_{z}}F(\xi^{z})^{2}\nu(\xi^{z})-2\nu_{T}(z)F(\xi_{z})F(\xi)\right).
\end{alignat*}

Dividing by $\nu_{T}(z)$ and rearranging, we get a 'contour integral'
of $F^{2}$ (noting phases of $F(\xi^{z}),F(\xi_{z})$)
\begin{alignat*}{1}
 & F(\xi^{z})^{2}\nu(\xi^{z})+\sin\hat{\theta}_{z}F(z)^{2}\nu_{T}(z)-F(\xi_{z})^{2}\nu(\xi_{z})\\
= & \frac{2\nu_{T}(z)F(\xi_{z})F(\xi^{z})-i\cos\hat{\theta}_{z}F(\xi_{z})^{2}\nu(\xi_{z})-i\cos\hat{\theta}_{z}F(\xi^{z})^{2}\nu(\xi^{z})}{\sin\hat{\theta}_{z}}\\
= & \frac{2\nu_{T}(z)F(\xi_{z})F(\xi^{z})-\cos\hat{\theta}_{z}\left(\left\vert F(\xi_{z})\right\vert ^{2}+\left\vert F(\xi^{z})\right\vert ^{2}\right)}{\sin\hat{\theta}_{z}}\in\mathbb{R}.
\end{alignat*}

Noting the phase of $F(\xi^{z})$ and $F(\xi_{z})$, we may re-write
the above in real and imaginary parts, both of which are useful. First,
since both values are projections of $F(z)$, we may naturally compute
$d:=\arg F(\xi_{z})-\arg F(\xi^{z})=-\bar{\theta}_{z},\pi-\bar{\theta}_{z}$.
Define $\left[\da_{u}F\right](z)=\cos\left(d+\bar{\theta}_{z}\right)\in\left\{ \pm1\right\} $,
chosen to be $-1$ if either value of $F$ is zero. Then we have
\begin{alignat}{1}
\re\left[\sin\hat{\theta}_{z}F(z)^{2}\nu_{T}(z)\right] & =\frac{2\left[\da_{u}F\right](z)\left\vert F(\xi_{z})\right\vert \left\vert F(\xi^{z})\right\vert -\cos\hat{\theta}_{z}\left(\left\vert F(\xi_{z})\right\vert ^{2}+\left\vert F(\xi^{z})\right\vert ^{2}\right)}{\sin\hat{\theta}_{z}}\label{eq:refsq}\\
\imm\left[\sin\hat{\theta}_{z}F(z)^{2}\nu_{T}(z)\right] & =\imm\left[-F(\xi^{z})^{2}\nu(\xi^{z})+F(\xi_{z})^{2}\nu(\xi_{z})\right]=\left\vert F(\xi^{z})\right\vert ^{2}-\left\vert F(\xi_{z})\right\vert ^{2}.\label{eq:imfsq}
\end{alignat}

\begin{rem}
\label{rem:fsq_duality}To have a complete treatment, we need to do
the analogue on a primal half-rhombus $T^{'}:=\left\langle uwu_{z}\right\rangle $.
Thanks to duality, $iF$ could be said to be s-holomorphic on the
same rhombus with the primal and dual vertices interchanged, as long
as we make the corresponding adjustments in (\ref{eq:shol}): $\bar{\theta}_{z}\to\frac{\pi}{2}-\bar{\theta}_{z}$,
$\hat{\theta}_{z}\to\frac{\pi}{2}-\hat{\theta}_{z}$. Carefully applying
the above, we have ($\xi_{z}':=\left\langle u_{z}w\right\rangle ,\nu_{T'}(z):=i\nu_{T}(z)$)
\begin{alignat*}{1}
F(\xi^{z})\nu(\xi^{z})+\cos\frac{\hat{\theta}_{z}+\bar{\theta}_{z}}{2}F(z)\nu_{T'}(z)-F(\xi_{z}')\nu(\xi_{z}')= & \sin\frac{\theta_{z}^{\dagger}}{2}\overline{F(z)},\\
-\text{Re}\sin\hat{\theta}_{z}\nu_{T'}(z)F(z)^{2}+\left\vert F(z)\right\vert ^{2}=\left\vert F(\xi^{z})\right\vert ^{2} & +\left\vert F(\xi_{z}')\right\vert ^{2},
\end{alignat*}
and
\begin{alignat*}{1}
\re\left[\cos\hat{\theta}_{z}F(z)^{2}\nu_{T'}(z)\right] & =\frac{2\left[\da_{w}F\right](z)\left\vert F(\xi'_{z})\right\vert \left\vert F(\xi^{z})\right\vert -\sin\hat{\theta}_{z}\left(\left\vert F(\xi'_{z})\right\vert ^{2}+\left\vert F(\xi^{z})\right\vert ^{2}\right)}{\cos\hat{\theta}_{z}},\\
\imm\left[\cos\hat{\theta}_{z}F(z)^{2}\nu_{T'}(z)\right] & =\imm\left[-F(\xi^{z})^{2}\nu(\xi^{z})+F(\xi_{z}')^{2}\nu(\xi_{z}')\right]=\left\vert F(\xi^{z})\right\vert ^{2}-\left\vert F(\xi_{z}')\right\vert ^{2}.
\end{alignat*}
where $\left[\da_{w}F\right](z)$ is defined similarly using the argument
turning from $\xi'_{z}$ to $\xi^{z}$.
\end{rem}

\subsubsection{$\bar{\partial}$ Calculations}

Given the computation on half-rhombi given in the previous section,
we may combine them to calculate contour integrals involving only
values on edges $z\in\lozenge$, which scale to 'physical' contour
integrals in the continuum. First, we will give exact calculations
for some auxiliary operators $\bar{\partial}_{1,2,3}^{\delta}$ which
we will introduce now; the terms arising from the difference $\bar{\partial}^{\delta}-\bar{\partial}_{1,2,3}^{\delta}$
will be estimated to be small for our purposes. Define (recall $2\mu^{\delta}(u)$
is the total area of the half-rhombi $T$ around $u$)
\begin{alignat*}{1}
\bar{\partial}_{1}^{\delta}F_{0}(u) & :=\frac{1}{2i\mu_{\Gamma}^{\delta}(u)}\sum_{z\sim u}\frac{\sin\frac{\hat{\theta}_{z}+\bar{\theta}_{z}}{2}}{\sin\bar{\theta}_{z}}(w_{z}-w)F_{0}(z),\\
\bar{\partial}_{2}^{\delta}F_{0}(u) & :=\frac{1}{2i\mu_{\Gamma}^{\delta}(u)}\sum_{z\sim u}\frac{\sin^{2}\hat{\theta}_{z}}{\sin^{2}\bar{\theta}_{z}}(w_{z}-w)F_{0}(z),\\
\bar{\partial}_{3}^{\delta}F_{0}(u) & :=\frac{1}{2i\mu_{\Gamma}^{\delta}(u)}\sum_{z\sim u}\frac{\cos\hat{\theta}_{z}}{\cos\bar{\theta}_{z}}(w_{z}-w)F_{0}(z),
\end{alignat*}
on $u\in\Gamma$ for functions $F_{0}$ defined on $\lozenge$. Both
degenerate to $\bar{\partial}^{\delta}$ when $k=0$. Then, summing
(\ref{eq:1}), (\ref{eq:refsq}) around $u$, it remains to study
the exact expressions
\[
\bar{\partial}_{1}^{\delta}F(u)=\frac{1}{2i\mu_{\Gamma}^{\delta}(u)}\sum_{z\sim u}2\delta\sin\frac{\theta_{z}^{\dagger}}{2}\overline{F(z)},
\]
and
\begin{equation}
\bar{\partial}_{2}^{\delta}F^{2}(u)=\frac{2\delta}{2i\mu_{\Gamma}^{\delta}(u)}\sum_{z\sim u}\frac{2\nu_{T}(z)F(\xi_{z})F(\xi^{z})-\cos\hat{\theta}_{z}\left(\left\vert F(\xi_{z})\right\vert ^{2}+\left\vert F(\xi^{z})\right\vert ^{2}\right)}{\sin\bar{\theta}_{z}}.\label{eq:massive_sum}
\end{equation}

We start by showing that the massive s-holomorphic functions indeed
do converge to massive holomorphic functions.
\begin{lem}
\label{lem:dmhol}If a family $\left\{ F^{\delta}\right\} _{\delta>0}$
of s-holomorphic functions on $\Omega_{\delta}$ converges locally
uniformly to a locally Lipschitz continuous function $f$ on $\Omega$,
then $f$ is massive holomorphic.
\end{lem}

\begin{proof}
Because of Remark \ref{rem:cgreens}, it suffices to show that for
any ball $\overline{B_{r}}\subset\Omega$, we have (\ref{eq:mchol})
in the areolar derivative sense:
\begin{equation}
\oint_{\partial B_{r}}fdz=2i\iint_{B_{r}}-mi\bar{f}d^{2}z.\label{eq:areolar_mchol}
\end{equation}

Find a simple path $\partial B_{r}^{\delta}$ of dual edges converging
in Hausdorff distance to $\partial B_{r}$ as $\delta\downarrow0$
and consider the discrete domain $B_{r}^{\delta}$ within (constructed,
e.g., as the boundary of a discrete ball thanks to bounded angle property).
Then by definition (\ref{eq:dwrit_def}) and telescoping (as in the
proof of discrete divergence theorem),
\begin{alignat*}{1}
\sum_{z\in\partial B_{r}^{\delta}}\frac{\sin\frac{\hat{\theta}_{z}+\bar{\theta}_{z}}{2}}{\sin\bar{\theta}_{z}}(w_{z}-w)F(z) & =\sum_{u\in\Gamma\cap B_{r}^{\delta}}2i\mu_{\Gamma}^{\delta}(u)\bar{\partial}_{1}^{\delta}F(u)\\
 & =\sum_{z'\in\lozenge\cap B_{r}^{\delta}}4\delta\sin\frac{\theta_{z'}^{\dagger}}{2}\overline{F(z')},
\end{alignat*}
where the factor of $2$ comes from the fact that any $z'$ is counted
twice from two incident $u\in\Gamma$. As $\delta\downarrow0$, it's
clear that the sums on LHS and RHS respectively converge to those
of (\ref{eq:areolar_mchol}).
\end{proof}
Now we move on to a more involved calculation. We need to study $\bar{\partial}_{2}^{\delta}F^{2}(u)$
in (\ref{eq:massive_sum}),and $\bar{\partial}_{3}^{\delta}F^{2}(u)$,
which corresponds exactly to the Laplacian with behaviour stated in
(\ref{eq:hlap_primal}). Sum of the same type was studied in \cite{chsm2012},
but the abstract angles $\hat{\theta}_{z}$ here do not sum to $\pi$
around a vertex $u$, giving rise to an $L^{2}$ -term of $F$.
\begin{prop}
\label{prop:fsq_zbar}There are real quantities $A\left(\left\{ \bar{\theta}_{z}\right\} ,m\right)$,
$B(\left\{ F(\xi_{z})\right\} _{z\sim u})\geq0$ such that
\begin{alignat}{1}
2i\delta^{-1}\mu_{\Gamma}^{\delta}(u)\bar{\partial}_{2}^{\delta}F^{2}(u) & =\sum_{z\sim u}\frac{2\left(\cos\bar{\theta}_{z}-\cos\hat{\theta}_{z}\right)}{\sin\bar{\theta}_{z}}\left(\left\vert F(\xi_{z})\right\vert ^{2}+\left\vert F(\xi^{z})\right\vert ^{2}\right)-B(\left\{ F(\xi_{z})\right\} _{z\sim u}),\label{eq:zbartwo_fsq}\\
2i\delta^{-1}\mu_{\Gamma}^{\delta}(u)\bar{\partial}_{3}^{\delta}F^{2}(u) & =\sum_{z\sim u}A\left(\bar{\theta}_{z},\hat{\theta}_{z}\right)\left(\left\vert F(\xi_{z})\right\vert ^{2}+\left\vert F(\xi^{z})\right\vert ^{2}\right)-B(\left\{ F(\xi_{z})\right\} _{z\sim u}),\label{eq:zbarthree_fsq}
\end{alignat}
where $A\left(\left\{ \bar{\theta}_{z}\right\} ,m\right)\apprle m\delta\sin2\bar{\theta}_{z}$
with asymptotic not dependent on the uniform bound $\eta$.
\end{prop}

\begin{proof}
As in the proof of Proposition \ref{prop:s-hol}, define $x_{\xi_{z}}:=\left(i\nu(\xi_{z})\right)^{1/2}F(\xi_{z})$,
where $\left(i\nu\left(\xi_{z}\right)\right)^{1/2}$ varies continuously
as we turn around $u$ once (i.e. the 'cut' is placed between the
last and the first corners). The proof of \cite[Proposition 3.6]{chsm2012}
proves that the quadratic form $Q$ (which we take to be $\frac{1}{2}B(\left\{ F(\xi_{z})\right\} _{z\sim u})$)
\[
\sum_{z\sim u}\frac{\cos\bar{\theta}_{z}\left(\left\vert F(\xi_{z})\right\vert ^{2}+\left\vert F(\xi^{z})\right\vert ^{2}\right)-2\nu_{T}(z)F(\xi_{z})F(\xi^{z})}{\sin\bar{\theta}_{z}}=:Q_{\left\{ \overline{\theta}_{z}\right\} _{z\sim u}}^{\deg u}\left(\left\{ x_{\xi_{z}}\right\} _{z\sim u}\right)
\]
is always nonnegative, no matter what the real numbers $x_{\xi_{z}}$
are. This establishes (\ref{eq:zbartwo_fsq}).

Now it remains to control $\left(\bar{\partial}_{3}^{\delta}-\bar{\partial}_{2}^{\delta}\right)F^{2}(u)$.
But by (\ref{eq:massivetheta})

\begin{alignat*}{1}
2i\delta^{-1}\mu_{\Gamma}^{\delta}(u)\left(\bar{\partial}_{3}^{\delta}-\bar{\partial}_{2}^{\delta}\right)F^{2}(u) & =\delta^{-1}\sum_{z\sim u}\left(\frac{\cos\hat{\theta}_{z}}{\cos\bar{\theta}_{z}}-\frac{\sin^{2}\hat{\theta}_{z}}{\sin^{2}\bar{\theta}_{z}}\right)(w_{z}-w)F^{2}(z)\\
 & =O(m\delta\sin2\bar{\theta}_{z})\sum_{z\sim u}\left(\left\vert F(\xi_{z})\right\vert ^{2}+\left\vert F(\xi^{z})\right\vert ^{2}\right).
\end{alignat*}
\end{proof}
\begin{figure}
\centering
\includegraphics[width=0.7\paperwidth]{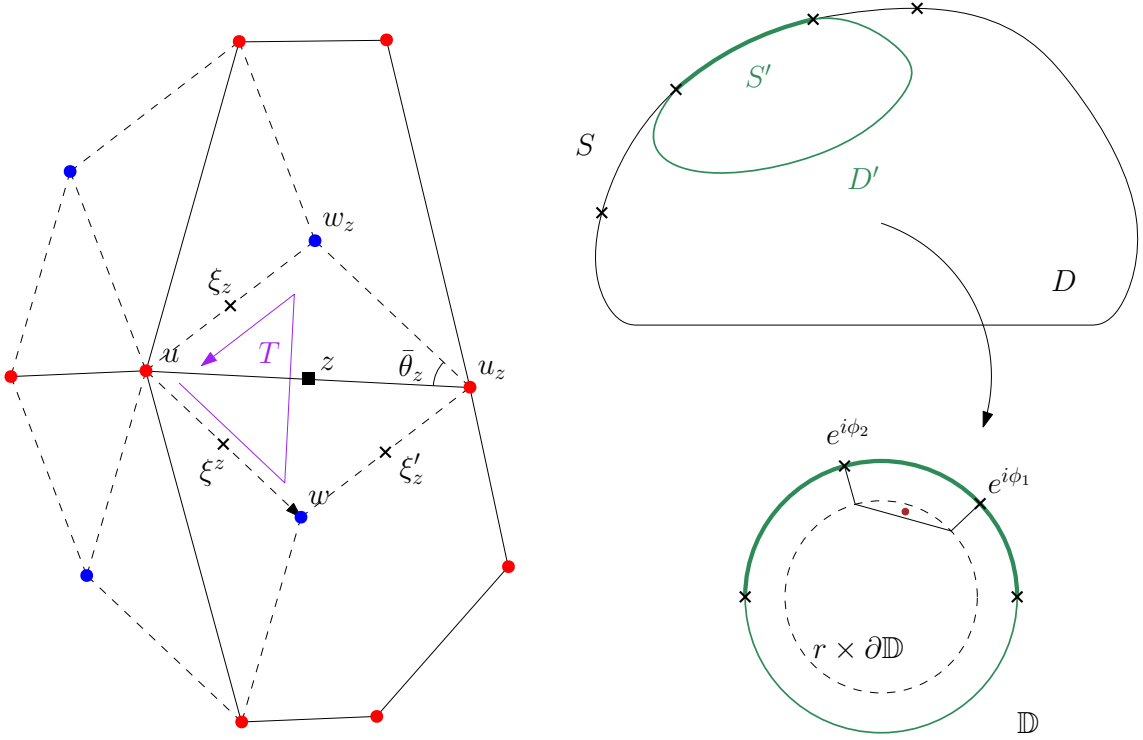}

\caption{(Left) Local notation given the primal vertex $u$ and an incident edge
$z$; dual half-rhombus $T$ and discrete integration contour around
$T$. (Right) Setting for the proof of Proposition \ref{prop:rhh1}: note
that the crescent-shaped domain (brown bullet) has a fixed shape.}
\label{fig:app}
\end{figure}

\subsection{Massive Cauchy Formulae}

In this section, we recall the Cauchy integral formula for the continuum
massive holomorphic functions, and construct its discrete counterpart. The discrete version is also a generalisation of its critical counterpart, namely \cite[Lemma A.6]{chsm2012}.

\subsubsection{Continuous Case}

Recall that for $l\in\mathbb{N}$ the modified Bessel functions of
the second kind $K_{l}(r)$ is the radial solution to the equation
$\Delta_{r,\phi}\left[e^{-il\phi}K_{l}(r)\right]=\left[e^{-il\phi}K_{l}(r)\right]$
with the asymptotics $K_{l}(r)\sim\frac{\left(l-1\right)!}{2}\left(\frac{1}{2}r\right)^{-l}$
(while $K_{0}(r)\sim-\ln r$) as $r\downarrow0$ and $K_{l}(r)\sim\sqrt{\frac{\pi}{2r}}e^{-r}$
as $r\to\infty$. In particular, $\left\vert K_{l}(r)\right\vert \leq\const_{l}r^{-l}$.

They can be combined in order to construct massive holomorphic functions
with the same asymptotics (e.g. \cite[(2.16)]{par19}), of which we
consider the following massive versions of the Cauchy kernel $\frac{1}{z'}$.
Recalling the radial derivative $\bar{\partial}g=\frac{e^{i\phi}}{2}\left(\partial_{r}+ir^{-1}\partial_{\phi}\right)$,
define the massive holomorphic (or rather, meromorphic) functions

\begin{alignat*}{1}
\zeta_{-1}^{1}\left(z'=re^{i\phi}\right) & :=2m\left[e^{-i\phi}K_{1}\left(2mr\right)+iK_{0}\left(2mr\right)\right];\\
\zeta_{-1}^{i}\left(z'=re^{i\phi}\right) & :=2m\left[ie^{-i\phi}K_{1}\left(2mr\right)+K_{0}\left(2mr\right)\right],
\end{alignat*}
having the asymptotics $\zeta_{-1}^{1}(z')\sim\frac{1}{z'}$ and $\zeta_{-1}^{i}(z')\sim\frac{i}{z'}$
as $w\to0$. Their massive holomorphicity follows straightforwardly
from the recurrence relation $K_{l}'(r)=-K_{l\pm1}(r)\pm lr^{-1}K_{l}(r)$.
Recursively, the following estimate holds as well:
\begin{lem}
\label{lem:cdiff}The functions $\zeta_{-1}^{1},\zeta_{-1}^{i}$ are
smooth away from $w=0$, and has asymptotics
\[
\left\vert \partial^{(l)}\text{\ensuremath{\zeta_{-1}^{1}}}(z')\right\vert ,\left\vert \partial^{(l)}\text{\ensuremath{\zeta_{-1}^{i}}}(z')\right\vert \leq\const_{l}\left\vert z'\right\vert ^{-l-1},
\]
where $\partial^{(l)}$ refers to any $l$-th derivative.
\end{lem}

Recall that smoothness is a consequence (Corollary \ref{cor:smoothness})
of the following proposition, so we only assume below that $f$ is
locally Lipschitz, and use Green's theorem as described in Remark
\ref{rem:cgreens}.
\begin{prop}[{Continuum Cauchy Integral Formula \cite[Section 6]{bers}}]
\label{prop:ccauchy}For any massive holomorphic function $f$ defined
in a neighbourhood of $w$, we have
\begin{equation}
f(z)=\frac{1}{2\pi}\left[\imm\oint_{\omega}f(z')\zeta_{-1}^{1}(z-z')dz'-i\imm\oint_{\omega}f(z')\zeta_{-1}^{i}(z-z')dz'\right],\label{eq:ccauchy}
\end{equation}
for any simple integration contour $\omega$ surrounding $w$ in the
domain of definition.
\end{prop}

\begin{proof}
Recall that the imaginary part of the integral of the square of a
massive holomorphic function is well-defined. Given bilinearity, by
polarisation the imaginary part of the integral of a product of two
massive holomorphic functions is well-defined as well. Therefore,
the above integrals may be computed around any contour around $z'$,
which we take as a small circle $\partial B_{r}(z')$ around $z'$.
Given uniform continuity of $f$ and the asymptotics for $\zeta_{-1}^{1,i}$,
we clearly have $\imm\oint_{\partial B_{r}(z)}f(z')\zeta_{-1}^{1}(z-z')dz'\to2\pi \re f(z)$
and $\imm\oint_{\partial B_{r}(z)}f(z')\zeta_{-1}^{i}(z-z')dz'\to-2\pi \imm f(z)$
as $r\to0$.
\end{proof}

\subsubsection{Discrete Case}
The discrete contour integral
for a product of two (massive) s-holomorphic functions should be understood as polarisation of the discrete
square integral in Lemma \ref{lem:defh}: 
\[
\imm\int^{\delta}2FZdz:=\imm\int^{\delta}(F+Z)^{2}dz-\imm\int^{\delta}F^{2}dz-\imm\int^{\delta}Z^{2}dz,
\]
defined on (a connected subset of) $\Gamma \cup \Gamma^*$. We also define $\imm\oint^{\delta}$ along a closed loop of primal and/or dual vertices on $\Gamma \cup \Gamma^*$, which is clearly locally well-defined as long as both functions are s-holomorphic but might have a monodromy around a 'singularity', as in below.

We need discrete s-holomorphic functions corresponding to the continuous kernels $\zeta_{-1}^{1,i}$. For each $z\in\lozenge$ consider discrete
functions $Z_{-1}^{1,i}(z,\cdot)$ having following properties:
\begin{itemize}
\item $Z_{-1}^{1,i}(z,\cdot)$ is defined and massive s-holomorphic on $\lozenge\cup\Upsilon\setminus\left\{ z\right\} $;
\item Write $Z^{\tau}=Z_{-1}^{\tau}(z,\cdot)$. If we were to \emph{define} $Z^{\tau}(z)$ as the
value satisfying (\ref{eq:shol}) for a pair of adjacent corners (i.e. using the formula \eqref{eq:reconst}),
we may define two values $Z_{z},Z_{z}'$ coming respectively from
the pairs $\xi_{z},\xi^{z}$ and $\xi_{z}',\left(\xi^{z}\right)'$ (see Figure \ref{fig:app}L). Then we have
\begin{equation} \label{eq:disres}
\sin\hat{\theta}_{z}(Z_{z}-Z_{z}')\nu_{T}(z)=i\tau\cdot \pi\delta^{-1}.
\end{equation}
\end{itemize}

\begin{prop}
\label{prop:dcauchy_formula}For such discrete kernels $Z_{-1}^{\tau}$ and any massive s-holomorphic function
$F$ defined in a neighbourhood of $z\in\lozenge$, we have

\[
F(z)=\frac{1}{2\pi} \left[ \imm\oint_{\omega^{\delta}}^{\delta}F(z')Z_{-1}^{1}(z,z')dz'-i\imm\oint_{\omega^{\delta}}^{\delta}F(z')Z_{-1}^{i}(z,z')dz' \right],
\]
for any simple discrete contour $\omega^{\delta}$ surrounding $z'$.
\end{prop}

\begin{proof}
Since the contour for the integral $\imm\int^{\delta}$ may be deformed without crossing singularities, we may take $\omega^{\delta}$ as the smallest contour $\omega_{z}=\xi_{z}\sim\xi^{z}\sim\xi_{z}'\sim\left(\xi^{z}\right)'$ around $z$.

Then inspecting (\ref{eq:inth}) and the condition \eqref{eq:disres}, we see (by splitting $\omega_z$ into two half-rhombi)
\[
\imm\oint_{\omega^{\delta}}^{\delta}F(z')Z_{-1}^{\tau}(z,z')dz=2\delta \times \imm\left[\sin\hat{\theta}_{z}(Z_{z}-Z_{z}')F(z)\nu_{T}(z)\right]=2\pi \imm\left[i\tau F(z)\right],
\]
which yields the result.
\end{proof}
In fact, construction for such discrete kernels is done in \cite{cim21} using \emph{massive discrete exponentials} (see also \cite{bdtr}). We explain how to apply their construction in our setup.
\begin{prop}[{\cite[Section 5]{cim21}}]
\label{prop:dcauchy_est}The kernels $Z_{-1}^{1,i}(z,z')$ may be
constructed using explicit contour integrals (involving the angles
of a path from $z$ to $z'$), and have the uniform asymptotic
\begin{equation}\label{eq:kernel_exact}
Z_{-1}^{1,i}(z,z')=\zeta_{-1}^{1,i}(z-z')+O\left(\delta\right),
\end{equation}and thus
\[
Z_{-1}^{1,i}(z,z')-Z_{-1}^{1,i}(z^\delta,z')=O(\delta),
\]
for $z\sim z^\delta \in\lozenge$, where $m\asymp q\delta^{-1}, \vert z-z'\vert $ are uniformly bounded away from $0$ and $\infty$.\\

In fact, for $z\neq z'$,
\begin{equation}\label{eq:kernel_crude}
Z_{-1}^{1,i}(z,z') \apprle \frac{1}{\vert z-z'\vert }.
\end{equation}

These estimates hold for small $q$, and only depend on the respective uniform conditions and the uniform angle bound $\eta$.
\end{prop}

\begin{proof}
Let us summarise the consequences of \cite[Proposition 5.8]{cim21}: for corners $a\neq c$, \cite[(5.15)]{cim21} constructs a real function $\mathrm{G}_{(a)}(c)$ which branches (switches sign) around every vertex and face and satisfies the three-point propagation equation in $c$. Then as in the proof of Proposition \ref{prop:s-hol}, we may define non-branching function $\Xi(a,c):=\left(i\nu(c)\right)^{-1/2}\mathrm{G}_{(a)}(c)$ which may be extended to a massive s-holomorphic function $\Xi(a,\cdot)$. At $a$, the value of $\Xi(a,\cdot)$ respecting the s-holomorphicity condition \eqref{eq:shol} is $\left(i\nu(a)\right)^{-1/2}$ with the edge to the left of $a$ facing the direction of $\nu(a)$, while it is $-\left(i\nu(a)\right)^{-1/2}$ to the right.

Now, given $z$ as in Figure \ref{fig:app}L, we will set $\Xi(\xi_{z},\xi_{z}):=\left(i\nu(\xi_{z})\right)^{-1/2}$ and $\Xi(\xi_{z}',\xi_{z}'):=\left(i\nu(\xi_{z}')\right)^{-1/2}:=-i \left(i\nu(\xi_{z})\right)^{-1/2}$, i.e. so that they are s-holomorphic \emph{away from} $z$. We would like to real-linear combine these two functions so that we have \eqref{eq:disres}. For $Z=\Xi(\xi_{z},\cdot)$, the difference $Z_z-Z_z'$ would have been zero had $\Xi(\xi_{z},\xi_{z})$ been defined as $-\left(i\nu(\xi_{z})\right)^{-1/2}$ (so that it would be s-holomorphic with respect to $z$): therefore, $\sin\hat{\theta}_{z}(Z_{z}-Z_{z}')\nu_{T}(z)=-2i e^{i\frac{\theta_{z}^{\dagger}}{2}}\left(i\nu(\xi_{z})\right)^{1/2}$, which is precisely the contribution in \eqref{eq:reconst} coming from switching $-\left(i\nu(\xi_{z})\right)^{-1/2}$ to $\left(i\nu(\xi_{z})\right)^{-1/2}$. For $Z=\Xi(\xi_{z}',\cdot)$, similar computation gives $\sin\hat{\theta}_{z}(Z_{z}-Z_{z}')\nu_{T}(z)=2i e^{i\frac{\theta_{z}^{\dagger}}{2}}\left(i\nu(\xi_{z}')\right)^{1/2}=-2i e^{i\frac{\theta_{z}^{\dagger}}{2}}\left(i\nu(\xi_{z})\right)^{1/2}\cdot i$.

Therefore, we may set (recalling the notation $\left(i\nu(\hat{\xi}_{z})\right)^{1/2}:={e^{i\frac{\theta_{z}^{\dagger}}{2}}\left(i\nu(\xi_{z})\right)^{1/2}}$ from Proposition \ref{prop:s-hol})
\begin{align*}
Z_{-1}^{1}(z,z')&:=\frac{\pi}{2\delta} \left(\imm{\left(i\nu(\hat{\xi}_{z})\right)^{1/2}}\Xi(\xi_{z},z') + \re{\left(i\nu(\hat{\xi}_{z})\right)^{1/2}}\Xi(\xi_{z}',z') \right)\\
Z_{-1}^{i}(z,z')&:=\frac{\pi}{2\delta} \left(\imm{\left(i\nu(\hat{\xi}_{z})\right)^{1/2}}\Xi(\xi_{z}',z') - \re{\left(i\nu(\hat{\xi}_{z})\right)^{1/2}}\Xi(\xi_{z},z') \right).
\end{align*}
Then \eqref{eq:kernel_exact} follows straightforwardly from \cite[(5.30), (5.25)]{cim21} (cf. also \cite[Theorem 3.16]{cim21}).

For the cruder uniform estimate \eqref{eq:kernel_crude}, we essentially give a cruder version of the proof of \cite[(5.30)]{cim21}: we will use their conventions, note especially the conversion $\check{x} = \frac{2K}{\pi} x$ \cite[(5.1)]{cim21} between 'elliptic' and physical angles. For the main definition \cite[(5.15)]{cim21} of the discrete exponential, the path $a \sim w_0 \sim \cdots \sim w_n \sim c$, $w_j \in \Lambda$ is chosen, such that the arguments of successive segments $\phi_{w_0 a} := \arg(w_0 - a), \ldots$ are all contained in a segment of length $\pi - 2\eta$, whose midpoint is denoted $\phi^\Lambda_{ca}$. After the variable shift $\check{\nu} = \check{\mu}-2iK'$, the integral in $\check{\nu}$ is computed on the broken line passing through the points $\check{\phi}^\Lambda_{ca} - 2iK', \check{\phi}^\Lambda_{ca} - iK', \check{\phi}_{ca} - iK', \check{\phi}_{ca} + iK', \check{\phi}^\Lambda_{ca} + iK', \check{\phi}^\Lambda_{ca} + 2iK'$. The real parameter $y\in [-1,1]$ is introduced so that $\imm \check{\nu} = 2K' y$ and $\imm \nu = -{y}\log q$.
\begin{itemize}
\item When $\left\vert  \imm \check{\nu} \right\vert  \leq iK'$, the estimate \cite[Remark 5.10]{cim21} corresponding to the first identity of \cite[(5.24)]{cim21} (which is uniform in $\vert c-a\vert $) implies that the integrand is bounded by a constant multiple of
\begin{align*}
& k^2(k')^{-\frac{1}{2}} \left\vert  \cos\left(\frac{1}{2}(\nu - \phi_{w_0 a}) \right) \cos\left(\frac{1}{2}(\nu - \phi_{c w_n}) \right) \exp\left[ -2m\vert c-a\vert  \cos\left(\nu - \phi_{ca} \right)\right] \right\vert  \\
& \apprle q q^{-\vert y\vert }\exp\left[ -2m(\sin\eta) \vert c-a\vert   \cosh(\imm \nu) \right]\leq q^{1-\vert y\vert } \exp\left[ -m(\sin\eta) \vert c-a\vert   q^{-\vert y\vert } \right],
\end{align*}
since $\re \check{\nu}$ is between $\check{\phi}_{ca}$ and $\check{\phi}^\Lambda_{ca}$, such that $\left\vert \re \nu - \phi_{ca} \right\vert  \leq \frac{\pi}{2} - \eta$.
So, for the three integration line segments in this region:
\begin{itemize}
\item For the middle vertical segment $\left[\check{\phi}_{ca} - iK', \check{\phi}_{ca} + iK \right]$, we multiply by $(-\log q) i\cdot dy$ and integrate over $y=[-1, 1]$, which can be done explicitly and yield an $O(\frac{q}{m(\sin \eta)\vert c-a\vert }) \apprle \frac{\delta}{\vert c-a\vert }$ estimate, as desired.
\item For the horizontal segments between $\check{\phi}^\Lambda_{ca} \pm iK'$ and $\check{\phi}_{ca} \pm iK $ (which is of $O(1)$ length), simply bound by the the maximum over $q^{-\vert y\vert }$, which is $\frac{e^{-1}q}{m(\sin\eta)\vert c-a\vert }$.
\end{itemize}

\item For the vertical line segments $\left[ \check{\phi}^\Lambda_{ca} \pm iK', \check{\phi}^\Lambda_{ca} \pm 2iK' \right]$, we give a similar integral bound, relying on the fact that if $\re \check{\nu} = \check{\phi}^\Lambda_{ca}$, any $z:=\frac{1}{2}(\nu - \phi)$ coming from the integrand has $\left\vert  \re 2z \right\vert  \leq \frac{\pi}{2} - \eta$, so $2q \cos(2z)\in \mathbb D^\eta := \{\zeta: \vert \zeta\vert <1, \arg \zeta \in \left[-\frac{\pi}{2} + \eta, \frac{\pi}{2} - \eta \right]\}$. The integrand to be estimated is the product of (cf. especially \cite[Remark 5.10, (5.19)]{cim21}): $-ik \left(k'\right)^{-\frac{1}{4}} \jacobicd\left(\frac{1}{2}(\check{\nu}-\check{\phi}_{w_0 a}) \vert  k \right)$, $ -\sqrt{k'} \jacobind\left(\frac{1}{2}(\check{\nu}-\check{\phi}_{w_{j+1} w_j}) \vert  k \right)$ for $j=0,\ldots,n-1$, and $ik \left(k'\right)^{\frac{1}{4}} \jacobicd\left(\frac{1}{2}(\check{\nu}-\check{\phi}_{cw_n}) \vert  k \right)$, each corresponding to the segments of the path chosen for the discrete exponential.

From \cite[22.2, 20.2]{dlmf} (cf. \cite[(5.22)]{cim21} and the discussion):
\begin{align*}
 \left\vert  \jacobicd\left(\check{z} \vert  k \right) \right\vert  &=  \left\vert  \frac{({1+O(q)})(\cos(z)+O(q^{1/2}))}{1+2q\cos(2z)+O(q^2)} \right\vert ;\\ \left\vert  \sqrt{k'}\jacobind\left(\check{z} \vert  k \right) \right\vert  &= \left\vert  \frac{1-2q\cos(2z)+O(q^2)}{1+2q\cos(2z)+O(q^2)} \right\vert .
\end{align*}
In $\mathbb D^\eta$, there is a constant $c(\eta)>0$ such that $\left\vert \frac{1-\zeta}{1+\zeta}\right\vert \leq 1-c(\eta)\re \zeta$. Therefore we have for small $q$:
\begin{align*}
 \left\vert  \jacobicd\left(\check{z} \vert  k \right) \right\vert  \apprle  \left\vert  \cos(z) \right\vert \leq q^{-\vert y\vert /2}; \log \left\vert  \sqrt{k'}\jacobind\left(\check{z} \vert  k \right) \right\vert  \apprle -  \re \left[ 2q\cos(2z) \right]  + q^2 \apprle -q^{1-\vert y\vert }.
\end{align*}
We use the first estimate twice (start and end) and the second $n \asymp \delta^{-1}\vert a-c\vert $ times, then integrate with $(-\log q)i\cdot dy$ as above, yielding an $O(\frac{q^{3/2}}{m(\sin \eta)\vert c-a\vert })$ estimate for $\mathrm{G}_{(a)}(c)$.
\end{itemize}
\end{proof}
In fact, we could have constructed the same kernels $Z_{-1}^{1,i}(z,z')$ starting from any two corners around $z$; although these kernels are not antisymmetric in its variables \emph{per se}, the two variables may be considered on an equal footing, so there is a relation for the variable $z$ which is equivalent to the s-holomorphicity for $z'$: see e.g. \cite{ghp}.
\subsection{$W^{1,p}$-factorisation}

In this section, we collect deeper Sobolev space theory and carry
out detailed analysis near the boundary $\partial D$. Recall the
linear (i.e. length) and planar (i.e. area) Hausdorff measures $\mathcal{H}^{1,2}$
on $\mathbb{C}$. We first recall the notion of Lebesgue points of
Sobolev functions, and how it may be used to define the trace $\mathcal{H}^{1}$-almost
everywhere.

We first recall some general facts, of crucial use in \cite{baratchart}.
\begin{lem}
\label{lem:general_sobolev}Let $p\in(1,2)$ and $U$ be a bounded
domain whose boundary is locally a graph of a Lipschitz function.
\begin{enumerate}
\item Any function $g\in W^{1,p}(U)$ may be \emph{strictly defined} (i.e.
by the limit of integral averages) pointwise at its \emph{Lebesgue
points}, whose complement is a $\mathcal{H}^{1}$-null set.
\item There is a continuous extension operator $W^{1,p}(U)\to W^{1,p}(\mathbb{C})$.
The restriction to $\mathcal{H}^{1}$-almost all vertical and horizontal
lines of a function in $W^{1,p}(\mathbb{C})$ is absolutely continuous.
\item The trace $\tr_{\partial U}g$ may be obtained as the restriction
on $\partial U$ of the strictly defined version of any extension
of $g$ to $\mathbb{C}$. More generally, the restriction of $g$
is well-defined $\mathcal{H}^{1}$-almost everywhere.
\item The exponential function maps $W^{1,2}(U)$ into $W^{1,p}(U)$ continuously
for any $p\in(1,2)$, and therefore into $L^{q}\left(U\right)$ for
any $q\in[1,\infty)$.
\item The trace operator commutes with exponential, and $\tr_{\partial U}e^{g}\in L^{q}\left(\partial U\right)$
for any $q\in[1,\infty)$.
\end{enumerate}
\end{lem}

\begin{proof}
The introduction of \cite{baratchart} recalls the standard facts
(1), (3) and the first half of (2). The second half of (2) is the
so-called ACL characterisation of Sobolev functions, e.g. \cite[Theorem 4.21]{evans-gariepy}.
(4) and (5) are proved in \cite[Proposition 8.4]{baratchart}.
\end{proof}
We are ready to show that the constant boundary condition, which constitutes half 
of the condition $(rh)_{h}$, is preserved under multiplication by
$e^{s}$. Note in the following we \emph{assume} local boundedness
near $S$ in one direction, while its counterpart follows from the
boundary condition in the other. We state it in slightly general terms,
to adapt to situations where multiple nested domains interact.
\begin{prop}
\label{prop:rhh1}Suppose a massive holomorphic pullback $f^{\text{pb}}$
on a smooth domain $D$ and a holomorphic function $g$ is related
by $f^{\text{pb}}=e^{s_{0}}g$ for $s_{0}\in W^{1,2}\left(D\right)$
with $\tr_{S}s_{0}\in\mathbb{R}$ on some segment $S\subset\partial D$. 

If $e^{-s_{1}}f^{\text{pb}}$ (for example, $f^{\text{pb}}$ or $g$) for some $s_{1}\in W^{1,2}\left(D\right)$ is bounded near $S$ and $h^{\text{pb}}:=\imm\int\left(f^{\text{pb}}\right)^{2}dz$
continuously extends as a constant to $S$, then $\imm\int g^{2}dz$
continuously extends as a constant to $S$. On the other hand, if
$\imm\int g^{2}dz$ continuously extends as a constant to $S$, then
$g$ is bounded on $S$ and $h^{\text{pb}}$ extends as a constant
there.
\end{prop}

\begin{proof}
Suppose $e^{-s_{1}}f^{\text{pb}}$ is locally bounded near some segment
$S\subset\partial D$ and $h^{\text{pb}}:=\imm\int\left(f^{\text{pb}}\right)^{2}dz$
continuously extends as a constant to $S$. Take a small domain $D'\subset D$
where $f^{\text{pb}}$ is bounded, sharing some segment $S'\subset S\cap\partial D'$
with $D$. By another smooth pullback, we may assume that $D'$ is
the unit disc $\mathbb{D}$ and $S'$ is the upper half-circle $\partial\mathbb{D}\cap\mathbb{H}$.
The restrictions of $s_{0,1}$ on $D'=\mathbb{D}$ are all in $W^{1,2}\left(\mathbb{D}\right)$.

By Lemma \ref{lem:general_sobolev}, $e^{s_{1}(r\times\cdot)}\vert_{\partial\mathbb{D}}$,
and therefore $f^{\text{pb}}(r\times\cdot)\vert_{\partial\mathbb{D}}$,
is bounded as $r\uparrow1$ in any $L^{q}(\partial\mathbb{D})$ for
$q\in(1,\infty)$. By the generalised analytic Fatou's theorem \cite[Theorem 5.1]{baratchart},
there is a radial limit which we also denote by $\tr_{\partial\mathbb{D}}f^{\text{pb}}$
on $\partial\mathbb{D}$, to which $f^{\text{pb}}(r\times\cdot)\vert_{\partial\mathbb{D}}$
converges as $r\uparrow1$ in any $L^{q}(\partial\mathbb{D})$. Therefore,
on any sub-segment $\left[e^{i\phi_{1}},e^{i\phi_{2}}\right]\subset S'$,
we have
\[
\imm\int_{\left[e^{i\phi_{1}},e^{i\phi_{2}}\right]}\left(f^{\text{pb}}\right)^{2}(r\times\cdot)dz\xrightarrow{r\uparrow1}\begin{cases}
h^{\text{pb}}\left(e^{i\phi_{2}}\right)-h^{\text{pb}}\left(e^{i\phi_{1}}\right)=0,\\
\imm\int_{\left[e^{i\phi_{1}},e^{i\phi_{2}}\right]}\left(\tr_{\partial\mathbb{D}}f^{\text{pb}}\right)^{2}dz.
\end{cases}
\]
The two limits coincide. Since this is true for arbitrary $\phi_{1,2}\in\left(0,\pi\right)$,
along $S'$ we have $\left(\tr_{\partial\mathbb{D}}f^{\text{pb}}\right)^{2}\in\nu_{\text{tan}}^{-1}\mathbb{R}$.

Now we show that $\imm\int_{e^{i\phi_{1}}}^{e^{i\phi_{2}}}g^{2}dz$
vanishes. For the radial segments $\left[e^{i\phi_{1}},re^{i\phi_{1}}\right]$,
$\left[re^{i\phi_{2}},e^{i\phi_{2}}\right]$ and the circular arc
$r\left[e^{i\phi_{1}},e^{i\phi_{2}}\right]$, it suffices to show
that the integral over each contour tends to zero as $r\uparrow1$.
We have
\begin{alignat}{1}
\left\vert \imm\int_{\left[e^{i\phi_{1}},re^{i\phi_{1}}\right]}g^{2}dz\right\vert  & \leq\int_{\left[e^{i\phi_{1}},re^{i\phi_{1}}\right]}\left\vert e^{-s_{1}}f^{\text{pb}}\right\vert ^{2}\left\vert e^{-2(s_{0}-s_{1})}\right\vert \left\vert dz\right\vert \nonumber \\
 & \leq O(1)\left\Vert e^{s_{1}-s_{0}}\right\Vert _{L^{2}\left[e^{i\phi_{1}},re^{i\phi_{1}}\right]}^{2}\nonumber \\
 & \leq O\left((1-r)^{\alpha}\right)\text{ for each }\alpha\in(0,1),\label{eq:sobolevholder}
\end{alignat}
 by H\"older since $e^{s_{1}-s_{0}}$ belongs to $L^{q}\left[e^{i\phi_{1}},re^{i\phi_{1}}\right]$
for each $q\in(1,\infty)$ by Lemma \ref{lem:general_sobolev}. The
trace may be taken on, e.g., the half-disc bordering $\left[e^{i\phi_{1}},re^{i\phi_{1}}\right]$
(see Figure \ref{fig:app}R). Similar estimates hold for the other
radius $\left[re^{i\phi_{2}},e^{i\phi_{2}}\right]$. On the circular
arc,

\begin{alignat*}{1}
\imm\int_{r\left[e^{i\phi_{1}},e^{i\phi_{2}}\right]}g^{2}dz & =\imm\int_{\left[e^{i\phi_{1}},e^{i\phi_{2}}\right]}e^{2s_{0}}\left(f^{\text{pb}}\right)^{2}(r\times\cdot)rdz\\
 & \xrightarrow{r\uparrow1}\imm\int_{\left[e^{i\phi_{1}},e^{i\phi_{2}}\right]}\tr_{\left[e^{i\phi_{1}},e^{i\phi_{2}}\right]}e^{2s_{0}}\left(\tr_{\partial\mathbb{D}}f^{\text{pb}}\right)^{2}dz,
\end{alignat*}
where we use the $L^{q}$ convergence of $f^{\text{pb}}(r\times\cdot)\vert_{\partial\mathbb{D}}$
to its trace as above and of $e^{2s_{0}}\vert_{r\left[e^{i\phi_{1}},e^{i\phi_{2}}\right]}$
to its trace (easily seen by applying trace and Sobolev inequalities
on the crescent-shaped domain as in Figure \ref{fig:app}R). Since
$\tr_{\left[e^{i\phi_{1}},e^{i\phi_{2}}\right]}e^{2s_{0}}\in\mathbb{R}$
and $\left(\tr_{\partial\mathbb{D}}f^{\text{pb}}\right)^{2}\in\nu_{\text{tan}}^{-1}\mathbb{R}$,
the imaginary part of their integral is zero. So we have $\imm\int_{e^{i\phi_{1}}}^{e^{i\phi_{2}}}g^{2}dz=0$
and $\imm\int g^{2}dz$ extends ($\alpha$-H\"older) continuously as
a constant on $S$.

In the other direction, the boundedness of $\underline{f}^{\text{pb}}$
on $S'\subset S$ comes from boundary regularity of harmonic function
$\imm\int g^{2}dz$: the gradient $g^{2}$ extends continuously to
$S'$. Then the remaining calculations are again analogous to above.
\end{proof}
The other component of $(rh)_{h}$ is the existence of the sequence
of values of a specific sign (if the boundary value is set to zero).
We note the following consequence (slightly stronger
than $(rh)_{h}$) of the condition $(rh)_{f}$.
\begin{prop}
\label{prop:rhh2}Suppose a massive holomorphic pullback $f^{\text{pb}}$
on a smooth domain $D$ and a holomorphic function $g$ is related
by $f^{\text{pb}}=e^{s_{0}}g$ for $s_{0}\in W^{1,2}\left(D\right)$
with $\tr_{\partial D}s_{0}\in\mathbb{R}$ on some segment $S\subset\partial D$.
Suppose $f^{\text{pb}}$ is not identically zero, $g$ extends smoothly
to $S$ and is in $\sqrt{\frac{\upsilon}{\nu_{\text{tan}}}}\mathbb{R}$
along $S$. Then $\mathcal{H}^{1}$-almost everywhere on $S$, the
inner normal derivative of $h^{\text{pb}}$ exists and is strictly
positive (if $\upsilon=1$) or negative (if $\upsilon=-1$).
\end{prop}

\begin{proof}
Since smooth pullbacks preserve normal derivatives, we may assume
that $D\subset\mathbb{H}$ and $S\subset\partial D\cap\mathbb{R}$.
Then the inner derivative of $h^{\text{pb}}$ at some $x\in S$ is
exactly $\partial_{y}\imm\int_{0}^{y}\left(f^{\text{pb}}\right)^{2}(x+iy')dy'$.
But $f^{\text{pb}}$ and thus $g$ is not identically zero, so $g$
(which continuously extends to $S$ by assumption) is nonzero $\mathcal{H}^{1}$-almost
everywhere on $S$. By Lemma \ref{lem:general_sobolev} on $\mathcal{H}^{1}$-almost
all vertical lines $e^{s_{0}}$ is (absolutely) continuous, and takes
a positive value at the intersection with $S\subset\mathbb{R}$. By
the fundamental theorem of calculus the desired normal derivative
exists at those intersections, whose sign is fixed by the fact that
$g\in\mathbb{R}$ if $\upsilon=1$ and $g\in i\mathbb{R}$ if $\upsilon=-1$
on $S$.
\end{proof}

\end{document}